\documentclass{amsart}

\usepackage{amsfonts,mathrsfs}
\usepackage{amsmath,amssymb,amsthm, amscd, bm}
\usepackage{tikz}
\usetikzlibrary{arrows,calc,matrix,topaths,positioning,scopes,shapes,decorations,decorations.markings} 
\usepackage{dsfont}
\usepackage{epsfig}
\usepackage{tikz-cd}
\usepackage{hyperref}
\usepackage{fullpage}
\usepackage[foot]{amsaddr}
\usepackage{adjustbox}
\usepackage{enumitem}

\usepackage[utf8x]{inputenc}

\makeatletter

\renewcommand{\email}[2][]{%
  \ifx\emails\@empty\relax\else{\g@addto@macro\emails{,\space}}\fi%
  \@ifnotempty{#1}{\g@addto@macro\emails{\textrm{(#1)}\space}}%
  \g@addto@macro\emails{#2}%
}
\makeatother

\author{Hiroaki  Karuo$^{(1)}$}
\address{${}^{(1)}$ Gakushuin University, Faculty of Sciences, Department of mathematics, 1-5-1 Mejiro, Toshima-ku, Tokyo 171-8588 Japan}
\email{hiroaki.karuo@gakushuin.ac.jp}
\author{Julien Korinman$^{(2)}$}
\address{${}^{(2)}$ Institut Montpelli\'erain Alexander Grothendieck - UMR 5149 Universit\'e de Montpellier. Place Eug\'ene Bataillon, 34090 Montpellier France}
\email{julien.korinman@gmail.com}
\urladdr{https://sites.google.com/site/homepagejulienkorinman/}

\subjclass{$57$R$56$,  $57$M$25$.}


\def\restriction#1#2{\mathchoice
              {\setbox1\hbox{${\displaystyle #1}_{\scriptstyle #2}$}
              \restrictionaux{#1}{#2}}
              {\setbox1\hbox{${\textstyle #1}_{\scriptstyle #2}$}
              \restrictionaux{#1}{#2}}
              {\setbox1\hbox{${\scriptstyle #1}_{\scriptscriptstyle #2}$}
              \restrictionaux{#1}{#2}}
              {\setbox1\hbox{${\scriptscriptstyle #1}_{\scriptscriptstyle #2}$}
              \restrictionaux{#1}{#2}}}
\def\restrictionaux#1#2{{#1\,\smash{\vrule height .8\ht1 depth .85\dp1}}_{\,#2}}

\newcommand{\quotient}[2]{{\raisebox{.2em}{$#1$}\left/\raisebox{-.2em}{$#2$}\right.}}

\newcommand{\Hom}{\operatorname{Hom}}
\newcommand{\tr}{\operatorname{tr}}
\newcommand{\Tr}{\operatorname{Tr}}

\newcommand{\SL}{\operatorname{SL}}
\newcommand{\id}{id}
\newcommand{\Span}{\operatorname{Span}}
\newcommand{\End}{\operatorname{End}}

\newcommand{\Specm}{\operatorname{MaxSpec}}
\newcommand{\Spec}{\operatorname{Spec}}
\newcommand{\Mod}{\operatorname{Mod}}

\newcommand{\Mat}{\operatorname{Mat}}

\newcommand{\col}{\operatorname{col}}

\newcommand{\Comod}{\operatorname{Comod}}
\newcommand{\Arc}{\operatorname{Arc}}
\newcommand{\Mut}{\operatorname{Mut}}
\newcommand{\Ob}{\operatorname{Ob}}
\newcommand{\res}{\operatorname{res}}
\newcommand{\Indecomp}{\operatorname{Indecomp}}
\newcommand{\uq}{\mathfrak{u}_q\mathfrak{sl}_2}
\newcommand{\Ext}{\operatorname{Ext}}
\newcommand{\Rep}{\operatorname{Rep}}
\newcommand{\Image}{\operatorname{Im}}
\newcommand{\QMA}{unrestricted quantum moduli algebra}
\newcommand{\QMAs}{unrestricted quantum moduli algebras}
\newcommand{\Dq}{D_qB}

\newcommand{\heightexch}[3]{
	\begin{tikzpicture}[baseline=-0.4ex,scale=0.5, >=stealth]
	\draw [fill=gray!60,gray!45] (-.7,-.75)  rectangle (.4,.75)   ;
	\draw[#1] (0.4,-0.75) to (.4,.75);
	\draw[line width=1.2] (0.4,-0.3) to (-.7,-.3);
	\draw[line width=1.2] (0.4,0.3) to (-.7,.3);
	\draw (0.65,0.3) node {\scriptsize{$#2$}}; 
	\draw (0.65,-0.3) node {\scriptsize{$#3$}}; 
	\end{tikzpicture}
}
\newcommand{\heightcurve}{
\begin{tikzpicture}[baseline=-0.4ex,scale=0.5]
\draw [fill=gray!20,gray!45] (-.7,-.75)  rectangle (.4,.75)   ;
\draw[-] (0.4,-0.75) to (.4,.75);
\draw[line width=1.2] (-.7,-0.3) to (-.4,-.3);
\draw[line width=1.2] (-.7,0.3) to (-.4,.3);
\draw[line width=1.15] (-.4,0) ++(-90:.3) arc (-90:90:.3);
\end{tikzpicture}
}

\begin{document}

\theoremstyle{plain}
\newtheorem{theorem}{Theorem}[section]
\newtheorem{proposition}[theorem]{Proposition}
\newtheorem{corollary}[theorem]{Corollary}
\newtheorem{lemma}[theorem]{Lemma}
\theoremstyle{definition}
\newtheorem{notations}[theorem]{Notations}
\newtheorem{convention}[theorem]{Convention}
\newtheorem{problem}[theorem]{Problem}
\newtheorem{definition}[theorem]{Definition}
\theoremstyle{remark}
\newtheorem{remark}[theorem]{Remark}
\newtheorem{conjecture}[theorem]{Conjecture}
\newtheorem{example}[theorem]{Example}
\newtheorem{strategy}[theorem]{Strategy}
\newtheorem{question}[theorem]{Question}

\title[Classification of semi-weight representations of reduced stated skein algebras]{Classification of semi-weight representations of reduced stated skein algebras}
%
%
%

\date{}
\maketitle


\begin{abstract} 
We classify the finite dimensional semi-weight representations of the  reduced stated skein algebras at odd roots of unity of connected essential marked surfaces which either have a boundary component with at least two boundary edges or which do not have any unmarked boundary component.  We deduce computations of the PI-degrees and Azumaya loci of unreduced stated skein algebras of essential marked surfaces having at most one boundary puncture per boundary component and of the \QMAs  {} of lattice gauge field theory. 
\end{abstract}

\tableofcontents


\section{Introduction}

\subsection{Background on reduced stated skein algebras and their representations}\label{sec_background}

Let $\Sigma$ be an oriented compact surface and $\mathcal{A}$ a finite set of disjoint arcs embedded in $\partial \Sigma$ named \textit{boundary edges}. The pair $\mathbf{\Sigma}:=(\Sigma, \mathcal{A})$ will be called a \textit{marked surface}.
For $A\in \mathbb{C}^*$ a root of unity of odd order $N$, the \textit{reduced stated skein algebra} $\overline{\mathcal{S}}_A(\mathbf{\Sigma})$ was introduced in \cite{CostantinoLe19} as the quotient of the stated skein algebra by the kernel of Bonahon-Wong quantum trace. In particular, when $\mathcal{A}=\emptyset$, $\overline{\mathcal{S}}_A(\mathbf{\Sigma})$ is the usual Kauffman-bracket skein algebra.
Its representations appear in quantum hyperbolic geometry and are conjectured to form the building blocks of some $\SL_2$-HQFT (see \cite{BaseilhacBenedettiInvariant, BaseilhacBenedetti05, BaseilhacBenedetti15, BaseilhacBenedettiHTQFT, BaseilhacBenedettiLocRep}). Let $Z_{\mathbf{\Sigma}}$ denote the center of  $\overline{\mathcal{S}}_A(\mathbf{\Sigma})$. By \cite{KojuQuesneyClassicalShadows, LePaprocki2018} (after the original work in \cite{BonahonWong1}), there exists a \textit{Frobenius morphism} 
$$ Fr_{\mathbf{\Sigma}} : \overline{\mathcal{S}}_{+1}(\mathbf{\Sigma}) \hookrightarrow Z_{\mathbf{\Sigma}}$$
which embeds the commutative algebra $\overline{\mathcal{S}}_{+1}(\mathbf{\Sigma}) $ at $A=+1$ into $Z_{\mathbf{\Sigma}}$. Let $Z_{\mathbf{\Sigma}}^0 \subset Z_{\mathbf{\Sigma}}$ denote the image of the Frobenius morphism. 
\par  A representation $r: \overline{\mathcal{S}}_A(\mathbf{\Sigma})\to \End(V)$ is a \textit{weight representation} if $V$ is semi-simple as a module over $Z_{\mathbf{\Sigma}}$ and is called a \textit{semi-weight representation} if $V$ is semi-simple as a $Z_{\mathbf{\Sigma}}^0$ module. The purpose of this paper is to make progresses towards the following problem.

\begin{problem}\label{problem_classification}
Classify all finite dimensional weight and semi-weight representations of  $\overline{\mathcal{S}}_A(\mathbf{\Sigma})$. 
\end{problem}

We will solve Problem \ref{problem_classification} in the case where $\mathbf{\Sigma}=(\Sigma, \mathcal{A})$ is a connected marked surface which either has a boundary component with at least two boundary edges or which has at least one boundary edge and no unmarked boundary component.

 The center $Z_{\mathbf{\Sigma}}$ was computed in \cite{KojuAzumayaSkein} and is described as follows. 
 Partition the set of boundary components of $\Sigma$ into two subsets $\pi_0(\partial \Sigma)= \mathring{\mathcal{P}} \bigsqcup \Gamma^{\partial}$ where 
  $\mathring{\mathcal{P}}$ is the subset of boundary components  which do not intersect $\mathcal{A}$ and $\Gamma^{\partial}$ the boundary components which contain some boundary edges. For each $p\in \mathring{\mathcal{P}}$, there is a central element $\gamma_p \in Z_{\mathbf{\Sigma}}$ and for each $\partial \in \Gamma^{\partial}$ there is an invertible central element $\alpha_{\partial} \in Z_{\mathbf{\Sigma}}$ such that $Z_{\mathbf{\Sigma}}$ is generated by the image of the Frobenius morphism together with the elements $\gamma_p$ and $\alpha_{\partial}^{\pm 1}$. Let $\widehat{X}(\mathbf{\Sigma}):=\Specm(Z_{\mathbf{\Sigma}})$,  $X(\mathbf{\Sigma}):=\Specm(Z_{\mathbf{\Sigma}}^0)$,  and $\pi: \widehat{X}(\mathbf{\Sigma})\to X(\mathbf{\Sigma})$ the map defined by the inclusion $Z_{\mathbf{\Sigma}}^0 \subset Z_{\mathbf{\Sigma}}$. We also denote by $\widehat{\mathcal{X}}(\mathbf{\Sigma})\subset\Spec(Z_{\mathbf{\Sigma}})$ the set of prime ideals $I\subset Z_{\mathbf{\Sigma}}$ such that $I\cap Z^0_{\mathbf{\Sigma}}$ is maximal in $Z^0_{\mathbf{\Sigma}}$ so we still have a projection $\widehat{\pi}: \widehat{\mathcal{X}}(\mathbf{\Sigma}) \to X(\mathbf{\Sigma})$ sending $I$ to $I\cap Z^0_{\mathbf{\Sigma}}$ extending $\pi$. 
  \par 
 For a point $\widehat{x} \in \widehat{X}(\mathbf{\Sigma})$ (i.e. a maximal ideal in $Z_{\mathbf{\Sigma}}$), we denote by $\chi_{\widehat{x}} : Z_{\mathbf{\Sigma}} \to \mathbb{C}$ the corresponding character with kernel $\widehat{x}$ and use similar notations for $X(\mathbf{\Sigma})$. The variety  $\widehat{X}(\mathbf{\Sigma})$ is then described as 
\begin{multline*}  \widehat{X}(\mathbf{\Sigma}) := \{ \widehat{x}= (x, h_p, h_{\partial})_{p,\partial}, x\in X(\mathbf{\Sigma}),
 h_p\in \mathbb{C} \mbox{ is s.t. } T_N(h_p)=\chi_x(\gamma_p) \mbox{ for }p\in \mathring{P}, 
\\ h_{\partial} \in \mathbb{C}^* \mbox{is s.t. } h_{\partial}^N = \chi_x(\alpha_{\partial}) \mbox{ for }\partial \in \Gamma^{\partial} \}.\end{multline*}
 Here $T_N(X)$ denotes the $N$-th Chebyshev polynomial of the first kind. The morphism $\pi$ is just the projection $\pi(x, h_p, h_{\partial})= x$. It was also proved in \cite{KojuAzumayaSkein} that $\overline{\mathcal{S}}_A(\mathbf{\Sigma})$ is finitely generated over $Z_{\mathbf{\Sigma}}$. Moreover, if $Q(Z_{\mathbf{\Sigma}})$ denotes the fraction field of $Z_{\mathbf{\Sigma}}$, the dimension $D=D_{\mathbf{\Sigma}}:= \dim_{Q(Z_{\mathbf{\Sigma}})} \overline{\mathcal{S}}_A(\mathbf{\Sigma})\otimes_{Z_{\mathbf{\Sigma}}} Q(Z_{\mathbf{\Sigma}})$, named $\textit{PI-degree}$,  was computed explicitly in  \cite{KojuAzumayaSkein}.
 \par For $\widehat{x} \in \widehat{\mathcal{X}}(\mathbf{\Sigma})$ write 
 $$ \overline{\mathcal{S}}_A(\mathbf{\Sigma})_{\widehat{x}}:= \quotient{ \overline{\mathcal{S}}_A(\mathbf{\Sigma})}{{\widehat{x}\overline{\mathcal{S}}_A(\mathbf{\Sigma})}}.$$
 Let $r: \overline{\mathcal{S}}_A(\mathbf{\Sigma})\to \End(V)$ be an indecomposable semi-weight representation. 
 It \textit{classical shadow} is the maximal ideal $x_r:= \ker(r)\cap Z_{\mathbf{\Sigma}}^0 \in X(\mathbf{\Sigma})$. Its \textit{full shadow} is the prime ideal $\mathcal{I}_r:=\ker(r)\cap Z_{\mathbf{\Sigma}}\in \widehat{\mathcal{X}}(\mathbf{\Sigma})$. It \textit{maximal shadow} is the unique maximal ideal $\widehat{x}_r \in \widehat{X}(\mathbf{\Sigma})$ containing $\mathcal{I}_r$ (see Theorem \ref{theorem_FAL} for a proof that $\mathcal{I}_r$ is prime and that $\widehat{x}_r$ is unique).
Note that 
 $r$ is a weight representation if and only if $\mathcal{I}_r=\widehat{x_r}\in \widehat{X}(\mathbf{\Sigma})$ is maximal. 
\par 
The \textit{Azumaya locus} of $\overline{\mathcal{S}}_A(\mathbf{\Sigma})$ is 
$$ \mathcal{AL}(\mathbf{\Sigma}):= \{ \widehat{x} \in \widehat{X}(\mathbf{\Sigma})\mbox{ such that } \overline{\mathcal{S}}_A(\mathbf{\Sigma})_{\widehat{x}}\cong \Mat_D(\mathbb{C}) \}$$
and the \textit{fully Azumaya locus} is 
$$ \mathcal{FAL}(\mathbf{\Sigma}):= \{ x\in X(\mathbf{\Sigma}) \mbox{ such that } \pi^{-1}(x)\subset \mathcal{AL}(\mathbf{\Sigma})\}.$$
Since the matrix algebra $\Mat_D(\mathbb{C}) $ has a unique indecomposable representation (up to isomorphism) which is irreducible, for each $\widehat{x} \in \widehat{X}(\mathbf{\Sigma})$ there exists a unique indecomposable weight representation with classical shadow $\widehat{x}$ (up to isomorphism) named the \textit{Azumaya representation at } $\widehat{x}$.
\par A marked surface $\mathbf{\Sigma}=(\Sigma, \mathcal{A})$ is said \textit{essential} if each of its connected component contains at least one boundary edge in $\mathcal{A}$. 
In this case, it was proved  in \cite{KojuQuesneyClassicalShadows} that the variety $X(\mathbf{\Sigma})$ admits a geometric interpretation as follows. For each boundary edge $a\in \mathcal{A}$, fix a point $v_a \in a$ and write $\mathbb{V}:=\{v_a, a\in \mathcal{A}\}$. Let $\Pi_1(\Sigma, \mathbb{V})$ denote the full subcategory of the fundamental groupoid $\Pi_1(\Sigma)$ generated by $\mathbb{V}$; so the set of objects of $\Pi_1(\Sigma, \mathbb{V})$ is $\mathbb{V}$ and a morphism $\alpha : v_a \to v_b$ is an homotopy class of a path starting at $v_a$ and ending at $v_b$. Let $\SL_2$ be the category with a single element whose endomorphism set is $\SL_2(\mathbb{C})$ and let $\mathcal{R}_{\SL_2}(\mathbf{\Sigma})$ denote the set of functors $\rho : \Pi_1(\Sigma, \mathbb{V})\to \SL_2$. In  \cite{KojuQuesneyClassicalShadows, CostantinoLe19} it is proved that $X(\mathbf{\Sigma})$ embeds into $\mathcal{R}_{\SL_2}(\mathbf{\Sigma})$ (see Section \ref{sec_moduli_spaces} for details).
For $x\in X(\mathbf{\Sigma})$, let $\rho_x$ be the associated functor. 
 For $p\in \mathring{\mathcal{P}}$, let $\gamma_p \in \Pi_1(\Sigma, \mathbb{V})$ be the homotopy class of a peripheral curve encircling $p$ once. Then $x$ is called \textit{central at } $p$ if $\rho_x(\gamma_p)=\pm \mathds{1}_2$. This property is clearly independent on the choices of $a$ and $\gamma_p$. 
\vspace{2mm}
\par It was proved in \cite{FrohmanKaniaLe_UnicityRep} for unmarked surfaces and in \cite{KojuAzumayaSkein} for marked surfaces that $\mathcal{AL}(\mathbf{\Sigma})$ is open dense so a \textit{generic} indecomposable weight representation is irreducible and determined by its full shadow. However the computation of the Azumaya locus remains a challenging problem. For closed (and thus unmarked) surfaces, the Azumaya loci was computed in \cite{KojuKaruo_Azumaya} using the original results in \cite{GanevJordanSafranov_FrobeniusMorphism}. 
For open unmarked surfaces, many partial results were found in \cite{FKL_GeometricSkein, KojuKaruo_Azumaya, Yu_ALSkein_SmallSurfaces}.
For $\mathbf{\Sigma}=(\Sigma_{g,1}, \{a\})$ a genus $g$ surface with a single boundary component and a single boundary edge, it was proved in \cite{KojuMCGRepQT} that $ \overline{\mathcal{S}}_{A}(\mathbf{\Sigma})$ is Azumaya, i.e. that $\mathcal{AL}(\mathbf{\Sigma})=\widehat{X}(\mathbf{\Sigma})$. For $\mathbf{\Sigma}=\mathbb{T}:=(\mathbb{D}^2, \{a,b,c\})$ a disc with three boundary edges (usually called the triangle), it was showed in \cite{CostantinoLe19} that $ \overline{\mathcal{S}}_{A}(\mathbb{T})$ is isomorphic to a quantum torus, thus it is Azumaya as well.

\subsection{Main results}

 In order to state our main theorem, let us introduce some notations. Write $\Delta_+(\widehat{\mathfrak{sl}}_2):= \{ (n,m) \in \mathbb{N}^2 | (n-m)^2 \leq 1\}$ and decompose it as 
   $\Delta_+(\widehat{\mathfrak{sl}}_2)= \{(0,0), (0,1), (1,0)\} \bigsqcup \Delta^{\Re}_{++}(\widehat{\mathfrak{sl}}_2)\bigsqcup \Delta^{Im}_+(\widehat{\mathfrak{sl}}_2)$ where 
 $$  \Delta^{\Re}_{++}(\widehat{\mathfrak{sl}}_2)= \{ (k, k+1), k\geq1\} \cup \{(k+1, k), k\geq 1\}, \quad \Delta^{Im}_+(\widehat{\mathfrak{sl}}_2)=\{ (k,k), k\geq 1\}.$$
 Further write 
 $$ \Delta:= \{S, P\} \sqcup   \Delta^{\Re}_+(\widehat{\mathfrak{sl}}_2) \sqcup  \left(\Delta^{Im}_+(\widehat{\mathfrak{sl}}_2)\times \mathbb{CP}^1\right), $$ 
where $S,P$ are some formal parameters and write $\Delta \sqcup \overline{\Delta}:= \Delta \times \{0,1\}$ (two copies of $\Delta$ where we denote an element $(\alpha, 0)$ simply by $\alpha$ and an element $(\alpha, 1)$ by $\overline{\alpha}$). Let $\widehat{x}=(x,h_p,h_{\partial})_{p, \partial} \in \widehat{X}(\mathbf{\Sigma})$ and decompose the set of inner punctures as $\mathring{\mathcal{P}}= \mathring{\mathcal{P}}_0 \sqcup \mathring{\mathcal{P}}_1 \sqcup \mathring{\mathcal{P}}_2$ where 
\begin{align*}
&  \mathring{\mathcal{P}}_2:= \{p\in \mathring{\mathcal{P}}, \mbox{ such that } x \mbox{ is central at }p \mbox{ and }h_p \neq \pm 2\} \\
&  \mathring{\mathcal{P}}_1:= \{p\in \mathring{\mathcal{P}}\setminus \mathring{\mathcal{P}}_2, \mbox{ such that } \chi_x(\gamma_p) = \pm 2 \mbox{ and }h_p \neq \pm 2\}.
\end{align*}
A map $\sigma : \mathring{\mathcal{P}}\to \Delta \bigsqcup \overline{\Delta} $ is called a \textit{coloring compatible with} $\widehat{x}$ if $(1)$ for $p\in \mathring{\mathcal{P}}_0$, then $\sigma(p)=S$ and $(2)$ if $p\in \mathring{\mathcal{P}}_1$, then $\sigma(p)\in \{S,P\}$. By convention, if $\mathring{\mathcal{P}}=\emptyset$, we consider that every $\widehat{x}$ has a unique coloring. We denote by $\col(\widehat{x})$ the set of colorings compatible with $\widehat{x}$. 
We write
$ m=m(\widehat{x}):= | \mathring{\mathcal{P}}_2|$.
\par 
Let $\widehat{X}^{reg}(\mathbf{\Sigma})$ denote the regular part (smooth points) of $\widehat{X}(\mathbf{\Sigma})$. Let  $\overline{\mathcal{C}}$ be the category of semi-weight modules. Let $\Indecomp(\mathbf{\Sigma})$ be the set of isomorphism classes of indecomposable semi-weight $\overline{\mathcal{S}}_A(\mathbf{\Sigma})$ modules and let $\Indecomp(\mathbf{\Sigma}, \widehat{x})\subset \Indecomp(\mathbf{\Sigma})$ the subset of these modules with maximal shadow $\widehat{x}$. 
 The main result of this paper is the following

\begin{theorem}\label{main_theorem_intro} Let $\mathbf{\Sigma}=(\Sigma, \mathcal{A})$ be a connected essential marked surface  which either has a boundary component with at least two boundary edges or which does not have any inner puncture.
 Let  $\widehat{x}=(x, h_p, h_{\partial}) \in \widehat{X}(\mathbf{\Sigma})$.
\begin{enumerate}
\item We have $\mathcal{AL}(\mathbf{\Sigma})= \widehat{X}^{reg}(\mathbf{\Sigma})$, i.e. the Azumaya locus of $\overline{\mathcal{S}}_A(\mathbf{\Sigma})$ is equal to the regular locus of $\widehat{X}(\mathbf{\Sigma})$. Moreover 
 $\widehat{x}\in \mathcal{AL}(\mathbf{\Sigma})$ if and only if $m(\widehat{x})=0$.
\item The fully Azumaya locus $\mathcal{FAL}(\mathbf{\Sigma})$ of  $\overline{\mathcal{S}}_A(\mathbf{\Sigma})$ is the locus of representations which are not central at any inner puncture. For $x\in \mathcal{FAL}(\mathbf{\Sigma})$ the indecomposable semi-weight representations with classical shadow $x$ are classified (up to isomorphism) by their full shadows, so their set  is in $1:1$ correspondence with the fiber $\widehat{\pi}^{-1}(x) \subset \widehat{\mathcal{X}}(\mathbf{\Sigma})$.
\item The indecomposable semi-weight representations with maximal shadow  $\widehat{x}$ are in $1:1$ correspondence with the set of colorings compatible with $\widehat{x}$ so one has 
$$ \Indecomp(\mathbf{\Sigma})= \bigsqcup_{\widehat{x}\in \widehat{X}(\mathbf{\Sigma})} \Indecomp(\mathbf{\Sigma}, \widehat{x}) \stackrel{1:1}{\longleftrightarrow} \bigsqcup_{\widehat{x}\in \widehat{X}(\mathbf{\Sigma})} \col(\widehat{x}).$$
\item The indecomposable weight representations with full shadow  $\widehat{x}$ correspond to the colorings taking values in $\{S,\overline{S}, ((1,1), 1), \overline{((1,1), 1)}\}$: there are thus $4^m$ such representations.
\item The irreducible representations with full shadow  $\widehat{x}$ correspond to the colorings taking values in $\{S, \overline{S}\}$: there are thus $2^m$ such representations.
\item The indecomposable semi-weight representations with maximal shadow $\widehat{x}$ which are projective objects in $\overline{\mathcal{C}}$ correspond to the colorings sending the elements of $\mathring{\mathcal{P}}_1$ to $P$ and the elements of $\mathring{\mathcal{P}}_2$ to an element in $\{P, \overline{P}\}$: there are thus $2^m$ such representations.
\end{enumerate}
\end{theorem}

We will construct explicitly the  indecomposable semi weight representations of Theorem \ref{main_theorem_intro}, so it solves Problem \ref{problem_classification} for these marked surfaces. 
\vspace{2mm}
\par The proof of Theorem \ref{main_theorem_intro} uses a preliminary result which is interesting on its own.  For $\mathbf{\Sigma}_1, \mathbf{\Sigma}_2$ two  marked surfaces and $a_i$ a boundary edges of $\mathbf{\Sigma}_i$ ($i=1,2$), we denote by $\mathbf{\Sigma}_1\cup_{a_1\# a_2} \mathbf{\Sigma}_2$ the marked surface obtained by gluing $a_1$ and $a_2$ together. There exists an injective algebra morphism, named the \textit{splitting morphism}
$$ \theta_{a_1 \# a_2}: \overline{\mathcal{S}}_A(\mathbf{\Sigma}_1\cup_{a_1 \# a_2} \mathbf{\Sigma}_2) \hookrightarrow \overline{\mathcal{S}}_A(\mathbf{\Sigma}_1)\otimes \overline{\mathcal{S}}_A(\mathbf{\Sigma}_2).$$
If $V_1$ and $V_2$ are modules over  $\overline{\mathcal{S}}_A(\mathbf{\Sigma}_1)$ and $ \overline{\mathcal{S}}_A(\mathbf{\Sigma}_2)$, then by precomposing with $ \theta_{a_1 \# a_2}$, 
$V_1\otimes V_2$ becomes a module over $\overline{\mathcal{S}}_A(\mathbf{\Sigma}_1\cup_{a_1 \# a_2} \mathbf{\Sigma}_2)$.

 \begin{theorem}\label{theorem_gluing_intro} Let  $\mathbf{\Sigma}_1$ and $\mathbf{\Sigma}_2$ be marked surfaces and for $i=1,2$ let $a_i$ be a boundary edge of $\mathbf{\Sigma}_i$ such that the connected component of $\partial \Sigma_i$ which contains $a_i$ also contains at least another boundary edge. 
 Write $\mathbf{\Sigma}:=\mathbf{\Sigma}_1 \cup_{a_1\# a_2} \mathbf{\Sigma}_2$.
 Then  any indecomposable (semi) weight $\overline{\mathcal{S}}_A(\mathbf{\Sigma})$-module is isomorphic to a module $V_1\otimes V_2$ with $V_i$ a (semi) weight indecomposable $ \overline{\mathcal{S}}_A(\mathbf{\Sigma}_i)$ module. Conversely, any such $ \overline{\mathcal{S}}_A(\mathbf{\Sigma})$ module $V_1\otimes V_2$ is (semi) weight indecomposable.
\end{theorem}

The proof of Theorem \ref{main_theorem_intro} goes as follows. A 2P marked surface is a connected marked surface which has at least two boundary edges and no unmarked boundary component.
When $\mathbf{\Sigma}$ is 2P marked surfaces, it follows from the work of M\"uller \cite{Muller} and L\^e-Yu \cite{LeYu_SSkeinQTraces} that $\overline{\mathcal{S}}_A(\mathbf{\Sigma})$ is isomorphic to a quantum cluster algebra. The Azumaya loci of quantum cluster algebras have been studied in \cite{MNTY_AzumayaClusterAlgebras} and we will deduce from this study that $\overline{\mathcal{S}}_A(\mathbf{\Sigma})$ is Azumaya.
For $\mathbf{\Sigma}=\mathbb{D}_1:=(\Sigma_{0,2}, \{a,b\})$ an annulus with two boundary arcs in the same boundary component (also called \textit{punctured bigon}), it was proved in \cite{KojuQGroupsBraidings} that $ \overline{\mathcal{S}}_{A}(\mathbb{D}_1) \cong \Dq$ is isomorphic to the Drinfel'd double of the quantum Borel algebra whose representation theory is very similar to the one of $U_q\mathfrak{sl}_2$.
The simple representations of $U_q\mathfrak{sl}_2$ were classified in \cite{DeConciniKacRepQGroups, ArnaudonRoche}, the finite dimensional indecomposable representations of the small quantum group $\mathfrak{u}_q\mathfrak{sl}_2$ were classified in \cite{Suter_Uqsl2Modules},  and many more indecomposable representations were found in \cite{Arnaudon_Uqsl2Rep}. However, at the authors knowledge, the indecomposable $U_q\mathfrak{sl}_2$ modules have never been classified. We will provide a classification of all indecomposable semi-weight $\Dq$ modules and prove that 
Theorem \ref{main_theorem_intro}  holds for $\mathbb{D}_1$. If  $\mathbf{\Sigma}$ is a connected marked surface which has a boundary component with at least two boundary edges, it can be obtained by gluing several copies of $\mathbb{D}_1$ to a 2P marked surface so Theorem \ref{main_theorem_intro} will follow from Theorem \ref{theorem_gluing_intro}. When $\mathbf{\Sigma}$ is an essential marked surface without inner puncture, it is either a 2P marked surface or is has a single boundary arc in which case its reduced stated skein algebra is Azumaya by \cite{KojuMCGRepQT}. 
\vspace{2mm}
\par The techniques of the present paper can also be used to study the representations of the unreduced stated skein algebras. Let $\mathbf{\Sigma}$ be an essential marked surface and let $\mathbf{\Sigma}^*$ be the same marked surface where each boundary edge has been replaced by a pair of boundary edges inside the same boundary component. L\^e and Yu constructed in \cite{LeYu_SSkeinQTraces} an embedding $j: \mathcal{S}_A(\mathbf{\Sigma}) \hookrightarrow \overline{\mathcal{S}}_A(\mathbf{\Sigma}^*)$ of the unreduced stated skein algebra of $\mathbf{\Sigma}$ into the reduced stated skein algebra of $\mathbf{\Sigma}^*$. Therefore every module over $\overline{\mathcal{S}}_A(\mathbf{\Sigma}^*)$ is also a module over $\mathcal{S}_A(\mathbf{\Sigma})$ so Theorem \ref{main_theorem_intro} provides a large family  of representations for the unreduced stated skein algebras as well. We will deduce the following:

 \begin{theorem}\label{theorem_Skein_intro}
 Let $\mathbf{\Sigma}=(\Sigma_{g,n}, \mathcal{A})$ be a connected essential marked surface of genus $g$ with $n$ boundary components such each boundary component contains at most one boundary edge. So it has $k:= n-|\mathcal{A}|$ inner punctures $p_1, \ldots, p_k$ and $|\mathcal{A}|$ boundary punctures $p_{\partial_1}, \ldots, p_{\partial_{|\mathcal{A}|}}$.
 \begin{enumerate}
 \item The center $\underline{Z}_{\mathbf{\Sigma}}$ of $\mathcal{S}_A(\mathbf{\Sigma})$ is generated by the image $\underline{Z}^0_{\mathbf{\Sigma}}$ of the Frobenius morphism together with the peripheral curves $\gamma_{p_1}, \ldots, \gamma_{p_k}$. So the closed points of $\widehat{\underline{X}}(\mathbf{\Sigma}):= \Specm(\underline{Z}_{\mathbf{\Sigma}})$ are in $1:1$ correspondence with the set of elements $\widehat{\rho}=(\rho, h_{p_1}, \ldots, h_{p_k})$ with $\rho: \Pi_1(\Sigma, \mathbb{V})\to \SL_2$ and $h_{p_i}\in \mathbb{C}$ is such that $T_N(h_{p_i})=-\tr(\rho(\gamma_{p_i}))$. 
 \item The PI-degree $\underline{D}_{\mathbf{\Sigma}}$ of  $\mathcal{S}_A(\mathbf{\Sigma})$ is equal to $N^{3g-3+n+2|\mathcal{A}|}(= D_{\mathbf{\Sigma}^*})$. 
 \item For $\rho: \Pi_1(\Sigma, \mathbb{V})\to \SL_2$, write $\mu(\rho):= (\rho(\alpha(p_{\partial_1})), \ldots, \rho(\alpha(p_{\partial_{|\mathcal{A}|}}))) \in (\SL_2)^{|\mathcal{A}|}$.
 Let  $\widehat{\rho}=(\rho, h_{p_i}) \in \widehat{\underline{X}}(\mathbf{\Sigma})$.
 \begin{itemize}
 \item[(i)] If $\mu(\rho)\in (\SL_2^0)^{|\mathcal{A}|}$ and for each $1\leq i \leq k$, either $\tr(\rho(\gamma_{p_i}))\neq \pm 2$ or $\tr(\rho(\gamma_{p_i}))=\pm 2$ and $h_{p_i}=\mp 2$, then $\widehat{\rho}$ belongs to the Azumaya locus of $\mathcal{S}_A(\mathbf{\Sigma})$.
 \item[(ii)] Suppose that either $\mu(\rho)\in (\SL_2^1)^{|\mathcal{A}|}$ or that $\mu(\rho)\in (\SL_2^0)^{|\mathcal{A}|}$ and there exists $i$ such that $\tr(\rho(\gamma_{p_i}))=\pm 2$ and $h_{p_i}\neq \mp 2$, then $\widehat{\rho}$ does not belong to the Azumaya locus of $\mathcal{S}_A(\mathbf{\Sigma})$.
\end{itemize}
 \end{enumerate}
 \end{theorem}
In this theorem, we have considered the simple Bruhat decomposition $\SL_2= \SL_2^0 \sqcup \SL_2^1$ where $\SL_2^0$ is the subset of matrices $\begin{pmatrix} a & b \\ c & d\end{pmatrix}$ such that $a\neq 0$ and $\SL_2^1$ the set of such matrices with $a=0$. Soon after the prepublication of the present paper, Yu computed in \cite{Yu_CenterSSkein} the center and PI-degrees of every stated skein algebras, thus generalizing the first and second items of Theorem \ref{theorem_Skein_intro}.
\par
In the particular case where $|\mathcal{A}|=1$, i.e. when $\mathbf{\Sigma}=(\Sigma_{g,n+1}, \{a\})$ is a genus $g$ surface with $n+1$ boundary components and a single boundary arc, the algebra $\mathcal{L}_{g,n}:= \mathcal{S}_A(\mathbf{\Sigma})$ is called the \textit{\QMA} in lattice gauge field theory and first  appeared in the work of Buffenoir-Roche and Alekseev-Grosse-Schomerus (\cite{AlekseevGrosseSchomerus_LatticeCS1,AlekseevGrosseSchomerus_LatticeCS2, AlekseevSchomerus_RepCS, BuffenoirRoche, BuffenoirRoche2}). For these algebras, Theorem \ref{theorem_Skein} specializes to 

\begin{corollary}\label{coro_QMA_intro}
\begin{enumerate}
\item The center of $\mathcal{L}_{g,n}$ is generated by the image of the Frobenius morphism and the peripheral curves $\gamma_{p_1}, \ldots, \gamma_{p_n}$.
\item The PI-degree of  $\mathcal{L}_{g,n}$ is $N^{3g+n}$.
\item The Azumaya locus of $\mathcal{L}_{g,n}$ is the locus of elements $\widehat{\rho}=(\rho, h_{p_i})$ such that $\mu(\rho)\in \SL_2^0$ and $\tr(\rho(\gamma_{p_i}))=\pm 2 \Rightarrow h_{p_i}=\mp 2$.
\end{enumerate}
\end{corollary}
When $n=0$, Corollary \ref{coro_QMA_intro} was proved in \cite[Theorem $1$]{GanevJordanSafranov_FrobeniusMorphism}. When $g=0$, the first point of Corollary \ref{coro_QMA_intro} was proved in \cite[Theorem $1.1$]{BaseilhacRoche_LGFT1}, while the second point was proved in \cite[Theorem $1.3$]{BaseilhacRoche_LGFT2}. S.Baseilhac and P.Roche have informed the authors that an alternative proof of the whole Corollary \ref{coro_QMA_intro} and a generalization for arbitrary gauge group (the case considered here is $\SL_2$) will appear in \cite{BaseilhacFaitgRoche_LGFT4}.

\vspace{2mm}
\par Reduced stated skein algebras have been recently generalized for $\SL_n$ gauge group in \cite{LeSikora_SSkein_SLN, LeYu_SLNQTraces} where similar relations with quantum cluster algebras are proved so it seems natural to expect that the techniques of the present paper should extend to higher rank cases. However, it is proved in \cite{FeldvossWitherspoon_RepSmallQG, FeldvossWitherspoon_RepSmallQG2} that the small quantum group $\mathfrak{u}_q\mathfrak{g}$ is wild for simple Lie algebras $\mathfrak{g}$ of rank at least $3$: the problem of classifying its indecomposable representations is unsolvable. Therefore there is little hope that the problem of classifying the semi-weight indecomposable representations of $\SL_n$ reduced stated skein algebras for $n\geq 3$ could be solvable. 
\vspace{2mm}
\par Let us conclude this introduction by explaining why we choose to classify the semi-weight representations. The long-term objective is to construct a $\SL_2$ HQFT whose vector spaces associated to (decorated) surfaces would be indecomposable representations of the reduced stated skein algebras; such HQFT are expected to extend the construction in \cite{BCGPTQFT} and to refine the ones in \cite{BaseilhacBenedettiHTQFT} and  the links invariants of these conjectural HQFT were built in \cite{KojuQGroupsBraidings} using such representations. Since we want to decorate our surfaces by some flat connections thought as points in $X(\mathbf{\Sigma})$, it is natural to consider at least the weight representations. Moreover, in non semi-simple TQFTs \cite{BCGPTQFT} some spaces are semi-weight modules for the skein algebras which are not weight modules; by analyzing the construction carefully we see that this necessity follows from the observation in \cite{Arnaudon_Uqsl2Rep} that the category of semi-weight representations is the smallest  full subcategory of $U_q\mathfrak{sl}_2-\Mod$ which contains the simple modules and is stable by $\otimes$. So it seems natural to consider the class of semi-weight representations.

\subsection{Plan of the paper}
In Section \ref{sec_sskein}, we review the definition and main properties of the reduced stated skein algebras. Section \ref{sec_moduli_spaces} is devoted to the study of the schemes $X(\mathbf{\Sigma})$ and $\widehat{X}(\mathbf{\Sigma})$.
 We then study the representation theory of $ \overline{\mathcal{S}}_{A}(\mathbb{D}_1) \cong \Dq$ in Section \ref{sec_D1}. We then study the case of 2P marked surfaces  in Section \ref{sec_cluster_algebras} using quantum cluster algebras. Theorem \ref{theorem_gluing_intro} is proved in Section \ref{sec_gluing} and  Theorem \ref{main_theorem_intro} is proved in Section \ref{sec_classification}. In Section \ref{sec_QMS}, we prove Theorem \ref{theorem_Skein_intro} and Corollary \ref{coro_QMA_intro}. In Appendix \ref{appendix_glossary} we summarize the terminologies and notations used in the paper.

\vspace{2mm}\par 
\textit{Acknowledgments.} The authors thank S.Baseilhac, D.Calaque,  F.Costantino, M.Faitg,  T.L\^e and P.Roche  for valuable conversations. The skein interpretation of Alekseev's morphism which appears in Section \ref{sec_QMS} was suggested to the authors by T.L\^e which we warmly thank. J.K. acknowledges support from  the European Research Council (ERC DerSympApp) under the European Union’s Horizon 2020 research and innovation program (Grant Agreement No. 768679).
H.K. was partially supported by JSPS KAKENHI Grant Numbers JP22K20342, JP23K12976. 


 \section{Reduced stated skein algebras} \label{sec_sskein}
 
 \subsection{Definition of reduced stated skein algebras}
 
 \begin{definition}[Marked surfaces]\label{def_marked_surfaces}
 \begin{enumerate}
 \item A \textbf{marked surface} is a pair $\mathbf{\Sigma}=(\Sigma, \mathcal{A})$ where $\Sigma$ is a compact oriented surface and $\mathcal{A}$ is a finite set of pairwise disjoint closed intervals embedded in $\partial \Sigma$ named \textbf{boundary edges}. 
 The orientation of $\Sigma$ induces an orientation of $\partial \Sigma$ and thus an orientation of each boundary edge.
 We impose that each boundary edge $a\subset \partial \Sigma$ is equipped with a parametrization $\varphi_a: [0,1] \hookrightarrow \partial \Sigma$ such that $\varphi_a([0,1])=a$ and such that $\varphi_a$ is an oriented embedding (where $[0,1]$ is oriented from $0$ to $1$). A marked surface  $\mathbf{\Sigma}=(\Sigma, \mathcal{A})$ is called \textbf{unmarked} if $\mathcal{A}=\emptyset$ and is said \textbf{essential} if each connected component of $\Sigma$ contains at least one boundary edge.
  \item A connected component of $\partial \Sigma \setminus \mathcal{A}$ is called a \textbf{puncture} and, pictorially, we will represent them as a punctures $\bullet$. There are two kinds of punctures: the ones homeomorphic to a circle, which correspond to unmarked boundary components are called \textbf{inner punctures} and their set is denoted by $\mathring{\mathcal{P}}$; the ones homeomorphic to an open interval, which lies between two boundary edges of the same boundary component, are called \textbf{boundary punctures} and their set is denoted by $\mathcal{P}^{\partial}$. We denote by $\Gamma^{\partial}$ the set of boundary components which intersect $\mathcal{A}$ non trivially. 
 \item If $\mathbf{\Sigma}_1=(\Sigma_1, \mathcal{A}_1)$ and $\mathbf{\Sigma}_2=(\Sigma_2, \mathcal{A}_2)$ are two marked surfaces and $a_1, a_2$ are boundary edges of $\mathbf{\Sigma}_1$ and $\mathbf{\Sigma}_2$ respectively, we denote by $\mathbf{\Sigma}_1\cup_{a_1\# a_2} \mathbf{\Sigma}_2:= (\Sigma_1\cup_{a_1\# a_2}\Sigma_2, (\mathcal{A}_1 \cup \mathcal{A}_2)\setminus \{a_1, a_2\})  $ the marked surface obtained by gluing $a_1$ to $a_2$ where 
 $$ \Sigma_1\cup_{a_1\# a_2}\Sigma_2:= \quotient{ \Sigma_1 \bigsqcup \Sigma_2}{ (\varphi_{a_1}(1-t)=\varphi_{a_2}(t), t\in [0,1])}.$$
\end{enumerate}
 \end{definition}

 \begin{definition}[Tangles and diagrams]\label{def_tangles} Let $\mathbf{\Sigma}=(\Sigma, \mathcal{A})$ be a marked surface.
 \begin{enumerate}
 \item  A \textbf{tangle} in $ \mathbf{\Sigma} \times (0,1)$   is a  compact framed, properly embedded $1$-dimensional manifold $T\subset \Sigma\times (0,1)$ such that  $\partial T \subset \mathcal{A}\times (0,1)$ and 
 for every point of $\partial T$ the framing is parallel to the $(0,1)$ factor and points to the direction of $1$.  Here, by framing, we refer to a section of the unitary normal bundle of $T$. The \textbf{height} of $(v,h)\in \Sigma\times (0,1)$ is $h$.  If $a\in \mathcal{A}$ is a boundary edge, we impose that no two points in $\partial_aT:= \partial T \cap a\times(0,1)$  have the same heights, hence the set $\partial_aT$ is totally ordered by the heights. Two tangles are isotopic if they are isotopic through the class of tangles that preserve the boundary height orders. By convention, the empty set is a tangle only isotopic to itself.
 \item Let $\pi : \Sigma\times (0,1)\to \Sigma$ be the projection with $\pi(v,h)=v$. A tangle $T$ is in a \textbf{standard position} if for each of its points, the framing is parallel to the $(0,1)$ factor and points in the direction of $1$ and is such that $\pi_{\big| T} : T\rightarrow \Sigma$ is an immersion with at most transversal double points in the interior of $\Sigma$. Every tangle is isotopic to a tangle in a standard position. We call \textbf{diagram}  the image $D=\pi(T)$ of a tangle in a standard position, together with the over/undercrossing information at each double point. An isotopy class of diagram $D$ together with a total order of $\partial_a D:=\partial D\cap a$ for each boundary arc $a$, define uniquely an isotopy class of tangle. When choosing an orientation $\mathfrak{o}(a)$ of a boundary edge $a$ and a diagram $D$, the set $\partial_aD$ receives a natural order by setting that the points are increasing when going in the direction of $\mathfrak{o}(a)$. Pictorially, we will depict tangles by drawing a diagram and an orientation (an arrow) for each boundary edge, as in Figure \ref{fig_statedtangle}. When a boundary edge $a$ is oriented we assume that $\partial_a D$ is ordered according to the orientation.
 \item A connected open diagram without double point is called an \textbf{arc}. A closed connected diagram without double point is called a \textbf{loop}.
 \item  A \textbf{state} of a tangle is a map $s:\partial T \rightarrow \{-, +\}$. A pair $(T,s)$ is called a \textbf{stated tangle}. We define a \textbf{stated diagram} $(D,s)$ in a similar manner.
 \end{enumerate}
 \end{definition}
 
 \begin{figure}[!h] 
\centerline{\includegraphics[width=6cm]{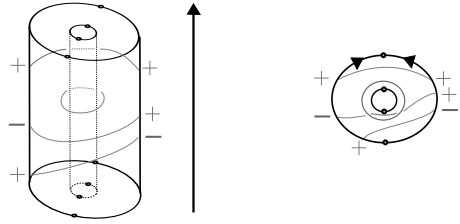} }
\caption{On the left: a stated tangle. On the right: its associated diagram. The arrows represent the height orders. } 
\label{fig_statedtangle} 
\end{figure}

 \begin{definition}[Stated skein algebras](\cite{BonahonWongqTrace, LeStatedSkein})\label{def_skein}
 Let $k$ be a commutative unital ring and $A^{1/2}\in k^{\times}$ and invertible element.
 The \textbf{stated skein algebra}  $\mathcal{S}_A(\mathbf{\Sigma})$ is the  free $k$-module generated by isotopy classes of stated tangles in $\mathbf{\Sigma}\times (0, 1)$ modulo the following relations \eqref{eq: skein 1} and \eqref{eq: skein 2}, 
  	\begin{equation}\label{eq: skein 1} 
\begin{tikzpicture}[baseline=-0.4ex,scale=0.5,>=stealth]	
\draw [fill=gray!45,gray!45] (-.6,-.6)  rectangle (.6,.6)   ;
\draw[line width=1.2,-] (-0.4,-0.52) -- (.4,.53);
\draw[line width=1.2,-] (0.4,-0.52) -- (0.1,-0.12);
\draw[line width=1.2,-] (-0.1,0.12) -- (-.4,.53);
\end{tikzpicture}
=A
\begin{tikzpicture}[baseline=-0.4ex,scale=0.5,>=stealth] 
\draw [fill=gray!45,gray!45] (-.6,-.6)  rectangle (.6,.6)   ;
\draw[line width=1.2] (-0.4,-0.52) ..controls +(.3,.5).. (-.4,.53);
\draw[line width=1.2] (0.4,-0.52) ..controls +(-.3,.5).. (.4,.53);
\end{tikzpicture}
+A^{-1}
\begin{tikzpicture}[baseline=-0.4ex,scale=0.5,rotate=90]	
\draw [fill=gray!45,gray!45] (-.6,-.6)  rectangle (.6,.6)   ;
\draw[line width=1.2] (-0.4,-0.52) ..controls +(.3,.5).. (-.4,.53);
\draw[line width=1.2] (0.4,-0.52) ..controls +(-.3,.5).. (.4,.53);
\end{tikzpicture}
\hspace{.5cm}
\text{ and }\hspace{.5cm}
\begin{tikzpicture}[baseline=-0.4ex,scale=0.5,rotate=90] 
\draw [fill=gray!45,gray!45] (-.6,-.6)  rectangle (.6,.6)   ;
\draw[line width=1.2,black] (0,0)  circle (.4)   ;
\end{tikzpicture}
= -(A^2+A^{-2}) 
\begin{tikzpicture}[baseline=-0.4ex,scale=0.5,rotate=90] 
\draw [fill=gray!45,gray!45] (-.6,-.6)  rectangle (.6,.6)   ;
\end{tikzpicture}
;
\end{equation}

\begin{equation}\label{eq: skein 2} 
\begin{tikzpicture}[baseline=-0.4ex,scale=0.5,>=stealth]
\draw [fill=gray!45,gray!45] (-.7,-.75)  rectangle (.4,.75)   ;
\draw[->] (0.4,-0.75) to (.4,.75);
\draw[line width=1.2] (0.4,-0.3) to (0,-.3);
\draw[line width=1.2] (0.4,0.3) to (0,.3);
\draw[line width=1.1] (0,0) ++(90:.3) arc (90:270:.3);
\draw (0.65,0.3) node {\scriptsize{$+$}}; 
\draw (0.65,-0.3) node {\scriptsize{$+$}}; 
\end{tikzpicture}
=
\begin{tikzpicture}[baseline=-0.4ex,scale=0.5,>=stealth]
\draw [fill=gray!45,gray!45] (-.7,-.75)  rectangle (.4,.75)   ;
\draw[->] (0.4,-0.75) to (.4,.75);
\draw[line width=1.2] (0.4,-0.3) to (0,-.3);
\draw[line width=1.2] (0.4,0.3) to (0,.3);
\draw[line width=1.1] (0,0) ++(90:.3) arc (90:270:.3);
\draw (0.65,0.3) node {\scriptsize{$-$}}; 
\draw (0.65,-0.3) node {\scriptsize{$-$}}; 
\end{tikzpicture}
=0,
\hspace{.2cm}
\begin{tikzpicture}[baseline=-0.4ex,scale=0.5,>=stealth]
\draw [fill=gray!45,gray!45] (-.7,-.75)  rectangle (.4,.75)   ;
\draw[->] (0.4,-0.75) to (.4,.75);
\draw[line width=1.2] (0.4,-0.3) to (0,-.3);
\draw[line width=1.2] (0.4,0.3) to (0,.3);
\draw[line width=1.1] (0,0) ++(90:.3) arc (90:270:.3);
\draw (0.65,0.3) node {\scriptsize{$+$}}; 
\draw (0.65,-0.3) node {\scriptsize{$-$}}; 
\end{tikzpicture}
=A^{-1/2}
\begin{tikzpicture}[baseline=-0.4ex,scale=0.5,>=stealth]
\draw [fill=gray!45,gray!45] (-.7,-.75)  rectangle (.4,.75)   ;
\draw[-] (0.4,-0.75) to (.4,.75);
\end{tikzpicture}
\hspace{.1cm} \text{ and }
\hspace{.1cm}
A^{1/2}
\heightexch{->}{-}{+}
- A^{5/2}
\heightexch{->}{+}{-}
=
\heightcurve.
\end{equation}
The product of two classes of stated tangles $[T_1,s_1]$ and $[T_2,s_2]$ is defined by  isotoping $T_1$ and $T_2$  in $\Sigma \times (1/2, 1) $ and $\Sigma \times (0, 1/2)$ respectively and then setting $[T_1,s_1]\cdot [T_2,s_2]=[T_1\cup T_2, s_1\cup s_2]$. Figure \ref{fig_product} illustrates this product.
\end{definition}
\par For an unmarked surface, $\mathcal{S}_A(\mathbf{\Sigma})$ coincides with the classical  Kauffman-bracket skein algebra.

\begin{figure}[!h] 
\centerline{\includegraphics[width=8cm]{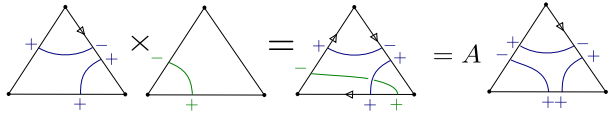} }
\caption{An illustration of the product in stated skein algebras.} 
\label{fig_product} 
\end{figure}

 \begin{definition}[Reduced stated skein algebras]
 \begin{enumerate}
 \item Let $\mathbf{\Sigma}$ a marked surface and $p\in \mathcal{P}^{\partial}$ a boundary puncture between two consecutive boundary edges $a$ and $b$ on the same boundary component $\partial\in \Gamma^{\partial}$.
The orientation of $\Sigma$ induces an orientation of $\partial$ so a cyclic ordering of the elements of $\partial \cap \mathcal{A}$ we suppose that $b$ is followed by $a$ in this ordering.  We denote by $\alpha(p)$ the arc with one endpoint $v_a\in a$ and one endpoint $v_b \in b$ such that $\alpha(p)$ can be isotoped inside $\partial$. $\alpha(p)$ is called a \textbf{corner arc}.
Let $\alpha(p)_{ij}\in \mathcal{S}_A(\mathbf{\Sigma})$ be the class of the stated arc $(\alpha(p), s)$ where $s(v_a)=i$ and $s(v_b)=j$.
 \item We call \textbf{bad arc} associated to $p$ the element $\alpha(p)_{-+}\in \mathcal{S}_A(\mathbf{\Sigma})$ (see Figure \ref{fig_bad_arc}). The \textbf{reduced stated skein algebra} $\overline{\mathcal{S}}_A(\mathbf{\Sigma})$ is the quotient of $\mathcal{S}_A(\mathbf{\Sigma})$ by the ideal generated by all bad arcs.
 \end{enumerate}
 \end{definition}

 \begin{figure}[!h] 
\centerline{\includegraphics[width=4cm]{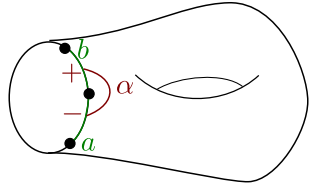} }
\caption{A bad arc.} 
\label{fig_bad_arc} 
\end{figure} 
 
 \begin{remark} The \textbf{quantum trace} is an algebra morphism $\Tr_q : \mathcal{S}_A(\mathbf{\Sigma}) \to \mathcal{Z}_A(\mathbf{\Sigma}, \Delta)$ between the stated skein algebra and the balanced Chekhov-Fock algebra (a variation of the quantum Teichm\"uller space) which is a quantum torus associated to an ideal triangulation $\Delta$ of $\mathbf{\Sigma}$ (see \cite{BonahonWongqTrace, LeStatedSkein}). Its kernel is precisely the ideal generated by bad arcs (\cite[Theorem $7.12$]{CostantinoLe19}), hence it defines an embedding $\Tr_q: \overline{\mathcal{S}}_A(\mathbf{\Sigma}) \hookrightarrow  \mathcal{Z}_A(\mathbf{\Sigma}, \Delta)$. The spaces which appear in quantum hyperbolic geometry are representations of $\mathcal{Z}_A(\mathbf{\Sigma}, \Delta)$, therefore of $ \overline{\mathcal{S}}_A(\mathbf{\Sigma}) $ as well. This is the motivation for the introduction of reduced stated skein algebras and for the study of their representations.
 \end{remark}
 
 The main interest in extending skein algebras to marked surfaces is the existence of splitting morphisms which we now describe.
  Let $a$, $b$ be two distinct boundary edges of $\mathbf{\Sigma}$, denote by $\pi : \Sigma\rightarrow \Sigma_{a\#b}$ the projection and $c:=\pi(a)=\pi(b)$. Let $(T_0, s_0)$ be a stated framed tangle of $\Sigma_{a\#b}\times (0,1)$ transverse to $c\times (0,1)$ and such that the heights of the points of $T_0 \cap c\times (0,1)$ are pairwise distinct and the framing of the points of $T_0 \cap c\times (0,1)$ is vertical towards $1$. Let $T\subset \Sigma \times (0,1)$ be the framed tangle obtained by cutting $T_0$ along $c$. 
Any two states $s_a : \partial_a T \rightarrow \{-,+\}$ and $s_b : \partial_b T \rightarrow \{-,+\}$ give rise to a state $(s_a, s, s_b)$ on $T$. 
Both the sets $\partial_a T$ and $\partial_b T$ are in canonical bijection with the set $T_0\cap c$ by the map $\pi$. Hence the two sets of states $s_a$ and $s_b$ are both in canonical bijection with the set $\mathrm{St}(c):=\{ s: c\cap T_0 \rightarrow \{-,+\} \}$. 

\begin{definition}[Splitting morphism]\label{def_gluing_map}
The \textbf{splitting morphism} $\theta_{a\#b}: \mathcal{S}_{A}(\mathbf{\Sigma}_{a\#b}) \rightarrow \mathcal{S}_{A}(\mathbf{\Sigma})$ is the linear map given, for any $(T_0, s_0)$ as above, by: 
$$ \theta_{a\#b} \left( [T_0,s_0] \right) := \sum_{s \in \mathrm{St}(c)} [T, (s, s_0 , s) ].$$
\end{definition}

\begin{figure}[!h] 
\centerline{\includegraphics[width=8cm]{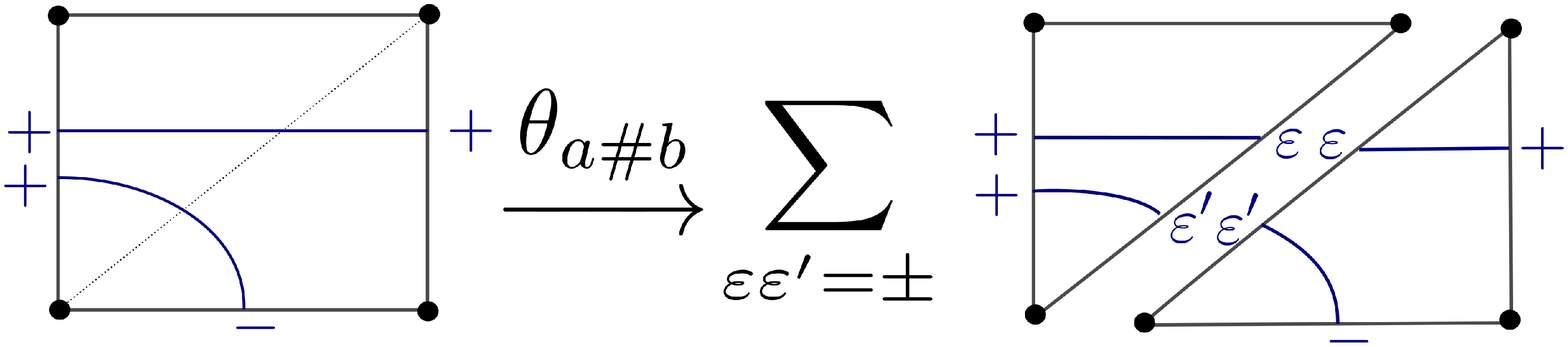} }
\caption{An illustration of the splitting morphism $\theta_{a\#b}$.} 
\label{fig_gluingmap} 
\end{figure} 

\begin{theorem}\label{theorem_gluing}(\cite[Theorem $3.1$]{LeStatedSkein}, \cite[Theorem $7.6$]{CostantinoLe19})
 The linear map $\theta_{a\#b}: \mathcal{S}_A(\mathbf{\Sigma}_{a\#b}) \hookrightarrow \mathcal{S}_A(\mathbf{\Sigma})$ is an injective morphism of algebras. It passes to the quotient to define an injective morphism (still denoted by the same letter) 
 $\theta_{a\#b}: \overline{\mathcal{S}}_A(\mathbf{\Sigma}_{a\#b}) \hookrightarrow \overline{\mathcal{S}}_A(\mathbf{\Sigma})$. 
   Moreover the gluing operation is coassociative in the sense that if $a,b,c,d$ are four distinct boundary arcs, then we have $\theta_{a\#b} \circ \theta_{c\#d} = \theta_{c\#d} \circ \theta_{a\#b}$.
\end{theorem}

 \subsection{Bases}
 
 We now define two bases for the reduced stated skein algebras. Let us first make a preliminary remark: if $T$ is a tangle in standard position and $D$ its planar diagram projection, we cannot recover the isotopy class $T$ from the isotopy class of  $D$ since we lost the information about the height orders of the sets $\partial_aT$, $a\in \mathcal{A}$ in the projection.
 
 \begin{convention}\label{convention_diagram} Let $(D,s)$ be a stated diagram in $\mathbf{\Sigma}$. Recall that each boundary edge $a\in \mathcal{A}$ receives an orientation from the orientation of $\Sigma$ and that for $v,w \in a$ we write $v<_{\mathfrak{o}^+} w$ if $a$ is oriented from $v$ to $w$. Let $(T,s)$ be a stated tangle such that $(1)$ $(T,s)$ is in standard position (in the sense of Definition \ref{def_tangles}) and its planar projection is $(D,s)$ and $(2)$ for every $a\in \mathcal{A}$, if $v, w \in \partial_a D$ are such that $v<_{\mathfrak{o}^+}w$ then $h(v)<h(w)$. Conditions $(1)$ and $(2)$ completely determine the isotopy class of $(T,s)$ and we will write $[D,s]:= [T,s]\in \overline{\mathcal{S}}_A(\mathbf{\Sigma})$. 
 \end{convention}
 
 \begin{definition}[L\^e's bases]
 \begin{enumerate}
 \item A closed component of a diagram $D$ is trivial if it bounds an embedded disc in $\Sigma$. An open component of $D$ is trivial if it can be isotoped, relatively to its boundary, inside some boundary arc. A diagram is \textbf{simple} if it has neither double point nor trivial component. By convention, the empty set is a simple diagram. Let $\mathfrak{o}^+$ denote the orientation of the boundary edges of $\mathbf{\Sigma}$ induced by the orientation of $\Sigma$ as in Definition \ref{def_marked_surfaces}. For each boundary edge $a$ we write $<_{\mathfrak{o}^+}$ the induced total order on $\partial_a D$. A state $s: \partial D \rightarrow \{ - , + \}$ is $\mathfrak{o}^+-$\textbf{increasing} if for any boundary edge $a$ and any two points $x,y \in \partial_a D$, then $x<_{\mathfrak{o}^+} y$ implies $s(x)\leq s(y)$, with the convention $- < +$. 
 \item We denote by $\mathcal{B}^{L}\subset \overline{\mathcal{S}}_{A}(\mathbf{\Sigma})$ the set of classes $[D,s]$ of stated diagrams such that $D$ is simple,  $s$ is $\mathfrak{o}^+$-increasing and such that $(D,s)$ does not contain bad arc component.
 \end{enumerate}
 \end{definition}
 
 \begin{theorem}(\cite[Theorem $7.1$]{CostantinoLe19}) $\mathcal{B}^{L}$ is a basis of $\overline{\mathcal{S}}_{A}(\mathbf{\Sigma})$.\end{theorem}
 
 \begin{lemma}(\cite[Proposition $7.4$]{CostantinoLe19}) Let $p \in \mathcal{P}^{\partial}$. Then $\alpha(p)_{--}$ is the inverse of $\alpha(p)_{++}$ in $ \overline{\mathcal{S}}_{A}(\mathbf{\Sigma})$ (i.e. $\alpha(p)_{++}\alpha(p)_{--}=1$). \end{lemma}
 
 \begin{definition}[M\"uller's bases]\label{def_Muller_basis1}
 Let $\mathcal{B}^M \subset  \overline{\mathcal{S}}_{A}(\mathbf{\Sigma})$ be the set of classes of stated diagrams $[D,s]$ such that:
 \begin{enumerate}
 \item $D$ is simple, 
 \item if $\alpha \subset D$ is an open connected component of $D$ which is not a corner arc, then both endpoints of $\alpha$ have state $+$, 
 \item for each $p\in \mathcal{P}^{\partial}$, there exists a sign $\varepsilon_p\in \{-,+\}$ such that every component of $D$ parallel to $\alpha(p)$ has both endpoints with state $\varepsilon_p$.
 \end{enumerate}
 \end{definition}
 
 \begin{proposition}\label{prop_Muller_basis}
 $\mathcal{B}^M$ is a basis of $\overline{\mathcal{S}}_{A}(\mathbf{\Sigma})$.
 \end{proposition}
 
 \begin{lemma}\label{lemma_Muller_basis}
 Let $(D,s)$ be a stated diagram which contains a bad arc $\alpha_{bad}=\alpha(p)_{-+}\subset (D,s)$. Further suppose that every stated arc $\alpha(p)_{ij} \subset (D,s)$ parallel to $\alpha(p)$ and distinct from $\alpha_{bad}$ has states $i=j=-$. Then $[D,s]=0$ in $\overline{\mathcal{S}}_A(\mathbf{\Sigma})$. 
 
 \end{lemma}
 
 \begin{proof}
 The proof is illustrated in Figure \ref{fig_lemma_Muller_basis} and goes as follows. Let $a$ be the boundary edge adjacent to $p$ containing the endpoint of $\alpha_{bad}$ with state $-$. Let $\alpha_1, \ldots, \alpha_k$ be the subarcs of $D$ which are parallel to $\alpha(p)$ so the underlying arc of $\alpha_{bad}$ is $\alpha_n$ for $1\leq n \leq k$. We suppose that $\alpha_1$ is the arc closest to $p$ in the sense that there exists another arc $\beta\subset \partial \Sigma$ such that $\alpha_1\cup \beta$ bounds a disc which does not intersect any $\alpha_i$ for $i\neq 1$. Let $s'$ be the state of $D$ which is equal to $s$ for every arcs distinct from the $\alpha_i$ and such that $\alpha_1$ has states $(-,+)$ (so is a bad arc) and all $\alpha_i$ for $i\neq 1$ have states $(-,-)$. By applying the skein relation 
   $$ 
\heightexch{->}{-}{+}
= A^{2}
\heightexch{->}{+}{-}
+ A^{-1/2}
\heightcurve.
$$ repeatedly and noting that the class of the stated diagram $\heightcurve$ vanishes due to the skein relation 
$ \begin{tikzpicture}[baseline=-0.4ex,scale=0.5,>=stealth]
\draw [fill=gray!45,gray!45] (-.7,-.75)  rectangle (.4,.75)   ;
\draw[->] (0.4,-0.75) to (.4,.75);
\draw[line width=1.2] (0.4,-0.3) to (0,-.3);
\draw[line width=1.2] (0.4,0.3) to (0,.3);
\draw[line width=1.1] (0,0) ++(90:.3) arc (90:270:.3);
\draw (0.65,0.3) node {\scriptsize{$-$}}; 
\draw (0.65,-0.3) node {\scriptsize{$-$}}; 
\end{tikzpicture}
=0$, we see that $[D,s]=A^{2n}[D,s']$. 
 Let $(T,s')$ be the stated tangle associated to $(D,s)$ using Convention \ref{convention_diagram} so that $[D,s]=[T,s]$ by definition. Let $v\subset \partial \alpha_1$ be the endpoint with state $s'(v)=-$ and let  $T'$ be the tangle obtained from $T$ by isotoping $v$ in the $[0,1]$ direction such that $h(v)<h(v')$ for all $v'\in \partial_aT'$. Using the following height exchange relations, proved in \cite[Lemma $2.4$]{LeStatedSkein}, 
 $$\heightexch{<-}{+}{-}=A \heightexch{->}{+}{-}, \quad \mbox{ and }\quad
\heightexch{<-}{-}{-}=A^{-1} \heightexch{->}{-}{-}$$
 we see that there exists $m\in \mathbb{Z}$ such that $[T,s']=A^m [T',s']$. Let $(D_0,s_0)$ be the stated diagram obtained from $(D,s)$ by removing $\alpha_{bad}$. Then we have $[T',s']= [D_0,s_0]\cdot \alpha_{bad}$ therefore $[D,s]=A^{2n+m}[D_0,s_0]\cdot \alpha_{bad}$ is in the ideal generated by bad arcs so it vanishes in $\overline{\mathcal{S}}_A(\mathbf{\Sigma})$. This concludes the proof.

 \begin{figure}[!h] 
\centerline{\includegraphics[width=16cm]{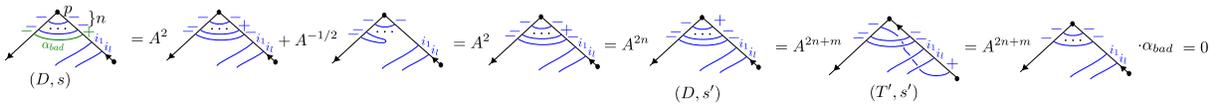} }
\caption{An illustration of the proof of Lemma \ref{lemma_Muller_basis}.} 
\label{fig_lemma_Muller_basis} 
\end{figure}

 \end{proof}

 \begin{proof}
 We claim that there exist a bijection $f: \mathcal{B}^M \to \mathcal{B}^L$ and a map $c: \mathcal{B}^M \to \mathbb{Z}$ such that for all $[D,s]\in \mathcal{B}^M$ we have $f([D,s])=A^{c([D,s])/2} [D,s]$. The fact that $\mathcal{B}^M$ is a basis will then follow from the fact that $\mathcal{B}^L$ is a basis. Let $[D,s] \in \mathcal{B}^M$. If $s$ is $\mathfrak{o}^+$ increasing, we set $f([D,s])=[D,s]$ and $c([D,s])=0$. Else, there exists $a\in \mathcal{A}$ and two consecutive endpoints $v_1 <_{\mathfrak{o}^+} v_2$ in $\partial_aD$ such that $s(v_1)=+$ and $s(v_2)=-$ (i.e. locally $[D,s]$ has the form
  $\heightexch{->}{-}{+}$). By definition of $\mathcal{B}^M$, the arc in $D$ which contains $v_2$ is a corner arc $\alpha(p)_{--}$ with both endpoints with state $-$. Let $[D',s']$ be the class of stated diagram obtained from $[D,s]$ by gluing $v_1$ and $v_2$ together and pushing the resulting point in the interior of $\Sigma$, i.e. is obtained by the local replacement $\heightexch{->}{-}{+} \mapsto \heightcurve$.
  Using the skein relation \eqref{eq: skein 2}, one has: 
  $$ 
\heightexch{->}{-}{+}
= A^{2}
\heightexch{->}{+}{-}
+ A^{-1/2}
\heightcurve = A^{-1/2}
\heightcurve, $$
where the second equality comes from the fact that the stated diagram corresponding to $\heightexch{->}{+}{-}$ contains the bad arc $\alpha(p)_{-+}$ and thus vanishes in  $\overline{\mathcal{S}}_{A}(\mathbf{\Sigma})$ by Lemma \ref{lemma_Muller_basis}. Therefore $[D,s]=A^{-1/2} [D',s']$. If $[D',s']$ is $\mathfrak{o}$-increasing, we set $f([D,s]):= [D',s']$ and $c([D,s])=1$. Else, $[D',s']$ contains a pair of consecutive points of the form $\heightexch{->}{-}{+}$ and we can repeat the process of replacing 
 $\heightexch{->}{-}{+}$ by $\heightcurve$. Since this process decreases the number of endpoints in $\partial D$, after a finite number of steps, say $n$, we obtain an element $f([D,s])\in \mathcal{B}^L$ such that $f([D,s])= A^{n/2} [D,s]$ and we set $c([D,s])=n$. It remains to prove that $f$ is a bijection; let us define an inverse map $g: \mathcal{B}^L \to \mathcal{B}^M$.
  Let $[D,s]\in \mathcal{B}^L$.  If $[D,s]\in \mathcal{B}^M$, we set $g([D,s]):=[D,s]$. Else, 
 since $s$ is $\mathfrak{o}^+$ increasing, $(D,s)$ cannot have two sub stated arcs isotopic to $\alpha(p)_{++}$ and $\alpha(p)_{--}$ respectively for some $p\in \mathcal{P}^{\partial}$. Therefore there exists $a\in \mathcal{A}$, an arc $\alpha \subset D$  
 and $v\in \partial_aD \cap \partial \alpha$ such that $s(v)=-$ and such that $(\alpha, s_{| \partial \alpha})$ is not isotopic to $\alpha(p)_{--}$ for any $p \in \mathcal{P}^{\partial}$. 
 Choose a such $v$ which is minimal for the order $<_{\mathfrak{o}^+}$ in $a$. Let $p, p'\in \mathcal{P}^{\partial}$ be the punctures adjacent to $a$ 
 such that the $\mathfrak{o}^+$ orientation of $a$ points from $p$ to $p'$. Let $b\in \mathcal{A}$ be the boundary edge adjacent to $p$ right before  $a$ (possibly $a=b$ if the connected component of $\partial \Sigma$ containing $a$ does not contain any other boundary edge, i.e. if $p=p'$). 
 If $w\in \partial_a D$ is such that $w<_{\mathfrak{o}^+}v$ and $\beta \subset D$ is the arc with endpoint $w$, then $(\beta, s_{| \partial \beta})$ is isotopic to $\alpha(p)_{--}$. Let $[D',s']$ be the element obtained from $[D,s]$ by slightly pushing $v\in a$ to $v' \in b$ above such arcs $\beta$ and by setting $s'(v')=+$ and by adding a copy of $\alpha(p)_{--}$ as illustrated in Figure \ref{fig_move}. Using the same skein relation \eqref{eq: skein 2} than previously, one finds that $[D',s']=A^{-1/2}[D,s]$. If $[D',s']\in \mathcal{B}^M$ we set $g([D,s]):=[D',s']$. Else, we can find another point $w'\in \partial D'$ and repeat the process. Since the operation $[D,s]\mapsto [D',s']$ decreases the number of endpoints of $D$ with state $-$, after a finite number of steps we obtain an element $g([D,s])\in \mathcal{B}^M$. Clearly $f$ and $g$ are inverse to each other so $f$ is a bijection and the proof is complete.

 \begin{figure}[!h] 
\centerline{\includegraphics[width=8cm]{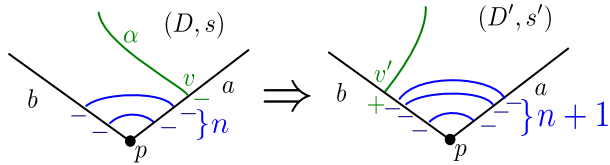} }
\caption{An illustration of the operation $[D,s]\mapsto [D',s']$.} 
\label{fig_move} 
\end{figure}

 \end{proof}
 
 \begin{lemma}\label{lemma_heightexch} Let  $(T,s)$ and $(T',s')$ be two stated tangles in $\mathbf{\Sigma}\times (0,1)$ in standard position (in the sense of Definition \ref{def_tangles})  whose projection on $\mathbf{\Sigma}\times \{1/2\}$ are both equal to the same element $(D,s)$ such that either $[D,s]\in \mathcal{B}^L$ or $[D,s]\in \mathcal{B}^M$.  Then there exists $n\in \mathbb{Z}$ such that $[T,s]=A^{n/2}[T',s']$.
 \end{lemma}
 
 \begin{proof}
 As we proved during the proof of Proposition \ref{prop_Muller_basis}, for each $[D,s]\in \mathcal{B}^M$ there exists $[D',s']\in \mathcal{B}^L$ such that $[D,s]=A^{n/2}[D',s']$ for some $n\in \mathbb{Z}$. So we can assume that the common projection $(D,s)$ of $(T,s)$ and $(T',s')$ is such that $[D,s] \in \mathcal{B}^L$. 
 We say that $(T,s)$ and $(T',s')$ differ by an height exchange move if there exists $a\in \mathcal{A}$ and $v_1, v_2$ two  points in $\partial_a T$ with $h(v_1)<h(v_2)$ which are consecutive, in the sense that there is no $w\in \partial_aT$ with 
 $h(v_1)<_{\mathfrak{o}^+} h(w)<_{\mathfrak{o}^+} h(v_2)$, such that $(T',s')$ is obtained from $(T,s)$ by exchanging the heights of $v_1$ and $v_2$ (this means that $(T,s)\mapsto (T',s')$ is obtained by a local relation $\heightexch{->}{\varepsilon}{\varepsilon'} \mapsto \heightexch{<-}{\varepsilon}{\varepsilon'}$). Any two stated tangles in standard position with the same projection are related by a finite sequence of height exchanges so it suffices to prove the lemma in this case. Exchanging $(T,s)$ and $(T',s')$ if necessary, we can suppose without loss of generality than $v_1 <_{\mathfrak{o}^+} v_2$. 
  Since $[D,s]\in \mathcal{B}^L$, $s$ is $\mathfrak{o}^+$ increasing so  $( s(v_1), s(v_2)) \in \{ (+, +), (-, -), (-,+)\}$. By \cite[Lemma $2.4$]{LeStatedSkein}, one has the following skein relations
 $$ \heightexch{->}{+}{+}= A \heightexch{<-}{+}{+}, \quad  \heightexch{->}{-}{-}= A \heightexch{<-}{-}{-}, \quad  \heightexch{->}{+}{-}= A^{-1} \heightexch{<-}{+}{-}.$$
 So $[T,s]=A^{\pm 1/2}[T',s']$ and the lemma is proved.
 
 \end{proof}

 \subsection{Center, Azumaya locus and shadows}
 
 \begin{convention}
 From now on and in all the rest of the paper with the exception of Subsections \ref{sec_Muller} and \ref{sec_QCA}, we suppose that $k=\mathbb{C}$ and that $A^{1/2}$ is a root of unity of odd order $N\geq 3$ which is fixed once and for all. We denote by $Z_{\mathbf{\Sigma}}$ the center of $\overline{\mathcal{S}}_A(\mathbf{\Sigma})$.
  \end{convention}
  In this subsection, we describe $Z_{\mathbf{\Sigma}}$ and define the Azumaya locus $\mathcal{AL}(\mathbf{\Sigma})$. Stated skein algebras and their reduced versions belong to the following specific class of algebras.
  
  \begin{definition}[Almost Azumaya algebras] A $\mathbb{C}$ associative, unital algebra $\mathscr{A}$ is said \textbf{almost Azumaya} if $(i)$ $\mathscr{A}$ is affine (finitely generated as an algebra), $(ii)$ $\mathscr{A}$ is prime and $(iii)$ $\mathscr{A}$ is finitely generated as a module over its center $Z$. 
  \end{definition}
  Let $Q(Z)$ denote the fraction field of $Z$ obtained by localizing by every non zero element. By a theorem of Posner-Formaneck (see \cite[Theorem $I.13.3$]{BrownGoodearl}), $\mathscr{A}\otimes_Z Q(Z)$ is a central simple algebra and there exists a finite extension $F$ of $Q(Z)$ such that $\mathscr{A} \otimes_Z F\cong \Mat_D(F)$ is a matrix algebra of size $D$. The integer $D$ is called the \textit{PI-degree} of $\mathscr{A}$ and is characterized by the formula $\dim_{Q(Z)}(\mathscr{A}\otimes_ZQ(Z)) = D^2$. Let $\mathcal{X}:= \Specm(Z)$ and for $x\in X$ corresponding to a maximal ideal $\mathfrak{m}_x$, consider the finite dimensional algebra $\mathscr{A}_x:= \quotient{\mathscr{A}}{\mathfrak{m}_x \mathscr{A}}$.

\begin{definition}
Let $\rho: \mathscr{A}\to \End(V)$ be a representation. 
\begin{enumerate}
\item 
$\rho$ is a  \textbf{central representation}  if 
it sends central elements  to  scalar operators. In this case, it induces a character over the center $Z$, so a  point in $\mathcal{X}$ named its \textbf{shadow}.
\item $\rho$ is a \textbf{weight representation} if $V$ is semi-simple as a $Z$-module, i.e. if $V$ is a direct sum of central representations.
\end{enumerate}
\end{definition}

Indecomposable weight representations are clearly central. A central representation  $\rho: \mathscr{A}\to \End(V)$ with shadow $x$ induces by quotient a representation $\rho : \mathscr{A}_x \to \End(V)$.
  
  \begin{definition}[Azumaya locus] The \textbf{Azumaya locus} of $\mathscr{A}$ is the subset
  $$ \mathcal{AL}:= \{x \in X : \mathscr{A}_x \cong \Mat_D(\mathbb{C}) \}.$$
  \end{definition}
  
  Note that $\Mat_D(\mathbb{C})$ is semi-simple and every non trivial simple module is isomorphic to the standard module $\mathbb{C}^D$.
So if $x\in \mathcal{AL}$, there exists a unique indecomposable representation of $\mathscr{A}$ with shadow $x$ (up to isomorphism); this representation is irreducible and has dimension $D$ and is called \textbf{the Azumaya representation} at $x$.
The following theorem is essentially due to De Concini-Kac in their original work on quantum enveloping algebras \cite{DeConciniKacRepQGroups} and was further generalized by several authors including \cite{DeConciniLyubashenko_OqG, Brown_AL_discriminant, BrownGoodearl, FrohmanKaniaLe_UnicityRep}.
\begin{theorem}\label{theorem_AL} Let $\mathscr{A}$ be an almost Azumaya algebra with PI-degree $D$.
\begin{enumerate}
\item (\cite[Theorem III.I.7]{BrownGoodearl}) The Azumaya locus is a Zariski open dense subset.
\item  (\cite[Theorem III.I.6]{BrownGoodearl}) If $\rho: \mathscr{A}\to \End(V)$ is an irreducible representation whose  shadow is not in the Azumaya locus, then $\dim(V)<D$. So the Azumaya locus admits the following alternative description:
$$\mathcal{AL}=\{ x \in \mathcal{X} | x \mbox{ is the shadow of an irreducible representation of maximal dimension }D\}.$$
  Therefore, if $\rho:\mathscr{A} \to \End(V)$ is a $D$-dimensional central representation inducing $x\in \mathcal{X}$, then $\rho$ is irreducible if and only if $x\in \mathcal{AL}$.
\end{enumerate}
\end{theorem}
   
   We now explain why (reduced) stated skein algebras are almost Azumaya and describe their centers. Condition $(ii)$ follows from the 
   \begin{theorem}\label{theorem_domain}(\cite{LeStatedSkein, BonahonWongqTrace, PrzytyckiSikora_SkeinDomain, CostantinoLe19}) Both $\mathcal{S}_A(\mathbf{\Sigma})$ and  $\overline{\mathcal{S}}_A(\mathbf{\Sigma})$ are domain.
   \end{theorem}
   More precisely, Theorem \ref{theorem_domain} is proved in \cite{PrzytyckiSikora_SkeinDomain} for closed (thus unmarked) surfaces, in \cite{BonahonWongqTrace} for unmarked open surfaces and in \cite{LeStatedSkein, CostantinoLe19} for marked surfaces. 
  \par 
  Consider a stated arc $\alpha_{ij}$ in some marked surface $\mathbf{\Sigma}$ and denote by $\alpha_{ij}^{(N)}$ the stated tangle obtained by taking $N$ parallel copies of $\alpha_{ij}$ pushed along the framing direction. In the case where both endpoints of $\alpha$ lye in two distinct boundary edges, one has the equality
$ \alpha_{ij}^{(N)} = (\alpha_{ij})^N$
in $\overline{\mathcal{S}}_A(\mathbf{\Sigma})$, but when both endpoints lye in the same boundary edge, they are distinct. More precisely, suppose the two endpoints, say $v$ and $w$, of $\alpha$ lie in the same boundary edge  with $h(v)>h(w)$. Then  $\alpha_{ij}^{(N)}$ is defined by a stated tangle $(\alpha^{(N)}, s)$ where $\alpha^{(N)}=\alpha_1 \cup \ldots \cup \alpha_N$ represents $N$ copies of $\alpha$ and the endpoints $v_i, w_i$ of the copy $\alpha_i$ are chosen such that $h(v_N) > \ldots > h(v_1) > h(w_N)>\ldots >h(w_1)$.

\begin{definition}[Chebyshev polynomials]  The $N$-th \textbf{Chebyshev polynomial} of the first kind is the polynomial  $T_N(X) \in \mathbb{Z}[X]$ defined by the recursive formulas $T_0(X)=2$, $T_1(X)=X$ and $T_{n+2}(X)=XT_{n+1}(X) -T_n(X)$ for $n\geq 0$.
\end{definition}

\begin{theorem}\label{theorem_Frobenius}(\cite{BonahonWongqTrace} for unmarked surfaces,  \cite{KojuQuesneyClassicalShadows} for marked surfaces;  see also \cite{BloomquistLe, LePaprocki2018})
There is an embedding 
$$ Fr_{\mathbf{\Sigma}}: {\mathcal{S}}_{+1}(\mathbf{\Sigma}) \hookrightarrow \mathcal{Z}({\mathcal{S}}_{A}(\mathbf{\Sigma}))$$
sending the (commutative) algebra at $+1$ into the center of the stated skein algebra at $A^{1/2}$. Moreover, $Fr_{\mathbf{\Sigma}}$ is characterized by the facts that if $\gamma$ is a loop, then $Fr_{\mathbf{\Sigma}}(\gamma) = T_N(\gamma)$ and if $\alpha_{ij}$ is a stated arc, then $Fr_{\mathbf{\Sigma}}(\alpha_{ij})= \alpha_{ij}^{(N)}$. It passes to the quotient to define an embedding (still denoted by the same letter): 
$$ Fr_{\mathbf{\Sigma}}: \overline{\mathcal{S}}_{+1}(\mathbf{\Sigma}) \hookrightarrow {Z}_{\mathbf{\Sigma}}.$$
\end{theorem}
$Fr_{\mathbf{\Sigma}}$ is called the \textbf{Frobenius morphism}.
It is proved in \cite{KojuQuesneyClassicalShadows} that  ${\mathcal{S}}_A(\mathbf{\Sigma})$ is generated by the classes of loops and stated arcs so $Fr_{\mathbf{\Sigma}}$ is indeed characterized by specifying its image on these generators.
A generalization of Theorem \ref{theorem_Frobenius} for marked $3$-manifolds was done in \cite{LeKauffmanBracket} in the unmarked case and \cite{BloomquistLe} for marked $3$-manifolds. As we shall review in Section \ref{sec_cluster_algebras}, by \cite[Theorem $5.2$]{LeYu_SSkeinQTraces} the reduced stated skein algebras are localizations of the M\"uller skein algebras. The existence of Frobenius morphisms for M\"uller skein algebras was proved in \cite{LePaprocki2018} from which we could derive $Fr_{\mathbf{\Sigma}}$ as well.

 \begin{definition}[Punctures and boundary central elements]\label{def_central_elements}
 \begin{enumerate}
 \item   For $p\in \mathring{\mathcal{P}}$ an inner puncture, we denote by $\gamma_p \in \overline{\mathcal{S}}_A(\mathbf{\Sigma})$ the class of a peripheral curve encircling $p$ once.
 \item  For $\partial \in \Gamma^{\partial}$ a boundary component which intersects $\mathcal{A}$ non trivially, denote by $p_1, \ldots, p_n$ the boundary punctures in $\partial$ cyclically ordered by $\mathfrak{o}^+$  and define the elements in $\overline{\mathcal{S}}_A(\mathbf{\Sigma})$:
  $$ \alpha_{\partial} := \alpha(p_1)_{++} \ldots \alpha(p_n)_{++}, \quad \mbox{ and } \quad \alpha_{\partial}^{-1}:= \alpha(p_1)_{--} \ldots \alpha(p_n)_{--}.$$
 \end{enumerate}
 \end{definition}
 
 The elements $\alpha_{\partial}$ and $\alpha_{\partial}^{-1}$ are inverse to each other (see \cite{KojuAzumayaSkein}), hence the notation. 
 
 \begin{definition}[PI-degree]
 \begin{enumerate}
 \item A connected marked surface is \textbf{small} if it is either a disc with $0$ or $1$  boundary arc or an unmarked sphere. 
  \item Let $\mathbf{\Sigma}=(\Sigma_{g,n}, \mathcal{A})$ be a connected marked surface of genus $g$ with $n$ boundary component which is not small. Set $D_{\mathbf{\Sigma}}:= N^{3g-3+n+|\mathcal{A}|}$. If $\mathbf{\Sigma}$ is small, we set $D_{\mathbf{\Sigma}}=1$. We extend it to non connected surfaces by the formula $D_{\mathbf{\Sigma}\bigsqcup \mathbf{\Sigma}'}:= D_{\mathbf{\Sigma}} D_{\mathbf{\Sigma}'}$.
  \end{enumerate}
 \end{definition}
 
 \begin{lemma}\label{lemma_additivityD}
 Let $\mathbf{\Sigma}_1, \mathbf{\Sigma}_2$ be two marked surfaces  and let $a_1, a_2$ be some boundary edges of $\mathbf{\Sigma}_1, \mathbf{\Sigma}_2$ respectively which do not belong to a small component. Then 
 $$ D_{\mathbf{\Sigma}_1 \cup_{a_1 \# a_2}\mathbf{\Sigma}_2} = D_{\mathbf{\Sigma}_1} D_{\mathbf{\Sigma}_2}.$$
 \end{lemma}
 
 \begin{proof} This follows from the facts that $\mathbf{\Sigma}_1 \cup_{a_1 \# a_2}\mathbf{\Sigma}_2$ has $| \mathcal{A}_1| + |\mathcal{A}_2| -2$ boundary edges and has $|\pi_0(\partial \Sigma_1)| + |\pi_0(\partial \Sigma_2)| -1$ boundary components.
  \end{proof}
 
 \begin{theorem}[\cite{FrohmanKaniaLe_UnicityRep} for unmarked surfaces, \cite{KojuAzumayaSkein} for marked surfaces]\label{theorem_center} 
 \begin{enumerate}
 \item The elements $\gamma_p$ and $\alpha_{\partial}^{\pm 1}$ are central and $Z_{\mathbf{\Sigma}}$ is generated by the image of the Frobenius morphism together with these elements. More precisely, $Z_{\mathbf{\Sigma}}$ is isomorphic to the quotient of $$ \overline{\mathcal{S}}_{+1}(\mathbf{\Sigma})[ \gamma_p, \alpha_{\partial}^{\pm 1}; p \in \mathring{P}, \partial \in \Gamma^{\partial} ]$$ by the relations $T_N(\gamma_p)=Fr_{\mathbf{\Sigma}}(\gamma_p)$ and $\alpha_{\partial}^N = Fr_{\mathbf{\Sigma}}(\alpha_{\partial})$.
 \item Both $\mathcal{S}_A(\mathbf{\Sigma})$ and $\overline{\mathcal{S}}_A(\mathbf{\Sigma})$ are finitely generated as algebras and are finitely generated over the images of the Frobenius morphisms (so over their center). Therefore they are almost Azumaya and their Azumaya loci are open dense.
 \item The PI-degree of $\overline{\mathcal{S}}_A(\mathbf{\Sigma})$ is $D_{\mathbf{\Sigma}}$.
 \end{enumerate}
 \end{theorem}
 
 \begin{definition}[Classical schemes]
 \begin{enumerate}
 \item We write $X(\mathbf{\Sigma}):= \Specm( \overline{\mathcal{S}}_{+1}(\mathbf{\Sigma}) )$ and $\widehat{X}(\mathbf{\Sigma}):= \Specm(Z_{\mathbf{\Sigma}})$ and denote by $\pi: \widehat{X}(\mathbf{\Sigma}) \to {X}(\mathbf{\Sigma})$ the morphism induced by $Fr_{\mathbf{\Sigma}}$.
 \item We also denote by $\widehat{\mathcal{X}}(\mathbf{\Sigma})$ the set of primes  $\mathcal{I} \in \Spec(Z_{\mathbf{\Sigma}})$ such that $\mathcal{I}\cap Z_{\mathbf{\Sigma}}^0$ is maximal in $Z_{\mathbf{\Sigma}}^0$ and denote by $\widehat{\pi}: \widehat{\mathcal{X}}(\mathbf{\Sigma}) \to X(\mathbf{\Sigma})$ the projection $\mathcal{I}\mapsto \mathcal{I}\cap Z_{\mathbf{\Sigma}}^0$.
 \item For $\mathbf{\Sigma}_1, \mathbf{\Sigma}_2$ with boundary edges $a_1$ and $a_2$, we denote by $ p_{a_1 \# a_2} : {X}(\mathbf{\Sigma}_1)\times {X}(\mathbf{\Sigma}_2) \to {X}(\mathbf{\Sigma}_1\cup_{a_1 \# a_2} \mathbf{\Sigma}_2)$ and $\widehat{p}_{a_1 \# a_2} : \widehat{X}(\mathbf{\Sigma}_1)\times \widehat{X}(\mathbf{\Sigma}_2) \to \widehat{X}(\mathbf{\Sigma}_1\cup_{a_1 \# a_2} \mathbf{\Sigma}_2)$ the dominant maps induced by $\theta_{a_1 \# a_2}$.
 \end{enumerate}
 \end{definition}
 Note that the Frobenius morphisms are compatible with the splitting morphisms in the sense that $Fr_{\mathbf{\Sigma}} \circ \theta_{a\# b} = \theta_{a\# b}\circ Fr_{\mathbf{\Sigma}}$.
 By Theorem \ref{theorem_center}, as a set  $\widehat{X}(\mathbf{\Sigma})$ is described as 
 \begin{multline*}  \widehat{X}(\mathbf{\Sigma}) = \{ \widehat{x}= (x, h_p, h_{\partial}): x\in X(\mathbf{\Sigma}),
 h_p\in \mathbb{C} \mbox{ is such that  } T_N(h_p)=\chi_x(\gamma_p) \mbox{ for }p\in \mathring{P}, 
\\ h_{\partial} \in \mathbb{C}^* \mbox{ is such that  } h_{\partial}^N = \chi_x(\alpha_{\partial}) \mbox{ for }\partial \in \Gamma^{\partial} \}.\end{multline*}

 \begin{definition}[Fully Azumaya locus]\label{def_AL} Let $\mathcal{AL}(\mathbf{\Sigma})\subset \widehat{X}(\mathbf{\Sigma})$ be the Azumaya locus of $\overline{\mathcal{S}}_A(\mathbf{\Sigma})$.
 The \textbf{fully Azumaya locus} of $ \overline{\mathcal{S}}_A(\mathbf{\Sigma})$ is the subset  $\mathcal{FAL}(\mathbf{\Sigma}) \subset {X}(\mathbf{\Sigma}) $ of elements $x\in {X}(\mathbf{\Sigma})$ such that $\pi^{-1}(x)\subset \mathcal{AL}(\mathbf{\Sigma})$. 
 \end{definition}

  Let $Z_{\mathbf{\Sigma}}^0\subset Z_{\mathbf{\Sigma}}$ denote the image of the Frobenius morphism (so $Z_{\mathbf{\Sigma}}^0\cong \overline{\mathcal{S}}_{+1}(\mathbf{\Sigma})$).
  Recall that a representation is called \textbf{semi-weight representation} if the induced module over $Z_{\mathbf{\Sigma}}^0$ is semi-simple.
  An indecomposable semi-weight representation $r$ induces a character over $Z_{\mathbf{\Sigma}}^0$ so defines a point in $x_r\in X(\mathbf{\Sigma})$ which we call the \textbf{classical shadow} of $r$. 
  
  \begin{theorem}\label{theorem_FAL}
  \begin{enumerate}
  \item Let $r:\overline{\mathcal{S}}_A(\mathbf{\Sigma})$ be an indecomposable semi-weight representation with classical shadow $x_r \in X(\mathbf{\Sigma})$. Then  $\mathcal{I}_r:=\ker(r)\cap Z_{\mathbf{\Sigma}}$ is a prime ideal of $Z_{\mathbf{\Sigma}}$, so $\mathcal{I}_r\in \widehat{\pi}^{-1}(x_r)$.
  \item Every prime ideal in $\widehat{\mathcal{X}}(\mathbf{\Sigma})$ is contained in a unique maximal ideal in ${\widehat{X}}(\mathbf{\Sigma})$.
  \item Let $\widehat{x} \in \widehat{X}(\mathbf{\Sigma})$ and decompose the set of inner punctures as $\mathring{\mathcal{P}}=\mathring{\mathcal{P}}_0 \sqcup \mathring{\mathcal{P}}_1$ where 
  $$ \mathring{\mathcal{P}}_1:= \left\{ p \in \mathring{\mathcal{P}}: \chi_x(\gamma_p)=\pm 2 \mbox{ and }h_p \neq \pm 2\right\}.$$
  Then the set of primes contained in ${\widehat{x}}$ is in bijection with the set $\col(\widehat{x})$ of maps $c: \mathring{\mathcal{P}}\to \{S, P\}$ such that $c(p)=S$ for all $p\in \mathring{\mathcal{P}}_0$.
  \item If $x\in \mathcal{FAL}(\mathbf{\Sigma})$, then the map $r\mapsto \mathcal{I}_r$ induces a $1:1$ correspondance between isomorphism classes of semi-weight indecomposable representations with classical shadow $x$ and the set $\widehat{\pi}^{-1}(x)$. Therefore, for $\widehat{x}\in \pi^{-1}(x)$, the set of isomorphism classes of  semi-weight indecomposable representations such that $\mathcal{I}_r\subset \widehat{x}$ is in $1:1$ correspondance with $\col(\widehat{x})$. Under this correspondance, there is a unique simple module with $\mathcal{I}_r\subset \widehat{x}$ which corresponds to the coloring $c$ such that $c(p)=S$ for all $p\in \mathring{\mathcal{P}}$ and there is a unique projective indecomposable semi-weight module with $\mathcal{I}_r\subset \widehat{x}$ which corresponds to the coloring $c$ such that $c(p)=P$ for all $p\in \mathring{\mathcal{P}}_1$.
  \end{enumerate}
  \end{theorem}

  \begin{definition}[Shadows]
  Let $r:\overline{\mathcal{S}}_A(\mathbf{\Sigma})$ be an indecomposable semi-weight representation.
  \begin{enumerate}
  \item The \textbf{classical shadow} of $r$ is the maximal ideal $x_r:= \ker(r) \cap Z^0_{\mathbf{\Sigma}} \in X(\mathbf{\Sigma})$.
  \item The \textbf{full shadow} of $r$ is the prime ideal $\mathcal{I}_r:= \ker(r)\cap Z_{\mathbf{\Sigma}} \in \widehat{\mathcal{X}}(\mathbf{\Sigma})$.
  \item The \textbf{maximal shadow} of $r$ is the (unique) maximal ideal $\widehat{x}_r \in \widehat{X}(\mathbf{\Sigma})$ containing $\mathcal{I}_r$. 
  \end{enumerate}
  \end{definition}
  
  So Theorem \ref{theorem_FAL} implies that semi-weight indecomposable modules with classical shadows in the fully Azumaya locus are completely classified, up to isomorphism, by their full shadows.   The proof of Theorem \ref{theorem_FAL} will occupy all the rest of the subsection.
  
  \begin{lemma}\label{lemma_FAL1}
  Let $\mathscr{A}$ be an associative, unital complex algebra with center $Z$. Let $r: \mathscr{A} \to \End(V)$ be a non zero indecomposable representation. Then $\mathcal{I}_r:= \ker(r)\cap Z$ is a prime ideal in $Z$.
  \end{lemma}
  
  \begin{proof}
  By contradiction, suppose that $\mathcal{I}_r$ is not prime. If $\mathcal{I}_r=Z$ then it contains the unit $1$, so $r$ is null. Else, there exist $\mathcal{J}_1, \mathcal{J}_2\subset Z$ some ideals such that $\mathcal{J}_1\mathcal{J}_2\subset \mathcal{I}_r$ and $\mathcal{J}_1\not\subset \mathcal{I}_r$, $\mathcal{J}_2 \not \subset \mathcal{I}_r$. For $i=1,2$, set 
  $$ V_i:= \{ v\in V : r(\mathcal{J}_i)v=0\}.$$
  Clearly, $V_i$ is preserved by $\mathscr{A}$. 
  The condition $\mathcal{J}_i \not \subset \mathcal{I}_r$ implies that $V_i \neq V$. The condition $\mathcal{J}_1\mathcal{J}_2\subset \mathcal{I}_r$ implies that $r(\mathcal{J}_2)(V) \subset V_1$. Since $\mathcal{J}_2\not \subset \mathcal{I}_r$, there exists $v\in V$ and $x\in \mathcal{J}_2$ such that $r(x)v\neq 0$, so $V_1\neq 0$. Therefore $0\neq V_1 \subsetneq V$ is a non zero proper submodule of $V$, so $r$ is not indecomposable. This concludes the proof.

  \end{proof}
  
  \begin{notations} For $x\in X(\mathbf{\Sigma})$ we set 
  $$ \overline{\mathcal{S}}_{A}(\mathbf{\Sigma})_x:= \quotient{\overline{\mathcal{S}}_{A}(\mathbf{\Sigma})}{\mathfrak{m}_x\overline{\mathcal{S}}_{A}(\mathbf{\Sigma})} \quad \mbox{ and } Z(x):= \quotient{Z_{\mathbf{\Sigma}}}{\mathfrak{m}_x Z_{\mathbf{\Sigma}}}.$$
  \end{notations}
  
   Note that $r$ has classical shadow $x$ if and only if it factorizes through the finite dimensional algebra $\overline{\mathcal{S}}_{A}(\mathbf{\Sigma})_x$.   The following theorem is useful to classify the indecomposable semi-weight representations whose classical shadows lye in the fully Azumaya locus.
  
  \begin{theorem}( \cite[Corollary $2.7$]{BrownGordon_ramificationcenters})\label{theorem_BG_FAL} If $x\in \mathcal{FAL}(\mathbf{\Sigma})$, then 
  $ \overline{\mathcal{S}}_{A}(\mathbf{\Sigma})_x \cong \Mat_D(Z(x))$.
  \end{theorem}
  
  Denote by  $\mathscr{D}:= \quotient{\mathbb{C}[X]}{(X^2)}$ the dual numbers algebra. 
  
 \begin{lemma}\label{lemma_FAL2} For $x\in X(\mathbf{\Sigma})$, one has an isomorphism of algebras
 $$ Z(x)\cong \left( \otimes _{p\in \mathring{\mathcal{P}}} Z(x;p) \right) \otimes \left( \otimes_{\partial \in \Gamma^{\partial}} Z(x; \partial)\right), $$
 where 
 $$ Z(x;p):=
 \left\{ \begin{array}{ll}
 \mathbb{C}\oplus \mathscr{D} ^{\oplus \frac{N-1}{2}} &, \mbox{ if }\chi_x(\gamma_p)=\pm 2; \\
 \mathbb{C}^{\oplus N} &, \mbox{ else.}
 \end{array} \right. 
 \mbox{ and }\quad Z(x; \partial):= \mathbb{C}^{\oplus N}.$$
 Therefore, if for $\widehat{x}\in \widehat{X}(\mathbf{\Sigma})$ we write $m(\widehat{x}):= | \mathring{\mathcal{P}}_1|$ in the partition $\mathring{\mathcal{P}}=\mathring{\mathcal{P}}_0 \sqcup \mathring{\mathcal{P}}_1$  of Theorem \ref{theorem_FAL}, and setting $\mathscr{D}^{\otimes 0}:= \mathbb{C}$, we have
 $$ Z(x)\cong \oplus_{\widehat{x} \in \pi^{-1}(x)} Z(\widehat{x}), \quad \mbox{ where } Z(\widehat{x}):= \mathscr{D}^{\otimes m(\widehat{x})}.$$
 \end{lemma}
 
 \begin{proof}
 Let us first make a preliminary remark. Let $P(X)\in\mathbb{C}[X]$ be a polynomial with decomposition $P(X)=\prod_{i=1}^d (X-\lambda_i)^{n_i}$ where $n_i \geq 1$ and the $\lambda_i$ are pairwise distinct. 
\par  \underline{Claim:} We have an isomorphism of algebras:
 $$ \quotient{\mathbb{C}[X]}{(P(X))} \cong \oplus_{i=1}^d \quotient{\mathbb{C}[X]}{ (X^{n_i})}.$$
 Indeed, since the ideals $(X-\lambda_i)^{n_i}$ are pairwise coprime, by the chinese reminder theorem we have an isomorphism $ \quotient{\mathbb{C}[X]}{(P(X))} \cong \oplus_{i=1}^d \quotient{\mathbb{C}[X]}{ \left((X-\lambda_i)^{n_i}\right)}$ and setting $Y:=X-\lambda_i$ one has $\quotient{\mathbb{C}[X]}{ \left((X-\lambda_i)^{n_i}\right)}\cong \quotient{\mathbb{C}[Y]}{ (Y^{n_i})}$.
 \par By Theorem \ref{theorem_Frobenius}, one has 
 \begin{multline*}
  Z(x)\cong \quotient{\mathbb{C}[X_p, X_{\partial};  p\in \mathring{P}, \partial \in \Gamma^{\partial} ]}{\left( T_N(X_p)=\chi_x(\gamma_p), X_{\partial}^N = \chi_x(\alpha_{\partial})\right)} \\
  \cong 
 \left( \otimes_{p} \quotient{\mathbb{C}[X_p]}{(T_N(X_p)-\chi_x(\alpha_p))}\right) \otimes \left(\otimes_{\partial} \quotient{\mathbb{C}[X_{\partial}]}{(X_{\partial}^N-\chi_x(\alpha_{\partial}))} \right).\end{multline*}
 Let  $z\in \mathbb{C}$ and consider $\widetilde{z}$ such that $\widetilde{z}^N=z$. 
 Then  $X^N-z= \prod_{n\in \mathbb{Z}/N\mathbb{Z}} (X-A^n \widetilde{z})$  so $\quotient{\mathbb{C}[X]}{(X^N-z)} \cong \mathbb{C}^{\oplus N}$ (the isomorphism sends $f(X)$ to $(f(A^n \widetilde{z}))_{n\in \mathbb{Z}/N\mathbb{Z}}$).
 \par  Suppose $z\neq 0$ and set $x:= z+z^{-1}$. Then $T_N(X)-x= \prod_{n\in \mathbb{Z}/N\mathbb{Z}} (X- A^n \widetilde{z} - A^{-n}\widetilde{z}^{-1})$. If $x\neq \pm 2$ then the roots $\lambda_n:=A^n \widetilde{z} + A^{-n}\widetilde{z}^{-1}$ are pairwise distinct so $\quotient{\mathbb{C}[X]}{(T_N(x)-z)} \cong \mathbb{C}^{\oplus N}$. If $x=\pm 2$, then $T_N(x)-\pm 2= (X-\pm 2) \prod_{n=1}^{(N-1)/2}( X-\pm A^n -\pm A^{-n})^2$ so $\quotient{\mathbb{C}[X]}{(T_N(x)-z)} \cong \mathbb{C}\oplus \mathscr{D}^{\oplus \frac{N-1}{2}}$. This proves the first assertion. The second is an immediate consequence.
 
 \end{proof}

 \begin{lemma}\label{lemma_FAL3} Let $\mathds{1}$ be the $1$-dimensional $\mathscr{D}$-module generated by the single vector $v$ such that $X\cdot v =0$. Let $W$ be the $2$-dimensional $\mathscr{D}$-module which corresponds to the left regular representation, i.e. which has basis $x,y$ such that $X\cdot x= 0$ and $X\cdot y = x$. Then every indecomposable $\mathscr{D}$-module is isomorphic to either $\mathds{1}$ or $W$ so the set of indecomposable $\mathscr{D}$ module is in $1:1$ correspondance with their annihilator (prime) ideal. Moreover we have a non split exact sequence of $\mathscr{D}$ modules
 $$ 0 \to \mathds{1} \xrightarrow{i} W \xrightarrow{p} \mathds{1} \to 0,$$
 where $i(v)=x$, $p(x)=0$ and $p(y)=v$. In particular $\mathds{1}$ is the only simple $\mathscr{D}$ module and $W$ is the only  projective indecomposable $\mathscr{D}$ module.
 \end{lemma}
 
 \begin{proof} Let $r: \mathscr{D} \to \End(V)$ be an indecomposable representation. Then $r$ is determined by $\pi:= r(X)$. Up to conjugacy, we can consider that  $\pi$ is in Jordan normal form. Since $r$ is indecomposable, $\pi$ is a Jordan matrix which squares to $0$. There exists exactly two Jordan matrices which squares to $0$: the matrix $(0)$ and $\begin{pmatrix}0 & 1 \\ 0 & 0\end{pmatrix}$. If $r(X)$ is $(0)$ then $V=\mathds{1}$; if $r(X)$ is conjugate to  $\begin{pmatrix}0 & 1 \\ 0 & 0\end{pmatrix}$ then   $r$ is isomorphic to $W$. 
 \par The fact that the short sequence of the lemma is exact and does not split follows from a straightforward computation. This implies that $W$ is not simple and that $\mathds{1}$ is not projective. Since $W$ is free, it is projective. 
 This concludes the proof.
 
 \end{proof}

 \begin{proof}[Proof of Theorem \ref{theorem_FAL}]
 Item $(1)$ follows from Lemma \ref{lemma_FAL1}. Note that the set of primes/maximal ideals $\mathcal{I}\subset Z$ containing $\widehat{x}\subset Z^0\subset Z$ is in $1:1$ correspondance with the set of primes/maximal ideals in $Z(\widehat{x})$. By Lemma \ref{lemma_FAL2}, we have $Z(\widehat{x})\cong \mathscr{D}^{\otimes \mathring{\mathcal{P}}_1}$. So Item $(2)$ follows from the fact that $Z(\widehat{x})$ is local whereas Item $(3)$ follows from the fact that we have a bijection $\col(\widehat{x}) \cong \Spec(Z(\widehat{x}))$ sending $c$ to the prime ideal $\oplus_{p \in \mathring{\mathcal{P}_1}} \mathcal{I}_{c(p)} \subset\mathscr{D}^{\otimes \mathring{\mathcal{P}}_1}\cong Z(\widehat{x}) $ defined by $\mathcal{I}_{c(p)}= \left\{ \begin{array}{ll} 0 \mbox{, if }c(p)=S; \\ (X) \mbox{, if }c(p)=P. \end{array} \right.$.  
 Let us prove $(4)$.
 Let $x\in \mathcal{FAL}(\mathbf{\Sigma})$. By definition, an indecomposable representation of $\overline{\mathcal{S}}_A(\mathbf{\Sigma})$ with classical shadow $x$ is a representation of $\overline{\mathcal{S}}_A(\mathbf{\Sigma})_x$. 
 By Theorem \ref{theorem_FAL}, $\overline{\mathcal{S}}_A(\mathbf{\Sigma})_x \cong \Mat_D(Z(x)) \cong \Mat_D(\mathbb{C}) \otimes Z(x)\cong \oplus_{\widehat{x}\in \pi^{-1}(x)} \Mat_D(\mathbb{C})\otimes Z(\widehat{x})$. The matrix algebra $\Mat_D(\mathbb{C})$ is semi-simple with unique simple module the standard $D$-dimensional module $V=\mathbb{C}^D$ which is projective. So the map $U \mapsto V\otimes U$ defines a $1:1$ correspondence between the indecomposable $Z(\widehat{x})$ modules and the indecomposable representations of $\overline{\mathcal{S}}_A(\mathbf{\Sigma})$ with maximal shadow $\widehat{x}$. In this correspondence $V\otimes U$ is simple (resp. projective) if and only if $U$ is simple (resp. projective). By Lemma \ref{lemma_FAL2} $Z(\widehat{x})\cong \mathscr{D}^{\otimes \mathring{\mathcal{P}}_1}$ so Lemma \ref{lemma_FAL3} implies that the (isomorphism classes of) indecomposable $Z(\widehat{x})$-modules are in 
 $1:1$ correspondance with $\Spec(Z(\widehat{x}))\cong \col(\widehat{x})$. The unique simple $Z(\widehat{x})$ module is the one dimensional module $\mathbb{C}$ which corresponds to the unique maximal ideal $0\in \Spec(Z(\widehat{x}))$ so to the coloring $c$ such that $c(p)=S$ for all $p$. The unique projective $Z(\widehat{x})$ module is the free rank one module $Z(\widehat{x})\cong W^{\otimes \mathring{\mathcal{P}}_1}$ which corresponds to the coloring sending all $p\in  \mathring{\mathcal{P}}_1$ to $P$. This concludes the proof.

 \end{proof}

  \section{Classical moduli spaces}\label{sec_moduli_spaces}
 
 In this subsection, we give a geometric interpretation of $X(\mathbf{\Sigma})$ following \cite{FockRosly, KojuTriangularCharVar} and study its basic properties.
 
 \subsection{Geometric interpretation of $X(\mathbf{\Sigma})$}
 
 \begin{definition}[Relative representation varieties]\label{def_modulispaces} Let $\mathbf{\Sigma}=(\Sigma, \mathcal{A})$ be an essential  marked surface.
 \begin{enumerate}
 \item The \textbf{fundamental groupoid} $\Pi_1(\Sigma)$ is the groupoid whose objects are points in $\Sigma$ and whose morphisms $\beta: v_1 \to v_2$ are homotopy classes of path $c_{\beta}: [0,1] \to \Sigma$ such that $c_{\beta}(0)=v_1$ and $c_{\beta}(1)=v_2$. We write $v_1=s(\beta)$ (the source) and $v_2= t(\beta)$ (the target). The composition is the concatenation of paths and the unit at $v\in \Sigma$ is the class $1_v$ of the constant path. For $\beta: v_1 \to v_2$, we denote by $\beta^{-1}: v_2 \to v_1$ the class of the path $c_{\beta^{-1}}(t) = c_{\beta}(1-t)$ (so $\beta \beta^{-1}=1_{t(\beta)}$). For $a\in \mathcal{A}$ we denote by $v_a \in a$ the middle point $v_a:= \varphi_a(1/2)$ and denote by $\mathbb{V}=\{ v_a, a\in \mathcal{A}\}$ the set of such points. $\Pi_1(\Sigma, \mathbb{V})$ is the full subcategory of $\Pi_1(\Sigma)$ generated by $\mathbb{V}$. By abuse of notation, we also denote by $\Pi_1(\Sigma, \mathbb{V})$ the set of morphisms of the underlying category.
 \item The \textbf{relative representation variety} $\mathcal{R}_{\SL_2}(\mathbf{\Sigma})$ is the set of functors $\rho: \Pi_1(\Sigma, \mathbb{V}) \to \SL_2$, where $\SL_2$ is seen as a category with only one object $*$ whose set of  endomorphisms is $\SL_2(\mathbb{C})$. It admits a structure of a complex affine variety whose algebra of regular functions is 
 $$ \mathcal{O}[\mathcal{R}_{\SL_2}(\mathbf{\Sigma})]:= \quotient{ \mathbb{C}[X_{ij}^{\beta}, i,j\in \{-,+\}, \beta \in \Pi_1(\Sigma, \mathbb{V})]}{\left( M_{\beta_1}M_{\beta_2}=M_{\beta_1 \beta_2}, \det(M_{\beta})=1 \right)}.$$
 Here, for $\beta \in \Pi_1(\Sigma, \mathbb{V})$, $M_{\beta}$ represents the $2\times 2$ matrix with coefficients in the polynomial ring $ \mathbb{C}[X_{ij}^{\beta}, i,j\in \{-,+\}, \beta \in \Pi_1(\Sigma, \mathbb{V})]$ defined by 
 $M_{\beta}=\begin{pmatrix} X_{++}^{\beta} & X_{+-}^{\beta} \\ X_{-+}^{\beta} & X_{--}^{\beta} \end{pmatrix}$ and we quotient by the relations $\det(M_{\beta}):=X_{++}^{\beta}X_{--}^{\beta}- X_{+-}^{\beta}X_{-+}^{\beta} =1$ for all $\beta \in  \Pi_1(\Sigma, \mathbb{V})$ and by the four matrix coefficients of $ M_{\beta_1}M_{\beta_2}-M_{\beta_1 \beta_2}$ for every pair of composable paths (i.e. such that $t(\beta_2)=s(\beta_1)$).
 Clearly the set of closed points of $\mathcal{R}_{\SL_2}(\mathbf{\Sigma}):= \Specm(\mathcal{O}[\mathcal{R}_{\SL_2}(\mathbf{\Sigma})])$ is in canonical bijection with the set of functors  $\rho: \Pi_1(\Sigma, \mathbb{V}) \to \SL_2$.
 \item A \textbf{presenting graph} $\Gamma$ for $\mathbf{\Sigma}$ is an embedded oriented graph $\Gamma \subset \Sigma$ whose set of vertices is $\mathbb{V}$ and such that $\Sigma$ retracts on $\Gamma$. We denote by $\mathcal{E}(\Gamma)$ the set of its oriented edges whose elements are seen as paths in $\Pi_1(\Sigma, \mathbb{V})$. 
 \item An oriented arc $\alpha$  in $\mathbf{\Sigma}$ naturally defines a path in $\Pi_1(\Sigma, \mathbb{V})$ which we abusively also denote by $\alpha$. For $i,j \in \{-, +\}$, we denote by $\alpha_{ij}\in \mathcal{S}_A(\mathbf{\Sigma})$ the class of the arc $\alpha$ with state $i$ at its source point $s(\alpha)$ at state $j$ at its target point $t(\alpha)$. For $p\in {\mathcal{P}}^{\partial}$ we fix the canonical orientation of the corner arc $\alpha(p)$ such that $\alpha(p)_{-+}$ is a bad arc. The \textbf{small Bruhat cell} is the subset of $\SL_2(\mathbb{C})$ defined by
 $$ \SL_2^{1}:= \{ M = \begin{pmatrix} a & b \\ c & d \end{pmatrix} \in \SL_2(\mathbb{C}) \mbox{ such that }a=0 \}.$$
 The \textbf{reduced relative representation variety} is the subvariety: 
 $$ \overline{\mathcal{R}}_{\SL_2}(\mathbf{\Sigma}) = \{ \rho : \Pi_1(\Sigma, \mathbb{V}) \to \SL_2 \mbox{ such that } \rho(\alpha(p))\in \SL_2^{1} \mbox{ for all }p\in {\mathcal{P}}^{\partial} \}.$$
 Said differently, its algebra of regular functions is 
 $$ \mathcal{O}[\overline{\mathcal{R}}_{\SL_2}(\mathbf{\Sigma})] := \quotient{ \mathcal{O}[\mathcal{R}_{\SL_2}(\mathbf{\Sigma})]}{ ( X_{++}^{\alpha(p)}, p\in {\mathcal{P}}^{\partial})}.$$
 \end{enumerate}
 
 \end{definition}
 
 Let $\Arc(\mathbf{\Sigma})$ be the set of oriented arcs in $\mathbf{\Sigma}$.
 
 \begin{theorem}\label{theorem_classical_limit}(\cite[Theorem $3.18$]{KojuQuesneyClassicalShadows}, see also \cite[Theorem $4.7$]{KojuPresentationSSkein}) Let $\mathbf{\Sigma}$ be an essential marked surface. There exists a map $w: \Arc(\mathbf{\Sigma})\to \{0, 1\}$ and an isomorphism $\Psi_w: \mathcal{S}_{+1}(\mathbf{\Sigma}) \xrightarrow{\cong} \mathcal{O}[\mathcal{R}_{\SL_2}(\mathbf{\Sigma})]$ characterized by the formula:
$$ \Psi_w \begin{pmatrix} \alpha_{++} & \alpha_{+-} \\ \alpha_{-+} & \alpha_{--} \end{pmatrix} = (-1)^{w(\alpha)} \begin{pmatrix} 0 & -1 \\ 1 & 0\end{pmatrix} \begin{pmatrix} X_{++}^{\alpha} & X_{+-}^{\alpha} \\ X_{-+}^{\alpha} & X_{--}^{\alpha} \end{pmatrix} = (-1)^{w(\alpha)} \begin{pmatrix} -X_{-+}^{\alpha} & -X_{--}^{\alpha} \\ X_{++}^{\alpha} & X_{+-}^{\alpha} \end{pmatrix}, \quad \mbox{ for all }\alpha \in \Arc(\mathbf{\Sigma}).$$
 \end{theorem}
 
 \begin{corollary}\label{coro_classical_limit}  Let $\mathbf{\Sigma}$ be an essential marked surface. There exists a map $w: \Arc(\mathbf{\Sigma})\to \{0, 1\}$ and an isomorphism $\overline{\Psi}_w: \overline{\mathcal{S}}_{+1}(\mathbf{\Sigma}) \xrightarrow{\cong} \mathcal{O}[\overline{\mathcal{R}}_{\SL_2}(\mathbf{\Sigma})]$ characterized by the formula:
$$ \overline{\Psi}_w \begin{pmatrix} \alpha_{++} & \alpha_{+-} \\ \alpha_{-+} & \alpha_{--} \end{pmatrix} = (-1)^{w(\alpha)} \begin{pmatrix} -X_{-+}^{\alpha} & -X_{--}^{\alpha} \\ X_{++}^{\alpha} & X_{+-}^{\alpha} \end{pmatrix}, \quad \mbox{ for all }\alpha \in \Arc(\mathbf{\Sigma}).$$
Therefore $X(\mathbf{\Sigma})\cong \overline{\mathcal{R}}_{\SL_2}(\mathbf{\Sigma})$.
 \end{corollary}

\begin{proof} The isomorphism $\Psi_w$ sends the bad arc $\alpha(p)_{-+}$ to the element $(-1)^{w(\alpha(p))}X_{++}^{\alpha(p)}$ so induces $\overline{\Psi}_w$ by passing to the quotient. \end{proof}
 
 \subsection{Smooth loci of moduli spaces}
 
 Let $\Gamma$ be a presenting graph for $\mathbf{\Sigma}$. Since $\Sigma$ retracts to $\Gamma$, we obtain an isomorphism
 $$ \varphi_{\Gamma}: \mathcal{R}_{\SL_2}(\mathbf{\Sigma}) \xrightarrow{\cong} (\SL_2(\mathbb{C}))^{\mathcal{E}(\Gamma)}, \quad \varphi_{\Gamma}: \rho \mapsto (\rho(\beta))_{\beta \in \mathcal{E}(\Gamma)}.$$
 In particular $\mathcal{R}_{\SL_2}(\mathbf{\Sigma})$ is smooth. 
 
 \begin{example}\label{example_presenting_graph} Let $\mathbf{\Sigma}=(\Sigma_{g,n}, \mathcal{A})$ be a connected essential marked surface of genus $g$ and let us define a presenting graph $\Gamma$. 
 Denote by $\mathcal{A}=\{a_0, \ldots, a_{|\mathcal{A}|-1}\}$ the boundary edges, by $\pi_0(\partial \Sigma) = \{\partial_0, \ldots, \partial_{n-1}\}$ the boundary components with $a_0\subset \partial_0$ and write  $v_i:=v_{a_i}$. We order the boundary components such that for $0\leq j < n^{\partial}$ then $\partial_j\in \Gamma^{\partial}$ and for $n^{\partial}\leq j \leq n-1$ then $\partial_j \in \mathring{\mathcal{P}}$.
Let $\overline{\Sigma}$ be the surface obtained from $\Sigma$ by gluing a disc along each boundary component $\partial_i$ for $1\leq i \leq r$, and choose $\lambda_1, \mu_1, \ldots, \lambda_g, \mu_g$ some paths in $\pi_1(\Sigma, v_0)$, such that their images in $\overline{\Sigma}$ generate the free group $\pi_1(\overline{\Sigma}, v_0)$ (said differently, the $\lambda_i$ and $\mu_i$ are longitudes and meridians of $\Sigma$). For each  $p\in \mathring{\mathcal{P}}$ choose a peripheral curve $\gamma_p \in \pi_1(\Sigma, v_0)$ encircling $p$ once. Eventually, for each boundary component $\partial_j$, with $1\leq j <n^{\partial}$, containing a boundary edge $a_{k_j} \subset \partial_j$ chosen arbitrarily,  choose a path $\delta_{\partial_j} : v_0 \rightarrow v_{k_j}$. The set 
$$\mathbb{G}:= \{ \lambda_i, \mu_i, \gamma_p,  \alpha(p_{\partial}), \delta_{\partial_j} | 1\leq i \leq g, p\in \mathring{\mathcal{P}}, p_{\partial} \in \mathcal{P}^{\partial},  1\leq j < n^{\partial}\}$$
is a generating set for $\Pi_1(\Sigma, \mathbb{V})$ (see Figure \ref{fig_generators_final} for an illustration).
 Moreover each of its generators which is of the form $\gamma_p$ or $\alpha(p_{\partial})$ can be expressed as a composition of the other ones and their inverse, therefore a set $\mathcal{E}(\Gamma)$ obtained from $\mathbb{G}$ by removing one of the elements of the form $ \alpha(p_{\partial})$ or $\gamma_p$, forms a set of edges of a presenting graph $\Gamma$.
Note that $\mathbb{G}$ has cardinality $2g-1+|\mathcal{A}|+n$,  so one has an isomorphism $\mathcal{R}_{\SL_2}(\mathbf{\Sigma}) \cong (\SL_2(\mathbb{C}))^{2g-2+|\mathcal{A}|+n}$. 

\begin{figure}[!h] 
\centerline{\includegraphics[width=9cm]{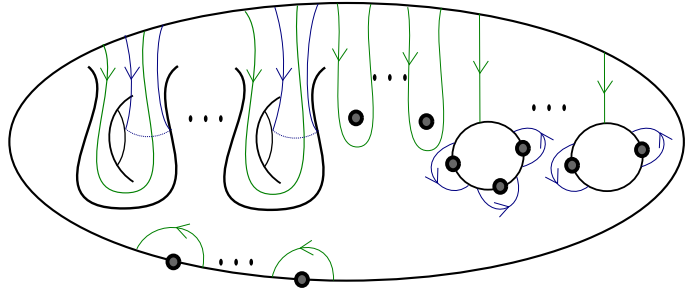} }
\caption{The geometric representatives of a set of generators for $\Pi_1(\Sigma, \mathbb{V})$.} 
\label{fig_generators_final} 
\end{figure} 

\end{example}
 
 \begin{theorem}\label{theorem_smooth} Let $\mathbf{\Sigma}$ be a connected marked surface with at least two boundary arcs. Then $X(\mathbf{\Sigma})$ is a smooth variety.
 \end{theorem}
 
 For $g=\begin{pmatrix} g_{++} & g_{+-} \\ g_{-+} & g_{--} \end{pmatrix} \in \SL_2$, the Zariski tangent space $T_g\SL_2$ is: 
 \begin{multline*}
  T_g\SL_2 = \left\{ X= \begin{pmatrix} X_{++} & X_{+-} \\ X_{-+} & X_{--} \end{pmatrix}, \mbox{ s.t. } \det(g+\varepsilon X)\equiv 1 \pmod{\varepsilon^2} \right\} \\ =  \left\{ X= \begin{pmatrix} X_{++} & X_{+-} \\ X_{-+} & X_{--} \end{pmatrix}, \mbox{ s.t. } g_{--}X_{++} + g_{++}X_{--} = g_{-+}X_{+-} + g_{+-}X_{-+} \right\}. 
  \end{multline*}
 Similarly, if $g\in \SL_2^{1}$ (i.e. if $g_{++}=0$) then 
 $$ T_g\SL_2^{1}= \left\{X= \begin{pmatrix} 0& X_{+-} \\ X_{-+} & X_{--} \end{pmatrix}, \mbox{ s.t. } g_{+-}X_{-+} + g_{-+}X_{+-}=0 \right\}.$$
 
 \begin{lemma}\label{lemma_smooth}
 \begin{enumerate}
 \item Let $X=\SL_2^{1} \times \SL_2$ and consider the regular map $F: X\to \mathbb{C}$, $F(h,g)=(hg)_{--}$. Let $x=(h,g)\in X$. Then there exists $v \in T_h \SL_2^{1} \subset T_h \SL_2^{1} \oplus T_g\SL_2=T_x X$ such that $D_xF(v)\neq 0$. 
 \item 
 Let $Y:= \SL_2^{1} \times (\SL_2)^2$ and consider the regular map $F: Y\to \mathbb{C}$, $F(h,g,f):= (ghg^{-1}f)_{--}$. Let $y=(h,g,f)\in Y$. Then there exists $v \in T_h \SL_2^{1}\oplus T_g\SL_2 \subset T_h \SL_2^{1}\oplus T_g\SL_2\oplus T_f \SL_2=T_yY$ such that $D_yF(v)\neq 0$.
 \end{enumerate}
 \end{lemma}
 
 \begin{proof}
 \par $(1)$ Let $X_1^h, X_2^h \in T_h\SL_2^{1}\subset T_xX$ be the vectors $X_1^h= \begin{pmatrix} 0 & 0 \\ 0 & 1 \end{pmatrix}$, $X_2^h= \begin{pmatrix} 0 & h_{+-} \\ -h_{-+} & 0 \end{pmatrix}$. Then $D_xF(X_1^h)= g_{--}$ and $D_xF(X_2^h)= -g_{+-}$. Since $\det(g)=1$, either $g_{--} \neq 0$ or $g_{+-}\neq 0$.
 \par $(2)$ Consider the vectors $X_1^h, X_2^h, X^g \in T_yY=T_h\SL_2^{1} \oplus T_g\SL_2 \oplus T_f\SL_2$ defined by $X_1^h= \begin{pmatrix} 0 & 0 \\ 0 & 1 \end{pmatrix} \in T_h \SL_2^{1} \subset T_yY$, $X_2^h= \begin{pmatrix} 0 & h_{+-} \\ -h_{-+} & 0 \end{pmatrix} \in T_h \SL_2^{1} \subset T_yY$ and $X^g= \begin{pmatrix} 0 & g_{+-} \\ -g_{-+} & 0 \end{pmatrix} \in T_g \SL_2 \subset T_yY$. 
 One has $D_yF(X_1^h)= g_{--} (g^{-1}f)_{--}$ and $D_yF(X_2^h)= h_{+-}g_{-+}(g^{-1}f)_{--} - h_{-+} g_{--}(g^{-1}f)_{+-}$ so $D_yF(X_1^h)=D_yF(X_2^h)=0$ if and only if $g_{--}=(g^{-1}f)_{--}=0$. In this case, one finds that 
  $$ D_yF(X^g) = -g_{+-}(gh)_{-+}f_{--} - g_{-+}(hg^{-1}f)_{+-} + g_{-+}(gh)_{--}f_{--} = (g_{-+})^2 h_{+-}f_{--}.$$
  Since $g\in \SL_2$, $g_{+-}g_{-+}=-1$ and $g_{-+}\neq 0$. Similarly, $h_{+-}\neq 0$. If $f_{--}=0$ then $f_{+-}\neq 0$ and  the equality $(g^{-1}f)_{--}=0$ would imply $g_{-+}f_{+-}=0$ which is a contradiction. Therefore $f_{--}\neq 0$ and $D_yF(X^g)\neq 0$.
 \end{proof}

 \begin{proof}[Proof of Theorem \ref{theorem_smooth}] Let $g$ be the genus of $\Sigma$ and $n$ the number of its boundary components. Let $\mathbb{G}$ be the generating set of Example \ref{example_presenting_graph}. Let us first suppose  that $\mathring{\mathcal{P}}\neq \emptyset$. In this case, we can remove from $\mathbb{G}$ one generator $\gamma_p$ to obtain a presenting graph $\Gamma$ such that every corner arc is an edge of $\Gamma$. We obtain an isomorphism
 $$ X(\mathbf{\Sigma}) \xrightarrow[\cong]{\Psi_w^*} \overline{\mathcal{R}}_{\SL_2}(\mathbf{\Sigma}) \xrightarrow[\cong]{\varphi_{\Gamma}} (\SL_2(\mathbb{C}))^{2g-2+n} \times (\SL_2^{1})^{|\mathcal{A}|} .$$
 So $ X(\mathbf{\Sigma})$ is smooth and irreducible because both $\SL_2(\mathbb{C})$ and $\SL_2^{1}\cong \mathbb{C}^* \times \mathbb{C}$ are smooth and irreducible. 
 If $\mathring{\mathcal{P}}=\emptyset$, let $p_0 \in \mathcal{P}^{\partial}$, let $\Gamma$ be the presenting graph obtained from $\mathbb{G}$ by removing $\alpha_{p_0}$ and consider the variety
 $$ \mathcal{R}' := \{ \rho : \Pi_1(\Sigma, \mathbb{V}) \to \SL_2, \mbox{ such that } \rho(\alpha(p)) \in \SL_2^{1} \mbox{ for all }p\in \mathcal{P}^{\partial} \setminus \{ p_0\} \}.$$
 The morphism $\varphi_{\Gamma}$ restricts to an isomorphism $\varphi_{\Gamma}: \mathcal{R}' \cong (\SL_2(\mathbb{C}))^{2g+n-1} \times (\SL_2^{1})^{|\mathcal{A}| -1}$, so $\mathcal{R}'$ is a smooth irreducible variety. 
 Let $F: \mathcal{R}' \to \mathbb{C}$ be the regular function $F=X_{++}^{\alpha(p_0)}$, i.e. $F(\rho)=\rho(\alpha(p_0))_{++}$. One has $ \overline{\mathcal{R}}_{\SL_2}(\mathbf{\Sigma})= F^{-1}(0)$ so it suffices to prove that for every $\rho \in \mathcal{R}'$, one has $D_{\rho}F\neq 0$. Using $\varphi_{\Gamma}$ we identify the Zariski tangent space $T_{\rho}\mathcal{R}'$ with 
 $$ T_{\rho}\mathcal{R}' \cong \oplus_{p\in \mathcal{P}^{\partial} \setminus \{p_0\}} T_{\rho(\alpha(p))} \SL_2^{1} \oplus \oplus_{\beta \in \mathbb{G}\setminus\{\alpha(p)\} } T_{\rho(\beta)}\SL_2.$$

\par  First suppose that $\mathbf{\Sigma}$ contains two boundary edges in the same boundary component $\partial$. We suppose that $p_0 \in \partial$ so one has two punctures $p_0,p_1\in \partial$ consecutive in the cyclic order of $\partial$.  
 By definition of $\mathbb{G}$, there exists  a path $\beta$ which is a composition of the elements of $\mathbb{G}\setminus \{\alpha(p_0), \alpha(p_1)\}$ such that $\alpha(p_0)\alpha(p_1)\beta= 1_{s(\alpha(p_0))}$. Therefore
 $$ F(\rho) = \rho( \alpha(p_0))_{++}= \rho( \alpha(p_0)^{-1})_{--} = \left(\rho(\alpha(p_1)) \rho(\beta) \right)_{--}.$$
 By the first item of Lemma \ref{lemma_smooth}, for every $\rho \in \mathcal{R}'$, there exists $v\in T_{\rho(\alpha(p_1))}\SL_2^{1} \subset T_{\rho}\mathcal{R}'$ such that $D_{\rho}F(v)\neq 0$.
 \par If $\mathbf{\Sigma}$ only contains boundary components having a single boundary arc, one can find a boundary component $\partial$ whose unique puncture $p_1\in \partial$ is distinct from $p_0$. By definition of $\mathbb{G}$, there exists  a path $\beta$ which is a composition of the elements of $\mathbb{G}\setminus \{\alpha(p_0), \alpha(p_1), \delta_{\partial}\}$ such that $\alpha(p_0)^{-1}=\delta_{\partial} \alpha(p_1) \delta_{\partial}^{-1} \beta$, so 
 $$F(\rho)= \left( \rho(\delta_{\partial}) \rho(\alpha(p_1)) \rho(\delta_{\partial})^{-1} \rho(\beta) \right)_{--}.$$
 By the second item of Lemma \ref{lemma_smooth},  for every $\rho \in \mathcal{R}'$, there exists $v\in T_{\rho(\alpha(p_1))}\SL_2^{1}\oplus T_{\rho(\delta_{\partial})}\SL_2 \subset T_{\rho}\mathcal{R}'$ such that $D_{\rho}F(v)\neq 0$. 
 In every cases, we have thus proved that $D_{\rho}F\neq 0$ for every $\rho \in \mathcal{R}'$, so $X(\mathbf{\Sigma})$ is smooth.
 \end{proof}
 
 \begin{definition}[Central at inner punctures]
 Let $x\in X(\mathbf{\Sigma})$ and write $\rho_x:= \overline{\Psi}_w^{-1}(x) \in \overline{R}_{\SL_2}(\mathbf{\Sigma})$. For $p\in \mathring{\mathcal{P}}$ an inner puncture, we say that $x$ \textbf{ is central at } $p$ if $\rho_x(\gamma_p)= \pm \id$. 
 \end{definition}
 
 \begin{theorem}\label{theorem_smooth2} Let $\mathbf{\Sigma}$ be an essential marked surface with at least two boundary arcs. Then the regular locus of $\widehat{X}(\mathbf{\Sigma})$ is the subset of elements $\widehat{x}=(x, h_p, h_{\partial})$ such that for every $p\in \mathring{\mathcal{P}}$ then either $x$ is not central at $p$ or $h_p=\pm 2$. 
 \end{theorem}
 
 \begin{lemma}\label{lemma_submersion}
 Consider the application $f: \SL_2\times \mathbb{C} \to \mathbb{C}$, $f : (g,z) \mapsto \tr(g) + T_N(z)$ and set $\widehat{X}:=f^{-1}(0)$. Then $(g,z)\in \widehat{X}$ is smooth if and only if either $g\neq \pm \mathds{1}$ or $g=\pm \mathds{1}$ and 
$z=\mp 2$.
\end{lemma}

\begin{proof}
First, note that  for $z\in \mathbb{C}$ and $z^{1/N}$ an $N$-th root of $z$, the equation $T_N(X)=z+z^{-1}$ has solutions $X= z^{n/N} + z^{-n/N}$ for $n\in \{0, \ldots, N-1\}$ counted with multiplicity, i.e. $T_N(X)-(z+z^{-1})= \prod_{n=0}^{N-1} (X- z^{n/N}-z^{-n/N})$. Therefore, for $c\in \mathbb{C}$, the polynomial $P_{N,c}(X):= T_N(X)-c$ has roots of multiplicity $\geq 2$ if and only if $c=\pm 2$ in which case the multiple roots are $X= \pm (q^n+q^{-n})$ with   $n\in \{1, \ldots, (N-3)/2\}$.
\par For $g=\begin{pmatrix} g_{++} & g_{+-} \\ g_{-+} & g_{--} \end{pmatrix} \in \SL_2$, let us consider the tangent space
$$T_g\SL_2=   \left\{ X= \begin{pmatrix} X_{++} & X_{+-} \\ X_{-+} & X_{--} \end{pmatrix}, \mbox{ s.t. } g_{--}X_{++} + g_{++}X_{--} = g_{-+}X_{+-} + g_{+-}X_{-+} \right\}$$
as before. Then the differential at $(g,z)\in \widehat{X}$ is computed as 
$$ D_{(g,z)}f : T_g\SL_2 \oplus \mathbb{C} \to \mathbb{C}, \quad D_{(g,z)}f: (X, z_0) \mapsto \tr(X) + T_N'(z)z_0.$$
In order to find at which point $f$ is a submersion, we need to find when this differential vanishes. This will happen if and only if we have the two conditions: $(i)$ $\tr(X)=0$ for all $X\in T_g\SL_2$ and $(ii)$ $T_N'(z)=0$. By the preceding discussion, $T_N'(z)=0$ if and only if $\tr(g)=\pm 2$ and $z\neq \mp 2$. Let us analyse $(i)$.  
First suppose that $g_{+-}\neq 0$. Then for every $x\in \mathbb{C}$ we have $X:= \begin{pmatrix} x & 0 \\ x \frac{g_{--}}{g_{+-}} & 0 \end{pmatrix} \in T_g\SL_2$ and $\tr(X)=x$ so $(i)$ is not satisfied. Similarly, if $g_{-+}\neq 0$, we can use $X:= \begin{pmatrix} x & x \frac{g_{--}}{g_{-+}}  \\ 0& 0 \end{pmatrix} \in T_g\SL_2$ and $\tr(X)=x$ so $(i)$ is not satisfied. If $g=\begin{pmatrix} x & 0 \\ 0 & x^{-1} \end{pmatrix}$ then every $X\in T_g\SL_2$ satisfies $\tr(X)=(1-x^{-2})X_{++}$ so $(i)$ is satisfied if and only if $g= \pm \mathds{1}$ in which case $(ii)$ is satisfied if and only if $z=\mp (q^n +q^{-n})$ for $n\in \{1, \ldots, (N-3)/2\}$. This concludes the proof.
\end{proof}

 \begin{proof}[Proof of Theorem \ref{theorem_smooth2}]

Write $\mathring{\mathcal{P}}=\{p_1, \ldots, p_s\}$ and consider the composition 
$$ g : \overline{R}_{\SL_2}(\mathbf{\Sigma}) \times \mathbb{C}^{s} \xrightarrow{h} (\SL_2\times \mathbb{C})^{s} \xrightarrow{ (f, \ldots, f)} \mathbb{C}^s$$
where $g(\rho, \{h_{p_i}\}_i) := (\rho(\gamma_{p_i}), h_{p_i})_{p_i}$ and $f$ is defined in Lemma \ref{lemma_submersion}. Let $X^0(\mathbf{\Sigma}):=g^{-1}(0, \ldots, 0)$ and consider the covering $\pi: \widehat{X}(\mathbf{\Sigma})\to X^0(\mathbf{\Sigma})$ defined by $\pi(x, h_{p_i}, h_{\partial}):=(\rho_x, h_{p_i})$ which consists in forgetting the boundary invariants $h_{\partial}$. For $(\rho_x, h_p)\in X^0(\mathbf{\Sigma})$, set $c_{\partial}:= \rho_x(\alpha_{\partial})$. The fiber $\pi^{-1}(\rho_x, h_{p_i})$ consists in the set of elements $(x, h_{p_i}, h_{\partial})$ where 
 the $h_{\partial}$ satisfy $h_{\partial}^N = c_{\partial}$ so are zeroes of the polynomials $X^N-c_{\partial}$. Since the polynomials $X^N-c_{\partial}$ have $N$ zeroes each with multiplicity one, $\pi$ is a regular covering so the smooth locus of $\widehat{X}(\mathbf{\Sigma})$ is the pull-back by $\pi$ of the smooth locus of $X^0(\mathbf{\Sigma})$. 
 \par  Since  $h$ is a submersion on $X^0(\mathbf{\Sigma})$, by the chain rule $g$ is a submersion at $(\rho, \{h_{p_i}\}_i)$ if and only if for every $i$ $f$ is a submersion at $(\rho(\gamma_{p_i}), h_{p_i})$. By Lemma \ref{lemma_submersion}, this condition is equivalent to the condition that for every $p\in \mathring{\mathcal{P}}$, either $\rho(\gamma_p)\neq \pm \mathds{1}$ or $h_p=\mp 2$.  This concludes the proof.
 
 \end{proof}

 \subsection{Behavior for the gluing operation}
 
 The goal of this subsection is to prove the 
 
 \begin{theorem}\label{theorem_gluing_surjective}
 Let  $\mathbf{\Sigma}_1$ and $\mathbf{\Sigma}_2$ be connected marked surfaces and for $i=1,2$ let $a_i$ be a boundary edges of $\mathbf{\Sigma}_i$ such that the connected component of $\partial \Sigma_i$ which contains $a_i$ also contains at least another boundary edge. 
 Then the maps $p_{a_1\# a_2}: {X}(\mathbf{\Sigma}_1)\times  {X}(\mathbf{\Sigma}_2) \to {X}(\mathbf{\Sigma}_1\cup_{a_1\# a_2} \mathbf{\Sigma}_2)$ and $\widehat{p}_{a_1\# a_2}: \widehat{X}(\mathbf{\Sigma}_1)\times  \widehat{X}(\mathbf{\Sigma}_2) \to \widehat{X}(\mathbf{\Sigma}_1\cup_{a_1\# a_2} \mathbf{\Sigma}_2)$ are surjective.
 \end{theorem}
 
 To prove Theorem \ref{theorem_gluing_surjective}, it is useful to introduce an alternative geometric interpretation of $X(\mathbf{\Sigma})$ defined in \cite{CostantinoLe19}. Let $\mathbf{\Sigma}$ be a connected essential surface and fix a Riemannian metric on $\Sigma$. Let $\pi: U\Sigma \to \Sigma$ be the unitary tangent bundle, so a point $\overrightarrow{v}\in U\Sigma$ is given by a pair $\overrightarrow{v}=(p,v)$ where $p\in \Sigma$ and $v\in T_p\Sigma$ has norm one.
 For such a  $\overrightarrow{v}=(p,v)$, write $-\overrightarrow{v}:= (p, -v)$. 
 \par Recall that the orientation of $\Sigma$ induces an orientation of $\partial \Sigma$.
    For $p \in \partial \Sigma$, let $\overrightarrow{p} \in U\Sigma$ be $\overrightarrow{p}:=(p,v)$ where $v\in T_p \partial \Sigma \subset T_p\Sigma$ is the unitary tangent vector oriented in the direction of $\partial \Sigma$. Said differently, if $n\in T_p\Sigma$ is a vector normal to $\partial \Sigma$ which points outside of $\Sigma$, then $(n,v)$ is a positive basis of $T_p\Sigma$. Write $\widehat{\mathbb{V}} := \{ \pm \overrightarrow{v_a}, a\in \mathcal{A} \}$ (recall that $v_a$ is a point in the middle of the boundary edge $a$). 
 \par For $\overrightarrow{v} \in U\Sigma$ with $p=\pi(\overrightarrow{v})$,  let $\theta_{\overrightarrow{v}} \in \pi_1(U\Sigma, \overrightarrow{v})$ be the homotopy class of the path in $U\Sigma$ which lies inside the circle fiber $\pi^{-1}(p)=T_p\Sigma$ and makes a single positive turn. Also denote by $\theta_{\overrightarrow{v}}^{1/2}$ be the homotopy class of the path starting at $\overrightarrow{v}$ and ending at $-\overrightarrow{v}$ which makes a positive half turn in the fiber $T_p\Sigma$. 
 \par Let $\alpha$ be an oriented arc of $\mathbf{\Sigma}$ with endpoints in some boundary edges $a$ and $b$ such that $\alpha$ is oriented from $a$ to $b$. Let $c_{\alpha}: [0,1]\to \Sigma$ be an oriented embedding such that $(1)$ $c_{\alpha}([0,1])$ is isotopic to $\alpha$ with the same orientation, $(2)$ $c_{\alpha}'(0)= - \overrightarrow{v_a}$ and $c_{\alpha}'(1)=\overrightarrow{v_b}$ and $(3)$ $c_{\alpha}'(t)$ has norm $1$ for all $t\in [0,1]$. We denote by $\widehat{\alpha}$ the homotopy class of the path $[0,1] \to U\Sigma$, $t\mapsto (c_{\alpha}(t), c'_{\alpha}(t))$. 
 
 \begin{definition}[Twisted relative representation varieties]
 \begin{enumerate}
 \item The \textbf{twisted relative representation variety} $\mathcal{R}^{tw}_{\SL_2}(\mathbf{\Sigma})$  is the set of functors $\widehat{\rho}: \Pi_1(U\Sigma, \widehat{\mathbb{V}}) \to \SL_2$ such that $\widehat{\rho}(\theta^{1/2}_{\overrightarrow{v}})= \begin{pmatrix} 0 & -1\\ 1 & 0 \end{pmatrix}$ for all $\overrightarrow{v} \in \widehat{\mathbb{V}}$. It admits a natural structure of affine variety as in Definition \ref{def_modulispaces}
 \item The  \textbf{reduced twisted relative representation variety} $\overline{\mathcal{R}}^{tw}_{\SL_2}(\mathbf{\Sigma})$ is the subvariety of  $\mathcal{R}^{tw}_{\SL_2}(\mathbf{\Sigma})$  of functors $\widehat{\rho}$ such that $\widehat{\rho}(\widehat{\alpha(p)})$ is upper triangular for all $p \in \mathcal{P}^{\partial}$.
 \end{enumerate}
 \end{definition}
 
 \begin{theorem}(\cite[Theorem $8.12$]{CostantinoLe19}) There exists an isomorphism $\Psi_{\mathbf{\Sigma}} : X(\mathbf{\Sigma}) \xrightarrow{\cong} {\mathcal{R}}^{tw}_{\SL_2}(\mathbf{\Sigma})$ sending a character $\chi : \mathcal{S}_{+1}(\mathbf{\Sigma}) \to \mathbb{C}$ to the functor $\widehat{\rho}: \Pi_1(U\Sigma, \widehat{\mathbb{V}}) \to \SL_2$ characterized by the fact that for every arc $\alpha$, then 
 $$ \widehat{\rho}(\widehat{\alpha}) = \begin{pmatrix}\chi(\alpha_{++}) & \chi(\alpha_{+-}) \\ \chi(\alpha_{-+}) & \chi(\alpha_{--})\end{pmatrix}.$$
 \end{theorem}
 
 In particular, the ring isomorphism $\Psi_{\mathbf{\Sigma}}^{*}$ sends a bad arc $\alpha(p)_{-+}$ to the function $X_{-+}^{\widehat{\alpha(p)}}$ which sends $\widehat{\rho}$ to the lower left matrix coefficient of $\widehat{\rho}(\widehat{\alpha(p)})$, so we get the 
 
 \begin{corollary}
  The isomorphism $\Psi_{\mathbf{\Sigma}}$ restricts to an isomorphism (still denoted by the same letter) $\Psi_{\mathbf{\Sigma}}: \overline{X}(\mathbf{\Sigma}) \xrightarrow{\cong} \overline{\mathcal{R}}^{tw}_{\SL_2}(\mathbf{\Sigma})$.
  \end{corollary}


Now consider $\mathbf{\Sigma}_1$ and $\mathbf{\Sigma}_2$ with $a_1$ and $a_2$ as in Theorem \ref{theorem_gluing_surjective} and write $\mathbf{\Sigma}:=  \mathbf{\Sigma}_1\cup_{a_1\# a_2} \mathbf{\Sigma}_2$.
For $i=1,2$, let $\iota_i : \Sigma_i \hookrightarrow \Sigma$ be the inclusion map.
 Fix a Riemannian metric on $\Sigma$ which induces, by restriction, some Riemannian metrics on $\Sigma_1$ and $\Sigma_2$. Let $\overrightarrow{w} \in U\Sigma$ be the image of $\overrightarrow{v_{a_1}}\in U\Sigma_1 \xrightarrow{\subset} U\Sigma$ by $(\iota_1)_*$. Note that the image of $\overrightarrow{v_{a_2}}\in U\Sigma_2$ in $U\Sigma$ by $(\iota_2)_*$ is $-\overrightarrow{v}$. 
 Write $\Pi_1:= \Pi_1(U\Sigma_1, \mathbb{V}_1)$, $\Pi_2:= \Pi_1(U\Sigma_2, \mathbb{V}_2)$ and $\Pi:= \Pi_1(U\Sigma, \mathbb{V} \cup \{\overrightarrow{w}\})$ and consider the variety: 
 $$ \mathcal{R}^{w}:= \{ \widehat{\rho}\in \Hom(\Pi, \SL_2), \widehat{\rho}(\theta^{1/2}_{\overrightarrow{v}})= \begin{pmatrix} 0 & -1\\ 1 & 0 \end{pmatrix} \mbox{, for all }\overrightarrow{v}\in \widehat{\mathbb{V}} \}.$$
 
 For $i=1,2$, the embedding $\iota_i : \Sigma_i \hookrightarrow \Sigma$ induces a fully faithful functor $(\iota_i)_*: \Pi_i \hookrightarrow \Pi$ so we can consider $\Pi_i$ as a full subcategory of $\Pi$. We are in the following situation: 
 
 \begin{lemma}\label{lemma_groupoid} Let $G$ be a groupoid and $G_1, G_2\subset G$ be two full subcategories of $G$ and $v\in G$ an object such that 
 
 \begin{enumerate}
 \item one has $\Ob(G_1)\cup \Ob(G_2)= \Ob(G)$ and $\Ob(G_1)\cap \Ob(G_2)=\{ v\}$; 
 \item for $v_1\in G_1$ and $v_2\in G_2$, the composition map $G_1(v_1, v)\times G_2(v, v_2) \to G(v_1,v_2)$ is a bijection.
 \end{enumerate}
 Then the restriction map $\res: \Hom(G, \SL_2) \to \Hom(G_1, \SL_2) \times \Hom(G_2, \SL_2)$ sending $\rho$ to the pair $(\restriction{\rho}{G_1}, \restriction{\rho}{G_2})$, is a bijection.
 \end{lemma}

 \begin{proof}
Consider the map  $$F: \Hom(G_1, \SL_2) \times \Hom(G_2, \SL_2) \to \Hom(G, \SL_2)$$ which sends a pair $(\rho_1, \rho_2)$ to the functor $\rho$ such that for $v_1, v_2 \in G$, the restriction $\rho: G(v_1,v_2) \to \SL_2(\mathbb{C})$ is equal to $(i)$ the restriction $\rho_1: G_1(v_1,v_2) \to \SL_2(\mathbb{C})$ if $v_1, v_2 \in \Pi_1$, $(ii)$ to the restriction $\rho_2: G_2(v_1,v_2) \to \SL_2(\mathbb{C})$ if $v_1,v_2\in G_2$, $(iii)$ to the composition $G(v_1,v_2) \cong G_1(v_1, v) \times G_2(v, v_2) \xrightarrow{(\rho_1, \rho_2)}\SL_2(\mathbb{C}) \times \SL_2(\mathbb{C}) \xrightarrow{ \times} \SL_2(\mathbb{C})$ if $v_1\in G_1$, $v_2\in G_2$, $(iv)$ to a similar composition map with $1$ and $2$ exchanged if $v_1\in G_2$ and $v_2\in G_1$. $F$ is clearly the inverse of  $\res$.
\end{proof}
The morphism $F$ of the proof, obtained by replacing $(G,G_1,G_2, v) $  by $(\Pi, \Pi_1, \Pi_2, \overrightarrow{w})$,   induces by restriction an isomorphism (still denoted by the same letter): 
 $$ F :\mathcal{R}^{tw}_{\SL_2}(\mathbf{\Sigma}_1) \times \mathcal{R}^{tw}_{\SL_2}(\mathbf{\Sigma}_2) \xrightarrow{\cong} \mathcal{R}^w.$$
The inclusion $\Pi_1(U\Sigma, \widehat{\mathbb{V}}) \subset \Pi_1(U\Sigma, \widehat{\mathbb{V}} \cup \{\overrightarrow{w}\})$ induces a regular (restriction) map $G: \mathcal{R}^w \to \mathcal{R}^{tw}_{\SL_2}(\mathbf{\Sigma})$.

\begin{lemma}\label{lemma_surjectivity_modulispaces}
The composition $G\circ F : \mathcal{R}^{tw}_{\SL_2}(\mathbf{\Sigma}_1) \times \mathcal{R}^{tw}_{\SL_2}(\mathbf{\Sigma}_2) \to  \mathcal{R}^{tw}_{\SL_2}(\mathbf{\Sigma})$ is surjective and induces by restriction a surjective map 
 $\varphi: \overline{\mathcal{R}}^{tw}_{\SL_2}(\mathbf{\Sigma}_1) \times \overline{\mathcal{R}}^{tw}_{\SL_2}(\mathbf{\Sigma}_2) \to  \overline{\mathcal{R}}^{tw}_{\SL_2}(\mathbf{\Sigma})$.
 \end{lemma}
 
 In order to prove Lemma \ref{lemma_surjectivity_modulispaces}, we first prove the following:
 
 \begin{lemma}\label{lemma_prout} Let $B^+\subset \SL_2(\mathbb{C})$ be the subgroup of upper triangular matrices and let $A,B,C,D \in \SL_2(\mathbb{C})$ be such that $AB \in B^+$ and $CD\in B^+$. Then there exists $g \in \SL_2(\mathbb{C})$ such that $Ag$, $g^{-1}B$, $Dg^{-1}$ and $gC$ are in $B^+$.
 \end{lemma}
 
 \begin{proof} Write $A=\begin{pmatrix} a_{++} & a_{+-} \\ a_{-+} & a_{--} \end{pmatrix}$ with similar notations for $B,C,D,g$. Then 
 $$ \left\{ \begin{array}{l} Ag \in B^+ \\ g^{-1}B \in B^+ \end{array} \right. \Leftrightarrow \begin{pmatrix} a_{-+} & a_{--} \\ b_{-+} & -b_{++} \end{pmatrix} \begin{pmatrix} g_{++} \\ g_{-+} \end{pmatrix} = \begin{pmatrix} 0 \\ 0 \end{pmatrix}.$$
 Since $AB \in B^+ \Leftrightarrow \det \begin{pmatrix} a_{-+} & a_{--} \\ b_{-+} & -b_{++} \end{pmatrix} =0$, one can find $(g_{++}, g_{-+})\neq (0,0)$ which satisfies this equation. Similarly, 
 $$ \left\{ \begin{array}{l} Dg^{-1} \in B^+ \\ gC \in B^+ \end{array} \right. \Leftrightarrow \begin{pmatrix} c_{++} & c_{-+} \\ -d_{--} & d_{-+} \end{pmatrix} \begin{pmatrix} g_{-+} \\ g_{--} \end{pmatrix} = \begin{pmatrix} 0 \\ 0 \end{pmatrix}.$$
 Since $CD\in B^+ \Leftrightarrow \det  \begin{pmatrix} c_{++} & c_{-+} \\ -d_{--} & d_{-+} \end{pmatrix} =0$, one can find $(g_{-+}, g_{--})\neq (0,0)$ which satisfies this equation. Up to multiplying $(g_{-+}, g_{--})$ by a non zero scalar, one can further suppose that $g_{++}g_{--}-g_{+-}g_{-+}=1$, so the matrix $g= \begin{pmatrix} g_{++} & g_{+-} \\ g_{-+} & g_{--} \end{pmatrix}$ satisfies the conclusion of the lemma.

 \end{proof}
 
 \begin{proof}[Proof of Lemma \ref{lemma_surjectivity_modulispaces}] Since $F$ is a bijection, we need to prove that $G$ is surjective. Let us first define an action of $\SL_2(\mathbb{C})$ on $\mathcal{R}^w$ as follows. For $g\in \SL_2(\mathbb{C})$ and $\rho \in \mathcal{R}^w$, we define $g\bullet \rho \in \mathcal{R}^w$ as the functor which sends a path $\beta: v_1 \to v_2$ in  $\Pi_1(U\Sigma, \widehat{\mathbb{V}} \cup \{ \overrightarrow{w}\})$ to 
 $$ (g\bullet \rho) (\beta) := \left\{ \begin{array}{ll}
 \rho(\beta) & \mbox{, if }v_1 \neq \overrightarrow{w}, v_2 \neq \overrightarrow{w}; \\
 \rho(\beta)g & \mbox{, if }v_1 \neq \overrightarrow{w}, v_2 = \overrightarrow{w}; \\
  g^{-1}\rho(\beta) & \mbox{, if }v_1 = \overrightarrow{w}, v_2 \neq \overrightarrow{w}; \\
  g^{-1}\rho(\beta)g & \mbox{, if }v_1 = v_2 = \overrightarrow{w}.
  \end{array} \right. $$
  Clearly $G(g\bullet \rho) = G(\rho)$ so $G$ induces a map $G': \quotient{ \mathcal{R}^w}{\SL_2(\mathbb{C})} \to \mathcal{R}^{tw}_{\SL_2}(\mathbf{\Sigma})$. Let us prove that $G'$ is surjective (it is not difficult to prove that it is an isomorphism).
\par  Let $c$ be a boundary edge of $\mathbf{\Sigma}$ (which exists by the  hypotheses made in Theorem \ref{theorem_gluing_surjective}) and consider an arbitrary path $\beta: \overrightarrow{v_c} \to \overrightarrow{w}$. Let $\rho \in  \mathcal{R}^{tw}_{\SL_2}(\mathbf{\Sigma})$ and define an extension $\widehat{\rho} \in \mathcal{R}^w$ as follows. First, the restriction of $\widehat{\rho}$ to $\Pi_1(U\Sigma, \widehat{\mathbb{V}})$ is $\rho$ and $\widehat{\rho}(\beta)= \mathds{1}_2$. Next for a path $\gamma: \overrightarrow{w} \to \overrightarrow{v}$ with $\overrightarrow{v}\in \widehat{V}$, set $\widehat{\rho}(\gamma):= \rho(\gamma \beta^{-1})$. Similarly for a path $\gamma': \overrightarrow{v} \to \overrightarrow{w}$ set $\widehat{\rho}(\gamma'):= \rho(\beta^{-1}\gamma' )$. Clearly  $\widehat{\rho} \in \mathcal{R}^w$ and $G(\widehat{\rho})=\rho$ so $G$ and $G'$  are surjective. 
 \par Let us prove that the restriction $\varphi$ of $G\circ F$ is surjective as well. Let $p_1, p_1'$ be the two punctures adjacent to $a_1$ and $p_2, p_2'$ the two punctures adjacent to $a_2$ such that while gluing $a_1$ with $a_2$, $p_1$ and $p_2$ have the same image $p\in \partial \Sigma$ and $p_1', p_2'$ have the same image $p'\in \partial \Sigma$ (see Figure \ref{fig_gluing_surjective}). By assumption, $p_1\neq p_1'$ and $p_2\neq p_2'$. 
 Let $\rho \in \overline{\mathcal{R}}^{tw}_{\SL_2}(\mathbf{\Sigma})$. By surjectivity of $G$, there exists $\widehat{\rho} \in \mathcal{R}^w$ such that $G(\widehat{\rho})=\rho$. Then $\widehat{\rho}$ is the image by $F$ of an element of 
$ \overline{\mathcal{R}}^{tw}_{\SL_2}(\mathbf{\Sigma}_1) \times \overline{\mathcal{R}}^{tw}_{\SL_2}(\mathbf{\Sigma}_2) $ if and only if for all $p\in \mathcal{P}^{\partial_1} \cup \mathcal{P}^{\partial_2}$, then $\widehat{\rho}(\widehat{\alpha(p)}) \in B^+$. Since  $G(\widehat{\rho})=\rho \in  \overline{\mathcal{R}}^{tw}_{\SL_2}(\mathbf{\Sigma})$, $\widehat{\rho}(\widehat{\alpha(p)}) \in B^+$ for all $p\notin \{p_1,p_1', p_2, p_2'\}$. Write $A:= \widehat{\rho}(\widehat{\alpha(p_1)})$, $B:=  \widehat{\rho}(\widehat{\alpha(p_2)})$, $C:=  \widehat{\rho}(\widehat{\alpha(p_2')})$, $D:=  \widehat{\rho}(\widehat{\alpha(p_1')})$. Then $AB= \rho (\widehat{\alpha(p)}) \in B^+$ and $CD=\rho (\widehat{\alpha(p')})\in B^+$ so by Lemma \ref{lemma_prout}, there exists $g\in \SL_2(\mathbb{C})$ such that  $Ag$, $g^{-1}B$, $Dg^{-1}$ and $gC$ are in $B^+$. Therefore, writing  $\rho':= g\bullet \widehat{\rho}$ one has $\restriction{\rho'}{\Pi_1}\in \overline{\mathcal{R}}^{tw}_{\SL_2}(\mathbf{\Sigma}_1)$, $\restriction{\rho'}{\Pi_2} \in \overline{\mathcal{R}}^{tw}_{\SL_2}(\mathbf{\Sigma}_2)$ and $\varphi(\restriction{\rho'}{\Pi_1}, \restriction{\rho'}{\Pi_2)} = \rho$. So $\varphi$ is surjective.

 \end{proof}
 
 \begin{proof}[Proof of Theorem \ref{theorem_gluing_surjective}] 
 The surjectivity of $p_{a_1\# a_2}$ follows from the commutativity of the following diagram
 $$ \begin{tikzcd}
 X(\mathbf{\Sigma}_1) \times X(\mathbf{\Sigma}_2) 
 \ar[r, "p_{a_1\# a_2}"] \ar[d, "\Psi_{\mathbf{\Sigma}_1}\times \Psi_{\mathbf{\Sigma}_2}", "\cong"'] &
 X(\mathbf{\Sigma}) \ar[d, "\Psi_{\mathbf{\Sigma}}", "\cong"'] \\
 \overline{\mathcal{R}}^{tw}_{\SL_2}(\mathbf{\Sigma}_1) \times \overline{\mathcal{R}}^{tw}_{\SL_2}(\mathbf{\Sigma}_2) 
 \ar[r, "\varphi"] & 
  \overline{\mathcal{R}}^{tw}_{\SL_2}(\mathbf{\Sigma}) 
 \end{tikzcd} $$
 together with the surjectivity of $\varphi$ proved in Lemma \ref{lemma_surjectivity_modulispaces}. Let us prove that  $\widehat{p}_{a_1\# a_2}$ is surjective as well. Let $\widehat{x}=({x}, h_p, h_{\partial})_{p, \partial} \in \widehat{X}(\mathbf{\Sigma})$ and fix $(x_1,x_2) \in X(\mathbf{\Sigma}_1)\times X(\mathbf{\Sigma}_2)$ such that $p_{a_1\# a_2}(x_1,x_2)=x$. Let $\partial_1$ be the boundary component of $\Sigma_1$ containing $a_1$ and $\partial_2$ the boundary component of $\Sigma_2$ containing $a_2$ and $\partial'$ the boundary component of $\Sigma$ obtained by $1$-surgery on $\partial_1\cup \partial_2$. By definition $h_{\partial'}^N=\chi_x(\alpha_{\partial})$. Since $\theta_{a_1\# a_2}(\alpha_{\partial_1}\otimes \alpha_{\partial_2})=\alpha_{\partial}$, one can find (and do fix) $h_{\partial_1}, h_{\partial_2} \in \mathbb{C}^*$ such that $h_{\partial}=h_{\partial_1}h_{\partial_2}$ and $h_{\partial_1}^N=\chi_{x_1}(\alpha_{\partial_1})$, $h_{\partial_2}^N= \chi_{x_2}(\alpha_{\partial_2})$. Define $(\widehat{x}_1, \widehat{x}_2)\in \widehat{X}(\mathbf{\Sigma}_1)\times \widehat{X}(\mathbf{\Sigma}_2)$ by, for $i=1,2$,  $\widehat{x}_i= (x_i, h_p, h_{\partial})_{p \in \mathring{\mathcal{P}}_i, \partial \in \Gamma_i^{\partial}}$ where for $p\in \mathring{\mathcal{P}}_i  \ \subset \mathring{\mathcal{P}}$, $h_p$ is the puncture invariant of $\widehat{x}$, for $\partial \in \Gamma_i^{\partial} \setminus \{ \partial_i\}$ then $h_{\partial}$ is the boundary invariant of $\widehat{x}$ and $h_{\partial_i}$ is the previously chosen scalar. Then $\widehat{p}_{a_1\# a_2}(\widehat{x}_1, \widehat{x}_2)= \widehat{x}$, thus $\widehat{p}_{a_1\# a_2}$ is surjective.

 \end{proof}

 \section{Representations of $ \overline{\mathcal{S}}_{A}(\mathbb{D}_1)$}\label{sec_D1}
 
 Let $\mathbb{D}_1=(\Sigma_{0,2}, \{a_L, a_R\})$ be an annulus with two boundary edges in the same boundary component (i.e. a punctured bigon). It admits a presenting graph with two edges $\alpha, \beta$ represented in Figure \ref{fig_puncturedBigon} (note that $\alpha$ and $\beta^{-1}$ are corner arcs). The algebra $ \overline{\mathcal{S}}_{A}(\mathbb{D}_1)$ is isomorphic to the Drinfeld double of the quantum Borel algebra $\Dq$ described as follows.

 \begin{figure}[!h] 
\centerline{\includegraphics[width=2cm]{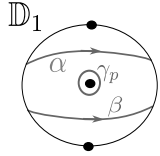} }
\caption{Two oriented arcs $\alpha$ and $\beta$ in $\mathbb{D}_1$.} 
\label{fig_puncturedBigon} 
\end{figure}

 \begin{definition}[Quantum enveloping algebra]
 Write $q:= A^2$. The algebra $\Dq$ is presented by the generators $K^{\pm 1/2}, L^{\pm 1/2}, E$ and $F$ together with the following relations:

\begin{align*}
&EK^{1/2}= q^{-1}K^{1/2}E; \quad EL^{1/2}= qL^{1/2} E; \quad  FK^{1/2}=qK^{1/2}F; \quad FL^{1/2}= q^{-1}L^{1/2}F; \\
&xy=yx \mbox{, for all }x,y\in \{K^{\pm 1/2}, L^{\pm 1/2} \}; \quad K^{1/2}K^{-1/2}=L^{1/2}L^{-1/2}= 1; \\
&EF-FE= \frac{K-L}{q-q^{-1}}. 
\end{align*}
 \end{definition}

 \begin{theorem}\label{theorem_skein_QG}(\cite{KojuQGroupsBraidings})
 There is an isomorphism of algebras $\Psi:  \Dq \xrightarrow{\cong}  \overline{\mathcal{S}}_{A}(\mathbb{D}_1)$ defined by 
 \begin{align*}
 {} & \Psi(K^{1/2})= \alpha_{--}, \quad \Psi(K^{-1/2})= \alpha_{++}, \quad \Psi(L^{1/2})= \beta_{--}, \quad \Psi(L^{-1/2})=\beta_{++} \\
 {} & \Psi(E)=\frac{-A}{q-q^{-1}} \alpha_{+-}\alpha_{--}, \quad \Psi(F)= \frac{A^{-1}}{q-q^{-1}}\beta_{--} \beta_{-+}
 \end{align*}
 \end{theorem}
 
 Note that $H_{\partial}:=K^{-1/2}L^{-1/2}$ is a central element such that $\Psi(H_{\partial})=h_{\partial}$ for $\partial$ the connected component containing $a_L$ and $a_R$. Also, if $p$ denotes the only inner puncture of $\mathbb{D}_1$, then 
 \begin{equation}\label{eq_Casimir}
 C:= - \frac{\Psi^{-1}(\gamma_p)H_{\partial}^{-1}}{(q-q^{-1})^2} = EF+\frac{qL +q^{-1}K}{(q-q^{-1})^2} = FE+ \frac{qK+q^{-1}L}{(q-q^{-1})^2}
 \end{equation}
 is the Casimir element of $\Dq$. Let $\widetilde{U}_q \mathfrak{sl}_2:= \quotient{\Dq}{(H_{\partial} -1)}$ and $U_q\mathfrak{sl}_2 \subset \widetilde{U}_q \mathfrak{sl}_2$ the subalgebra generated by $K^{\pm 1}, E$ and $F$. The simple modules of  $U_q\mathfrak{sl}_2$ have been classified in \cite{DeConciniKacRepQGroups, ArnaudonRoche} (see also  \cite[Chapter $1$ Section $3.3$]{KlimykSchmudgenQGroups}) whereas many indecomposable semi-weight $U_q\mathfrak{sl}_2$ modules were found and studied in \cite{Arnaudon_Uqsl2Rep}. We now provide a classification of the indecomposable (semi) weight $\Dq$ modules as follows.
 
 \begin{definition}[Weight representations of $\Dq$]\label{def_QGRep}
 We define three families of indecomposable weight $\Dq$-modules as follows. For $n\in \mathbb{Z}$, we write $[n]:=\frac{q^n-q^{-n}}{q-q^{-1}}$. 
\begin{enumerate}
\item For $\lambda, \mu\in \mathbb{C}^*$ and $a,b\in \mathbb{C}$, the module $V(\lambda, \mu, a,b)$ has canonical basis $(v_0, \ldots, v_{N-1})$ and module structure given by:
\begin{align*}
& K^{1/2}v_i = \lambda q^{-i}v_i; \quad L^{1/2}v_i = \mu q^i v_i \mbox{, for }i\in \{0, \ldots, N-1\}; \\
& Fv_{N-1}=bv_0; \quad Fv_i=v_{i+1} \mbox{, for }i \in \{0, \ldots, N-2\}; \\
& Ev_0=a v_{N-1}; \quad Ev_{i}= \left( \frac{q^{-i+1}\lambda^2 - q^{i-1}\mu^2}{q-q^{-1}} [i] +ab \right) v_{i-1}\mbox{, for }i\in \{1, \ldots, N-1\}.
\end{align*}

\item For $\lambda, \mu \in \mathbb{C}^*$, $c\in \mathbb{C}$, the module $\widetilde{V}(\lambda, \mu, c)$ has canonical basis $(w_0, \ldots, w_{N-1})$  and module structure given by: 
\begin{align*}
& K^{1/2}w_i= \lambda q^{i}w_i; \quad L^{1/2}w_i= \mu q^{-i}w_i \mbox{, for }i \in \{0, \ldots, N-1\}; \\
& Ew_{N-1}= cw_0; \quad Ew_i= w_{i+1} \mbox{, for }i\in \{0, \ldots, N-2\}; \\
& F w_0= 0; \quad Fw_{i} = \left( \frac{ \mu^2q^{-i+1} - \lambda^2 q^{i-1} }{q-q^{-1}}[i] \right) w_i \mbox{, for }i\in \{1, \ldots, N-1\}.
\end{align*}

\item For $\mu \in \mathbb{C}^*$, $\varepsilon \in \{-1,+1\}$, $0\leq n \leq N-1$, the module $S_{\mu, \varepsilon, n}$ has canonical basis $(e_0, \ldots, e_n)$ and module structure given by: 
\begin{align*}
&K^{1/2}e_i = \varepsilon \mu A^{n-2i} e_i; \quad L^{1/2}e_i= \mu A^{2i-n}e_i \mbox{, for }i\in \{0, \ldots, n\}; \\
&Fe_n=0; \quad Fe_i = e_{i+1} \mbox{, for }i\in \{0,\ldots, n-1\}; \\
&Ee_0=0; \quad Ee_i= \mu^2 [i][n-i+1]e_{i-1}\mbox{, for }i\in \{1,\ldots, n\}.
\end{align*}

\end{enumerate}
 \end{definition}
 
 \begin{definition}[Standard semi-weight representations of $\Dq$] \label{def_semiweight_modules} We define three families of indecomposable semi-weight representations as follows.
 \begin{enumerate}
 \item For $\lambda, \mu, b \in \mathbb{C}^*$ and $a\in \mathbb{C}$ such that 
 $$ h_p:= -(q-q^{-1})^2(\lambda\mu)^{-1}ab-\lambda \mu^{-1}q - \mu\lambda^{-1}q^{-1} = \pm (q^n+q^{-n}) \quad \mbox{ for some }n\in \{1, \ldots, (N-1)/2\}$$
 the module $P(\lambda, \mu, a,b)$ has canonical basis $(x_0, \ldots, x_{N-1}, y_0, \ldots, y_{N-1})$ and module structure given by: 
 \begin{align*}
 & K^{1/2}x_i=\lambda q^{-i}x_i, K^{1/2}y_i=\lambda q^{-i}y_i, L^{1/2}x_i= \mu q^i x_i, L^{1/2}y_i=\mu q^i y_i \mbox{ for }i \in \{0, \ldots, N-1\}; \\
 & Fx_{N-1}=b x_0, Fy_{N-1}=by_0, Fx_i= x_{i+1}, Fy_i=y_{i+1} \mbox{ for }i\in \{0, \ldots, N-2\}; \\
 & Ex_0=ax_{N-1}, Ex_i=\left( \frac{q^{-i+1}\lambda^2 - q^{i-1}\mu^2}{q-q^{-1}} [i] +ab \right)x_{i-1} \mbox{ for }i\in \{1, \ldots, N-1\}; \\
 & Ey_0=a y_{N-1}+b^{-1}x_{N-1}, E y_i= \left( \frac{q^{-i+1}\lambda^2 - q^{i-1}\mu^2}{q-q^{-1}} [i] +ab \right)y_{i-1} + x_{i-1}  \mbox{ for }i\in \{1, \ldots, N-1\}.
 \end{align*}
 
 \item For $\lambda, \mu, c \in \mathbb{C}^*$ such that $\lambda \mu^{-1}= \pm q^n$ for some $n\in \{1, \ldots, (N-1)/2\}$, the module $\widetilde{P}(\lambda, \mu, c)$ has canonical basis $(x_0, \ldots, x_{N-1}, y_0, \ldots, y_{N-1})$ and module structure given by: 
 \begin{align*}
 & K^{1/2}x_i= \lambda q^i x_i, K^{1/2}y_i= \lambda q^i y_i, L^{1/2}x_i= \mu q^{-i}x_i, L^{1/2}y_i= \mu q^{-i} y_i  \mbox{ for }i \in \{0, \ldots, N-1\}; \\
 & Ex_{N-1} = cx_0, Ey_{N-1}=cy_0, Ex_i= x_{i+1}, Ey_i= y_{i+1} \mbox{ for }i\in \{0, \ldots, N-2\}; \\
 & Fx_0= 0, Fx_i= \left( \frac{ \mu^2q^{-i+1} - \lambda^2 q^{i-1} }{q-q^{-1}}[i] \right) x_{i-1}  \mbox{ for }i\in \{1, \ldots, N-1\}; \\
 & Fy_0=x_0, Fy_i=  \left( \frac{ \mu^2q^{-i+1} - \lambda^2 q^{i-1} }{q-q^{-1}}[i] \right) y_{i-1} + x_i  \mbox{ for }i\in \{1, \ldots, N-1\}.
 \end{align*}

 \item For $\mu \in \mathbb{C}^*$, $\varepsilon \in \{-1,+1\}$, $0\leq n \leq N-2$, the module $P_{\mu, \varepsilon, n}$ has canonical basis $(x_0, \ldots, x_{N-1}, y_0, \ldots, y_{N-1})$ and module structure given by: 
\begin{align*}
& K^{1/2} x_i = \varepsilon \mu A^{-2-n-2i}x_i, K^{1/2}y_i= \varepsilon \mu A^{n-2i}y_i, \quad L^{1/2} x_i= \mu A^{n+2i+2} x_i, L^{1/2}y_i=\mu A^{2i-n}y_i \mbox{, for }i\in \{0, \ldots, N-1\}; \\
& F x_i = x_{i+1}, \quad Fy_i= y_{i+1} \quad \mbox{( where }y_N=x_N=0)\mbox{, for }i\in \{0, \ldots, N-1\}; \\
& Ex_0=0, \quad Ex_i= -\mu^2[i][n+i+1] x_{i-1} \mbox{, for }i\in \{1, \ldots, N-1\}; \\
& Ey_i= \mu^2[i][n-i+1]y_{i-1} + x_{N-n+i-2} \mbox{, for }i\in \{0, \ldots, n\}, \quad E y_i = \mu^2[i] [n-i+1]y_{i-1}   \mbox{, for }i\in \{n+1, \ldots, N-1\}. 
\end{align*}

 \end{enumerate}
 
 \end{definition}
 
 \begin{definition}[Exceptional semi-weight representations]\label{def_exceptional_rep} For $0\leq n \leq N-1$, write $\overline{n}:=N-2-n$. 
 \begin{enumerate}
 \item For $\mu \in \mathbb{C}^*$, $\varepsilon\in \{-1, +1\}$, $0\leq n \leq N-2$, $k\geq 1$, the module $\Omega^{-k}_{\mu, \varepsilon, n}$ has basis $(e_i^j, \overline{e}_{i'}^{j'}, 0\leq i \leq n, 0\leq i' \leq \overline{n}, 1\leq j \leq k, 1\leq j' \leq k+1)$ and module structure given by:
 \begin{align*}
 & K^{1/2}e_i^j = \varepsilon \mu A^{n-2i}e_i^j, K^{1/2}\overline{e}_i^j =\varepsilon \mu A^{\overline{n}-2i}\overline{e}_i^j, L^{1/2}e_i^j = \mu A^{2i-n}e_i^j, L^{1/2}\overline{e}_i^j= \mu A^{2i-\overline{n}}\overline{e}_i^j; \\
 & F e_n^j = \overline{e}_0^{j+1},  F\overline{e}_{\overline{n}}^j = 0, Fe_i^j = e_{i+1}^j \mbox{ for }i\neq n , F \overline{e}_i^j = \overline{e}_{i+1}^j \mbox{ for } i\neq \overline{n}; \\
 & E e_0^j= \overline{e}_{\overline{n}}^j ,  E \overline{e}_0^j = 0,  Ee_i^j = \mu^2 [i][n-i+1] e_{i-1}^j, E \overline{e}_i^j = \mu^2 [i][\overline{n}-i+1]\overline{e}_{i-1}^j \mbox{ for } i\neq 0.
 \end{align*}
 \item  For $\mu \in \mathbb{C}^*$, $\varepsilon\in \{-1, +1\}$, $0\leq n \leq N-2$, $k\geq 1$, the module $\Omega^{k}_{\mu, \varepsilon, n}$ has basis $(e_i^j, \overline{e}_{i'}^{j'}, 0\leq i \leq n, 0\leq i' \leq \overline{n}, 1\leq j \leq k+1, 1\leq j' \leq k)$ and module structure given by:
 \begin{align*}
 & K^{1/2}e_i^j = \varepsilon \mu A^{n-2i}e_i^j, K^{1/2}\overline{e}_i^j =\varepsilon \mu A^{\overline{n}-2i}\overline{e}_i^j, L^{1/2}e_i^j = \mu A^{2i-n}e_i^j, L^{1/2}\overline{e}_i^j= \mu A^{2i-\overline{n}}\overline{e}_i^j; \\
 & Fe_n^{k+1}=0, F e_n^j = \overline{e}_0^{j} \mbox{ for }j\neq k+1,  F\overline{e}_{\overline{n}}^j = 0, Fe_i^j = e_{i+1}^j \mbox{ for }i\neq n , F \overline{e}_i^j = \overline{e}_{i+1}^j \mbox{ for } i\neq \overline{n}; \\
 & Ee_0^1=0, E e_0^j= \overline{e}_{\overline{n}}^{j-1} \mbox{ for }j\neq 1 ,  E \overline{e}_0^j = 0,  Ee_i^j = \mu^2 [i][n-i+1] e_{i-1}^j, E \overline{e}_i^j = \mu^2 [i][\overline{n}-i+1]\overline{e}_{i-1}^j \mbox{ for } i\neq 0.
 \end{align*}
\item   For $\mu \in \mathbb{C}^*$, $\varepsilon\in \{-1, +1\}$, $0\leq n \leq N-2$, $k\geq 1$ and $\lambda= \in \mathbb{CP}^1= \mathbb{C} \cup \{\infty \}$, the module  $M^k_{\mu, \varepsilon, n}(\lambda)$ has basis $(e_i^j, \overline{e}_{i'}^{j'}, 0\leq i \leq n, 0\leq i' \leq \overline{n}, 1\leq j \leq k, 1\leq j' \leq k)$ and module structure given as follows.
\par  If $\lambda \in \mathbb{C}$ then
 \begin{align*}
 & K^{1/2}e_i^j = \varepsilon \mu A^{n-2i}e_i^j, K^{1/2}\overline{e}_i^j =\varepsilon \mu A^{\overline{n}-2i}\overline{e}_i^j, L^{1/2}e_i^j = \mu A^{2i-n}e_i^j, L^{1/2}\overline{e}_i^j= \mu A^{2i-\overline{n}}\overline{e}_i^j; \\
 & Fe_n^k = \lambda \overline{e}_0^k,  F e_n^j = \lambda \overline{e}_0^{j} + \overline{e}_0^{j+1} \mbox{ for }j\neq k,  F\overline{e}_{\overline{n}}^j = 0, Fe_i^j = e_{i+1}^j \mbox{ for }i\neq n , F \overline{e}_i^j = \overline{e}_{i+1}^j \mbox{ for } i\neq \overline{n}; \\
 &  E e_0^j= \overline{e}_{\overline{n}}^j,  E \overline{e}_0^j = 0,  Ee_i^j = \mu^2 [i][n-i+1] e_{i-1}^j, E \overline{e}_i^j = \mu^2 [i][\overline{n}-i+1]\overline{e}_{i-1}^j \mbox{ for } i\neq 0.
 \end{align*}
If $\lambda= \infty \in \mathbb{CP}^1$, then 
\begin{align*}
 & K^{1/2}e_i^j = \varepsilon \mu A^{n-2i}e_i^j, K^{1/2}\overline{e}_i^j =\varepsilon \mu A^{\overline{n}-2i}\overline{e}_i^j, L^{1/2}e_i^j = \mu A^{2i-n}e_i^j, L^{1/2}\overline{e}_i^j= \mu A^{2i-\overline{n}}\overline{e}_i^j; \\
 &  F e_n^j =  \overline{e}_0^{j},  F\overline{e}_{\overline{n}}^j = 0, Fe_i^j = e_{i+1}^j \mbox{ for }i\neq n , F \overline{e}_i^j = \overline{e}_{i+1}^j \mbox{ for } i\neq \overline{n}; \\
 &  E e_0^k= 0, E e_0^j= \overline{e}_{\overline{n}}^{j+1} \mbox{ for }j \neq k,  E \overline{e}_0^j = 0,  Ee_i^j = \mu^2 [i][n-i+1] e_{i-1}^j, E \overline{e}_i^j = \mu^2 [i][\overline{n}-i+1]\overline{e}_{i-1}^j \mbox{ for } i\neq 0.
 \end{align*}

 \end{enumerate}
 \end{definition}
 We postpone the proof of the following theorem to the next subsection. Let $\mathcal{C}$ be the category of weight $\Dq$ modules and $\overline{\mathcal{C}}$ the category of semi-weight $\Dq$ modules.
 
 \begin{theorem}\label{theorem_representations_QG}
 \begin{enumerate}
 \item Every modules $V(\lambda, \mu, a,b), \widetilde{V}(\lambda, \mu, c), S_{\mu, \varepsilon, n}$ are weight indecomposable and every weight indecomposable $\Dq$ module is isomorphic to one of them. 
 \item Every modules $P(\lambda, \mu, a,b), \widetilde{P}(\lambda, \mu, c), P_{\mu, \varepsilon, n}, \Omega^k_{\mu, \varepsilon, n}, \Omega^{-k}_{\mu, \varepsilon, n}, M_{\mu, \varepsilon, n}^k(\lambda)$ are semi weight indecomposable and every indecomposable semi-weight module is either isomorphic to one of them or is a weight module.
 \item  The modules $S_{\mu, \varepsilon, n}$ are simple. The module $\widetilde{V}(\lambda, \mu, c)$ is simple if and only if either $c\neq 0$ or $\lambda \mu^{-1} \neq \pm q^{n-1}$ for all $n\in \{1, \ldots, N-1\}$. The module $V(\lambda, \mu, a, b)$ is simple if and only if either $\prod_{i\in \mathbb{Z}/N\mathbb{Z}} (ab+\frac{q^{1-i}\lambda^2 - q^{i-1}\mu^2}{q-q^{-1}}[i])\neq 0$ or $\lambda \mu^{-1} \neq \pm q^{n-1}$ for all $n\in \{1, \ldots, N-1\}$. Any simple $\Dq$-module is isomorphic to one of these modules.
 \item The weight representations $S_{\mu, \epsilon, N-1}$, $V(\lambda, \mu, a,b)$ and $\widetilde{V}(\lambda, \mu, c)$ are projective objects in $\mathcal{C}$ and any indecomposable projective object of $\mathcal{C}$ is isomorphic to one of them.
 \item The semi weight representations $S_{\mu, \epsilon, N-1}$, $P_{\mu, \epsilon, n}$, $P(\lambda, \mu, a,b)$ and $\widetilde{P}(\lambda, \mu, c)$ are projective objects in $\overline{\mathcal{C}}$ and any indecomposable projective object of $\overline{\mathcal{C}}$ is isomorphic to one of them.
\end{enumerate}
\end{theorem}

Let us identify $X(\mathbb{D}_1)$ with the set of pairs $g=(g_+, g_-)$ where $g_+ \in B^+$ and $g_-\in B^-$  are upper and lower triangular matrices of $\SL_2(\mathbb{C})$. More precisely, for $x\in X(\mathbb{D}_1)$, we set 
\begin{align*}
{} & g_+:= \begin{pmatrix} \chi_x(\alpha_{++}^N) &  \chi_x(\alpha_{+-}^N) \\ 0 & \chi_x(\alpha_{--}^N) \end{pmatrix} = \chi_x\begin{pmatrix} K^{-N/2} & -(q-q^{-1})^N E^N K^{-N/2} \\ 0 & K^{N/2} \end{pmatrix}, \quad 
  \mbox{ and } \\
{}&  g_-:= \begin{pmatrix} \chi_x(\beta_{++}^N) &  0\\ \chi_x(\beta_{-+}^N)  & \chi_x(\beta_{--}^N) \end{pmatrix}= \chi_x \begin{pmatrix} L^{-N/2} & 0 \\ (q-q^{-1})^N L^{-N/2}F^N & L^{N/2} \end{pmatrix}
 . 
 \end{align*}
  Write $\varphi(g):=g_{-}^{-1}g_+$. $\widehat{X}(\mathbb{D}_1)$ is identified with the set of elements $(g, h_p, h_{\partial})$ with $g=(g_+, g_-) \in X(\mathbb{D}_1)$, $T_N(\gamma_p)=-\tr (\varphi(g))$ and $h_{\partial}^{-N}$ is the left upper matrix coefficient of $g_-g_+$. An indecomposable representation of weight   $(g, h_p, h_{\partial})$ is called cyclic/ semi-cyclic/ diagonal/ central if $\varphi(g)\in \SL_2(\mathbb{C})$ is not triangular/ triangular not diagonal/ diagonal not central/ central (i.e. $\varphi(g)=\pm \mathds{1}_2$) respectively. 
 \begin{notations} 
  Write $\Delta_+(\widehat{\mathfrak{sl}}_2):= \{ (n,m) \in \mathbb{N}^2 | (n-m)^2 \leq 1\}$ and decompose it as 
  \\ $\Delta_+(\widehat{\mathfrak{sl}}_2)= \{(0,1), (1,0)\} \sqcup \Delta^{\Re}_{++}(\widehat{\mathfrak{sl}}_2)\sqcup \Delta^{Im}_+(\widehat{\mathfrak{sl}}_2)$ where 
 $$  \Delta^{\Re}_{++}(\widehat{\mathfrak{sl}}_2)= \{ (k, k+1), k\geq1\} \cup \{(k+1, k), k\geq 1\}, \quad \Delta^{Im}_+(\widehat{\mathfrak{sl}}_2)=\{ (k,k), k\geq 1\}.$$
 Further write 
 $$ \Delta:= \{S, P\} \sqcup   \Delta^{\Re}_{++}(\widehat{\mathfrak{sl}}_2) \sqcup  \Delta^{Im}_+(\widehat{\mathfrak{sl}}_2)\times \mathbb{CP}^1.$$ 
 Given $\mu \in \mathbb{C}^*$, $\varepsilon \in \{-1, +1\}$, $n\in \{0, \ldots, N-2\}$, define a map $\underline{\sigma}_{\mu, \varepsilon, n}: \Delta \to \Indecomp(U_q\mathfrak{gl}_2)$ by
 $$ \underline{\sigma}_{\mu, \varepsilon, n}: S \mapsto S_{\mu, \varepsilon, n}, P\mapsto P_{\mu, \varepsilon, n}, (k+1, k)\mapsto \Omega^k_{\mu, \varepsilon, n}, (k, k+1) \mapsto \Omega^{-k}_{\mu, \varepsilon, n}, ((k,k), \Lambda) \mapsto M^k_{\mu, \varepsilon, n}(\Lambda).$$
 Consider the set $\Delta \sqcup \overline{\Delta}$ made of two copies of $\Delta$; said differently $\Delta \sqcup \overline{\Delta}:= \Delta \times \{0,1\}$ and for $\alpha \in \Delta$ we write $\alpha:= (\alpha, 0) \in \Delta \bigsqcup \overline{\Delta}$ and $\overline{\alpha}:= (\alpha, 1) \in \Delta \sqcup \overline{\Delta}$. For $\mu \in \mathbb{C}^*$, $\varepsilon \in \{-1, +1\}$, $n\in \{0, \ldots, (N-3)/2\}$, recall that $\overline{n}=N-n-2$ and  consider the map 
 $$\sigma_{\mu, \varepsilon, n}:= \underline{\sigma}_{\mu, \varepsilon, n}\sqcup \underline{\sigma}_{\mu, \varepsilon, \overline{n}} : \Delta \sqcup \overline{\Delta} \to  \Indecomp(\Dq).$$
 \end{notations}

  As a consequence of Theorem \ref{theorem_representations_QG}, we obtain the 

\begin{corollary}\label{coro_QG} Let $\overline{\mathcal{C}}$ be the category of finite dimensional semi-weight $D_q$ modules. 
\begin{enumerate}
\item The Azumaya locus of $\Dq \cong \overline{\mathcal{S}}_A(\mathbb{D}_1)$ is the set of elements $(g, h_p, h_{\partial})$ such that either $\varphi(g)$ is not central or $\varphi(g)=\pm \mathds{1}_2$ and $h_p=\mp 2$ so it is equal to the regular locus $\widehat{X}^{reg}(\mathbb{D}_1)$.
\item The fully Azumaya locus of $\Dq \cong \overline{\mathcal{S}}_A(\mathbb{D}_1)$ is the set of elements $g\in X(\mathbb{D}_1)$ such that $\varphi(g)\neq \pm \mathds{1}$. For such a $g$, the isomorphism classes of semi-weight indecomposable $D_q$ modules with classical shadow $g$ are completely determined by their full shadow.
\item  For $\widehat{g}=(g, h_p, h_{\partial})\in \widehat{X}(\mathbb{D}_1)$ with $\varphi(g)\neq \pm \mathds{1}$ then 
\begin{itemize}
\item If $\tr(\varphi(g))\neq \pm 2$, then up to isomorphism there is a unique indecomposable semi-weight module representation $S_{\widehat{g}}$ with maximal shadow $\widehat{g}$: it is the Azumaya representation which is both simple and projective in $\overline{\mathcal{C}}$.
\item If $\tr(\varphi(g))=\pm 2$, then up to isomorphism there exist two indecomposable semi-weight representations $S_{\widehat{g}}, P_{\widehat{g}}$ with maximal shadow $\widehat{g}$. $S_{\widehat{g}}$ is the Azumaya representation and is simple not projective. $P_{\widehat{g}}$ is projective in $\overline{\mathcal{C}}$ not simple and we have a non split exact sequence
$$ 0 \to S_{\widehat{g}} \to P_{\widehat{g}} \to S_{\widehat{g}} \to 0.$$
\end{itemize}
\item Let $\widehat{x}=(g,h_p, h_{\partial})\in \widehat{X}(\mathbb{D}_1)$ which is not in the Azumaya locus, so one has $g=\left( \varepsilon \begin{pmatrix} \Lambda & 0 \\ 0 & \Lambda^{-1} \end{pmatrix},  \begin{pmatrix} \Lambda & 0 \\ 0 & \Lambda^{-1} \end{pmatrix} \right)$ for some $\varepsilon = \pm 1$ and $\Lambda \in \mathbb{C}^*$ and $h_p=-\varepsilon (q^{n+1} + q^{-(1+n)})$ for some $n\in \{0, \ldots, (N-3)/2\}$. Let $\mu \in \mathbb{C}^*$ be the unique scalar such that $\mu^{-N} = \Lambda$ and $h_{\partial}= \varepsilon \mu^{-2}$. Then $\sigma_{\mu, \varepsilon, n}$ is a bijection between $\Delta \sqcup \overline{\Delta} $ and the set of isomorphism classes of indecomposable $\Dq$ modules with classical shadow $\widehat{x}$.
\end{enumerate}
\end{corollary}

In particular, Theorem \ref{main_theorem_intro} is verified for $\mathbb{D}_1$. 

\begin{remark}
Being a projective object in the category $\overline{\mathcal{C}}$ of semi-weight modules is different from being a projective module (i.e. from being a projective object in $\Dq-\Mod$). As we shall see, for $\widehat{g}=(g,h_p,h_{\partial})\in \widehat{X}(\mathbb{D}_1)$ with $\tr(\varphi(g))\neq \pm 2$, the Azumaya representation $S_{\widehat{g}}$ is a proper submodule of an indecomposable module of the form $P(\lambda, \mu, a, b)$ which is not semi-weight. So  $S_{\widehat{g}}$ 
is not a projective module though it is a projective object in $\overline{\mathcal{C}}$.
\end{remark}

 \subsection{Classification of semi-weight $\Dq$ modules}\label{sec_DqModules}
 
In this subsection, we prove Theorem  \ref{theorem_representations_QG}. Denote by $\mathcal{C}$ the category of  weight $\Dq$ modules and by $\overline{\mathcal{C}}$ its category of semi-weight modules. For $0\leq n \leq N-1$, we write $\overline{n}:= N-n-2$.
 
 \subsubsection{Representations in the fully Azumaya locus}
 
  \begin{lemma}\label{lemm1}
\begin{enumerate}
\item 
  For $0\leq n \leq N-2$, we have a non-split exact sequence 
\begin{equation}\label{suitexacte}
 0 \rightarrow S_{\mu, \varepsilon, \overline{n}} \xrightarrow{i} V(\varepsilon \mu A^n, \mu A^{-n}, 0, 0) \xrightarrow{p} S_{\mu, \varepsilon, n} \rightarrow 0, 
 \end{equation}
where the equivariant maps $i$ and $p$ are defined by $i(e_j):= v_{j+n+1}$ for $0\leq j  \leq N-n-2$ and $p(v_j):= \left\{ \begin{array}{ll} e_j &\mbox{, if }0\leq i \leq n; \\ 0 & \mbox{, else.} \end{array} \right.$ 
 \item One has non-split  exact sequences:
$$ 0 \to V(\varepsilon \mu A^{-2-n}, \mu A^{2+n}, 0, 0) \xrightarrow{\iota} P_{\mu, \varepsilon, n} \xrightarrow{p} V(\varepsilon \mu A^n, \mu A^{-n}, 0, 0) \to 0, $$
$$ 0 \to V(\lambda, \mu, a,b)  \xrightarrow{\iota} P(\lambda, \mu, a,b)  \xrightarrow{p}  V(\lambda, \mu, a,b) \to 0,$$
and 
$$0 \to \widetilde{V}(\lambda, \mu, c)  \xrightarrow{\iota} \widetilde{P}(\lambda, \mu, c)  \xrightarrow{p}  \widetilde{V}(\lambda, \mu, c)\to 0, $$
 where $\iota(v_i):= x_i$ and $p(x_i):=0$, $p(y_i):=v_i$ (resp. $\iota(w_i)=x_i$, $p(x_i)=0$, $p(y_i)=w_i$ in the third case). 
 \end{enumerate}
 \end{lemma}
 
 \begin{proof} The proof is a straightforward verification left to the reader. \end{proof}
 
 For $(\lambda, \mu, b) \in \mathbb{C}^*$ and $a \in \mathbb{C}$, the module $P(\lambda, \mu, a,b)$ of Definition \ref{def_semiweight_modules} is always well defined. Also for $c\in \mathbb{C}^*$, $\widetilde{P}(\lambda, \mu, c)$  is well defined as well. However they are not always semi weight modules.

\begin{lemma}\label{lemm2}
\begin{enumerate}
\item $P(\lambda, \mu, a,b)$ is a semi weight module if and only if 
$$ h_p:= -(q-q^{-1})^2(\lambda\mu)^{-1}ab-\lambda \mu^{-1}q - \mu\lambda^{-1}q^{-1} = \pm (q^n+q^{-n}) \quad \mbox{ for some }n\in \{1, \ldots, (N-1)/2\}$$
\item $\widetilde{P}(\lambda, \mu, c)$  is a semi weight module if and only if  $\lambda \mu^{-1}= \pm q^n$ for some $n\in \{1, \ldots, (N-1)/2\}$.
\item $P_{\lambda, \varepsilon, n}$ is always a semi weight module.
\end{enumerate}
\end{lemma}

\begin{proof}
 Denote by $\rho : \Dq \to \End (P(\lambda, \mu, a,b))$ the associated representation. Clearly we have $\rho(K^{N/2})= \lambda^N \id$, $\rho(L^{N/2})=\mu^N \id$, $\rho(F^N)=b \id$. 
Let $V:= \Span(x_i, i=0, \ldots, N-1) \subset P(\lambda, \mu, a,b)$ such that $V\cong V(\lambda, \mu, a, b)$ and denote by $\widehat{x}=(g,h_p,h_{\partial})$ its classical shadow. 
In particular: 
$$ h_p:= -(q-q^{-1})^2(\lambda\mu)^{-1}ab-\lambda \mu^{-1}q - \mu\lambda^{-1}q^{-1}$$
and  $\restriction{\rho}{V}(\gamma_p) = h_p \id_V$. Write
$$ e:= -\frac{(\lambda \mu)^N}{(q-q^{-1})^{2N}} \prod_{i\in \mathbb{Z}/N\mathbb{Z}} (h_p +\lambda \mu^{-1} q^{1-2i} + \mu\lambda^{-1}q^{2i-1}).$$
Clearly $\restriction{\rho(E^N)}{V}= e \id_V$. Let us prove that $\rho(E^N)=e \id$ if and only if $h_p \in \{ \pm( q^n+q^{-n}), n\in \{1, \ldots, N-1\} \}$. Write 
$$ e_i(X):= -\frac{\lambda \mu}{(q-q^{-1})^2} (X+ \lambda \mu^{-1} q^{1-2i} + \mu \lambda^{-1}q^{2i-1})$$
and $P(X):= \prod_{i \in \mathbb{Z}/N\mathbb{Z}} e_i(X) - e$ so that $P(h_p)=0$. Using the equality
$$ \rho(E) y_i= e_i(h_p)y_{i-1} +x_{i-1}, $$
we find that 
$$ \rho(E^N)y_i= e y_i + P'(h_p) x_i.$$
So $\rho(E^N)=e \id $ if and only if $h_p$ is a zero of $P(X)$ of multiplicity $>1$. On the other hand, since the minimal polynomial of $\restriction{\rho}{V}(\gamma_p)$ is $T_N(X)+\tr(\varphi(g))$ and since $P(\restriction{\rho}{V}(\gamma_p))=0$, we have $P(X)=-\frac{(\lambda \mu)^N}{(q-q^{-1})^{2N}} (T_N(X)+ \tr(\varphi(g)))$ so $h_p$ is a root or $P(X)$ of multiplicity $>1$ if and only if $h_p \in \{ \pm( q^n+q^{-n}), n\in \{1, \ldots, N-1\} \}$. 
\par The cases of  $\widetilde{P}(\lambda, \mu, c)$ and $P_{\lambda, \varepsilon, n}$  are done similarly (and much easier) and left to the reader.
\end{proof}

 \begin{lemma}\label{lemm3} 
 \begin{enumerate}
 \item The modules $S_{\mu, \varepsilon, n}, V(\lambda, \mu, a,b), \widetilde{V}(\lambda, \mu, c), P(\lambda, \mu, a,b), \widetilde{P}(\lambda, \mu, c), P_{\mu, \varepsilon, n}$ are indecomposable.
 \item The modules $S_{\mu, \varepsilon, n}$ are simple. The module $\widetilde{V}(\lambda, \mu, c)$ is simple if and only if either $c\neq 0$ or $\lambda \mu^{-1} \neq \pm q^{n-1}$ for all $n\in \{1, \ldots, N-1\}$. The module $V(\lambda, \mu, a, b)$ is simple if and only if either $\prod_{i\in \mathbb{Z}/N\mathbb{Z}} (ab+\frac{q^{1-i}\lambda^2 - q^{i-1}\mu^2}{q-q^{-1}}[i])\neq 0$ or $\lambda \mu^{-1} \neq \pm q^{n-1}$ for all $n\in \{1, \ldots, N-1\}$. 
  \end{enumerate}
\end{lemma}

\begin{proof} Let $V$ be one module of the form $S_{\mu, \varepsilon, n}, V(\lambda, \mu, a,b), \widetilde{V}(\lambda, \mu, c)$ and denote by $(u_i)_i$ its canonical basis (\textit{i.e.} $u_i$ is either $e_i$, $v_i$ or $w_i$ if $V$ has the form $S_{\mu, \varepsilon, n}, V(\lambda, \mu, a,b)$ or $\widetilde{V}(\lambda, \mu, c)$ respectively). Write $\rho : \Dq \rightarrow \End(V)$ the morphism associated to the module structure. Let $\End_{\mathcal{C}}(V)$ be the space of equivariant endomorphisms of $V$ and fix $0 \neq \theta \in \End_{\mathcal{C}}(V)$. By Definition \ref{def_QGRep}, the operator $\rho(K^{1/2})$ is diagonal in the basis $(u_i)_i$ and each eigenvalue arises with multiplicity one. Since $\theta$ commutes with $\rho(K^{1/2})$, it preserves the eigenspace $\mathbb{C}u_0$, hence there exists $z\in \mathbb{C}$ such that $\theta \cdot u_0= z u_0$. First suppose that $z\neq 0$. By definition, the vectors of the canonical basis are characterized by either $u_i= \rho(F)^i u_0$ or $u_i= \rho(E)^i u_0$. Since $\theta$ commutes with both $\rho(E)$ and $\rho(F)$, one has $\theta u_i= z u_i$ and $\theta$ is scalar. Next suppose that $\theta u_0=0$. Since $\theta\neq 0$, there exists an index $i_0$  and  $z\in \mathbb{C}^*$, such that $\theta u_{i_0}=z u_{i_0}$ and $\theta u_i = 0$ for all $0\leq i \leq i_0-1$. Let $W:= \Span \left( u_j, j\geq i_0 \right)\subset V$ and denote by $P_W$ the projection onto $W$ parallel to $\Span\left( u_j, 0\leq j \leq i_0 \right)$. Since the vectors $u_j$, for $j\geq i_0$ are obtained from $u_{i_0}$ by composition with some power of either $\rho(E)$ or $\rho(F)$, the equivariance of $\theta$ implies that $\theta= z P_W$. 
\vspace{2mm}
\par 
Conversely, if there exists such an equivariant projector $P_W$, then one has either $\rho(E)u_{i_0}=0$, if $V$ has the form $V_{\mu, \varepsilon, n}$ or $V(\lambda, \mu, a,b)$, and $\rho(F) u_{i_0}= 0$, if $V$ has the form $\widetilde{V}(\lambda, \mu, c)$. By Definition \ref{def_QGRep}, such an index $i_0$ does not exist for the modules $S_{\mu, \varepsilon, n}$, neither for the modules  $\widetilde{V}(\lambda, \mu, c)$ such that either $c\neq 0$ or $\lambda \mu^{-1} \neq \pm q^{n-1}$ for all $n\in \{1, \ldots, N-1\}$, nor for the modules $V(\lambda, \mu, a, b)$ such that either  $\prod_{i\in \mathbb{Z}/N\mathbb{Z}} (ab+\frac{q^{1-i}\lambda^2 - q^{i-1}\mu^2}{q-q^{-1}}[i])\neq 0$ or $\lambda \mu^{-1} \neq \pm q^{n-1}$ for all $n\in \{1, \ldots, N-1\}$. Hence in each of these cases, one has $\End_{\mathcal{C}}(V)= \mathbb{C} \id$ and these modules are simple. In the other cases, one finds that such an index $i_0$ is unique and by Lemmas \ref{lemm1},  the corresponding projector $P_W$ is equivariant. Hence  $\End_{\mathcal{C}}(V)=\id \oplus P_W$. Moreover in this case, $V$ is isomorphic to a module $V(\varepsilon \mu A^n, \mu A^{-n}, 0, 0)$
and since the exact sequence \eqref{suitexacte} in Lemma \ref{lemm1} does not split, $V$ is indecomposable. Similarly, the fact that the modules $P(\lambda, \mu, a,b), \widetilde{P}(\lambda, \mu, c)$ are indecomposable follows from the existence of the non-split exact sequences in Lemma \ref{lemm1}.
\par 
Let us prove that  $P_{\mu, \varepsilon, n}$ is indecomposable.  Let $\pi:= \rho(C - \frac{\mu^2}{q-q^{-1}} [n+1]) \in \End_{\mathcal{C}}(P_{\mu, \varepsilon, n})$ be the nilpotent map  such that   $\pi(x_j)=0$,  $\pi (y_i)=x_{N-n-i}$ for $i\leq n-1$, $\pi(y_i)=0$ for $i\geq n$ , and let us prove that $\End_{\mathcal{C}}(P_{\mu, \varepsilon, n})=\id \oplus \pi$. This will imply that $\End_{\mathcal{C}}(P_{\mu, \varepsilon, n})\cong \quotient{\mathbb{C}[\pi]}{(\pi^2)}$ is a local ring, thus that  $P_{\mu, \varepsilon, n}$ is indecomposable.
Let  $\theta \in \End_{\mathcal{C}}(P_{\mu, \varepsilon, n})$. Then $\theta(y_0)$ is an eigenvector of $K^{1/2}$ with eigenvalue $\varepsilon\mu A^n$ and is annihilated by $(FE)^2$, so there exists $c,c'\in \mathbb{C}$ such that $\theta(y_0)=cy_0+c' x_{N-n}$. Let us prove that $\theta= c\id + c' \pi$. First, since $\theta(F^i y_0) = F^i \theta(y_0)$, one finds:
$$ \theta(y_i) = \left\{ \begin{array}{ll}
cy_i + c' x_{N-n+i} & \mbox{, if }i\leq n-1; \\
cy_i & \mbox{ if }i\geq n.
\end{array} \right. 
= (c\id +c'\pi) (y_i).$$
Next, since $\theta(Ey_0)=E\theta(y_0)$ one finds $\theta(x_{N-n-1})=cx_{N-n-1}= (c\id +c' \pi) x_{N-n-1}$. Since $\theta(F^i x_{N-n-1})=F^i \theta(x_{N-n-1})$, one gets $\theta(x_i)=cx_i = (c\id+c' \pi)(x_i)$ for $i\geq N-n-1$ and since $\theta(E^i x_{N-n-1})= E^i \theta(x_{N-n-i})$ one finds that $\theta(x_i)=cx_i= (c\id +c' \pi)x_i$ for $i\leq N-n-1$. Thus $\theta= c\id +c' \pi$.

\end{proof}

\begin{lemma}\label{lemm4} \begin{enumerate}
\item The Azumaya locus $\mathcal{AL}(\mathbb{D}_1)$ is the set of elements $(g,h_p,h_{\partial})$ such that either $\varphi(g)\neq \pm \mathds{1}_2$ or $\varphi(g)=\pm \mathds{1}_2$ and $h_p=\mp 2$. 
\item The fully Azumaya locus $\mathcal{FAL}(\mathbb{D}_1)$ is the set of elements $g$ such that $\varphi(g)\neq \pm \mathds{1}_2$.
\end{enumerate}
\end{lemma}

\begin{proof} The second assertion is an immediate consequence of the first one. To compute  $\mathcal{AL}(\mathbb{D}_1)$, recall from Theorem \ref{theorem_AL} that an irreducible  representation $\rho : \Dq \to \End(V)$ with classical shadow $\widehat{x}$ has dimension $N$ if and only $\widehat{x} \in \mathcal{AL}(\mathbb{D}_1)$. When $\widehat{x}= (g,h_p,h_{\partial})$ is such that  $\varphi(g)\neq \pm \mathds{1}_2$, then it is the classical shadow of an $N$-dimensional representation of the form $V(\lambda,\mu, a,b)$ or $\widetilde{V}(\lambda, \mu, c)$ which is irreducible by Lemma \ref{lemm3}. When $\varphi(g)=\pm \mathds{1}_2$ and $h_p=\mp 2$, then $\widehat{x}$ is the shadow of an $N$ dimensional representation $S_{\mu, \varepsilon, N-1}$ which is irreducible by Lemma \ref{lemm3}. So $\widehat{x} \in \mathcal{AL}(\mathbb{D}_1)$ in these cases. If $\varphi(g)=\pm \mathds{1}_2$ and $h_p\neq \mp 2$, then $\widehat{x}$ is the shadow of a representation $S_{\mu, \varepsilon, n}$ with $0\leq n \leq N-2$ whose dimension if $<N$ and which is irreducible by Lemma \ref{lemm3}, so $\widehat{x} \notin \mathcal{AL}(\mathbb{D}_1)$.

\end{proof}

 \begin{lemma}\label{lemm5} Let $\widehat{x}=(g,h_p,h_{\partial}) \in \widehat{X}(\mathbb{D}_1)$. 
 \begin{enumerate}
 \item If $\tr( \varphi(g))\neq \pm 2$, there exists (up to isomorphism) a unique indecomposable semi-weight representation with maximal shadow  $\widehat{x}$; it is of the form $V(\lambda, \mu, a,b)$ or $\widetilde{V}(\lambda, \mu, c)$.
 \item If $\tr( \varphi(g))= \pm 2$ and $\varphi(g)\neq \mp \mathds{1}_2$, there exist  (up to isomorphism) two indecomposable semi-weight representations with maximal shadow $\widehat{x}$; one is of the form $V(\lambda, \mu, a,b)$ or $\widetilde{V}(\lambda, \mu, c)$ and the other has the form $P(\lambda, \mu, a,b)$ or $\widetilde{P}(\lambda, \mu, c)$.
 \end{enumerate}
 \end{lemma}
 
 \begin{proof} When $\varphi(g)\neq \pm \mathds{1}_2$, then $g\in \mathcal{FAL}(\mathbb{D}_1)$ by Lemma \ref{lemm4} so the lemma follows from Theorem \ref{theorem_FAL}
   \end{proof}
 
 \subsubsection{Representations outside the fully Azumaya locus}
 
 For $\mathscr{A}$ an algebra, we denote by $\Indecomp(\mathscr{A})$ the set of isomorphism classes of indecomposable $\mathscr{A}$-modules.
 \begin{definition}(Small quantum group) The (simply connected) \textbf{small quantum enveloping algebra} $\uq$ has generators $E,F, k^{\pm 1/2}$ and relations: 
 $$
Ek^{1/2}= q^{-1}k^{1/2}E; \quad   Fk^{1/2}=qk^{1/2}F; \quad F^N=E^N=k^{N/2}-1=0, \quad EF-FE= \frac{k-k^{-1}}{q-q^{-1}}. 
$$
 \end{definition}
 
Let $\rho: \Dq \to \End(V)$ be an indecomposable semi-weight representation whose shadow $g\in X(\mathbb{D}_1)$ is not in the fully Azumaya locus. By Lemma \ref{lemm4}, $g=\left( \varepsilon \begin{pmatrix} \Lambda & 0 \\ 0 & \Lambda^{-1} \end{pmatrix},  \begin{pmatrix} \Lambda & 0 \\ 0 & \Lambda^{-1} \end{pmatrix} \right)$ for some $\varepsilon \in \{-1, +1\}$  and $\Lambda \in \mathbb{C}^*$. 
Since $\rho$ is indecomposable,  there exists $h_{\partial}\in \mathbb{C}^*$ such that $\rho(H_{\partial})=h_{\partial}\id_V$. Let $\mu \in \mathbb{C}^*$ be the unique scalar such that $\Lambda= \mu^N$ and $h_{\partial}=\varepsilon \mu^{-2}$ and denote by $\Indecomp(U_q\mathfrak{gl}_2; \varepsilon, \mu) \subset \Indecomp(\Dq)$ the subset of modules with shadow $g$ and $h_{\partial}$. Define a map 
$$ (\cdot)_{\mu, \varepsilon}: \Indecomp(\uq) \to \Indecomp(\Dq; \varepsilon, \mu), \quad V\mapsto V_{\mu, \varepsilon}, $$
 by sending $\rho : \uq \to \End(V)$ to the representation $\rho_{\mu, \varepsilon}: \Dq \to \End(V_{\mu, \varepsilon})$ where $V_{\mu, \varepsilon}=V$ as a vector space and 
 $$ \rho_{\mu, \varepsilon}(K^{1/2})=\mu \varepsilon \rho(k^{1/2}), \rho_{\mu, \varepsilon}(L^{1/2})=\mu  \rho(k^{-1/2}), \rho_{\mu, \varepsilon}(E)=\mu^2 \rho(E), \rho_{\mu, \varepsilon}(F)= \rho(F).$$
 
 \begin{lemma}\label{lemm6} $ (\cdot)_{\mu, \varepsilon}: \Indecomp(\uq) \to \Indecomp(\Dq; \varepsilon, \mu)$ is a bijection. \end{lemma}
 
 \begin{proof} 
 The construction of the inverse map is straightforward and left to the reader.
 \end{proof}
 
 \begin{notations} We denote by $S_n, P_n, \Omega^{\pm k}_n, M^k_n(\Lambda)$ the $\uq$ modules which correspond to the modules $S_{\mu, \varepsilon, n}, P_{\mu, \varepsilon, n}, \Omega^{\pm k}_{\mu, \varepsilon, n}, M^k_{\mu, \varepsilon, n}(\Lambda)$ respectively.
 \end{notations}
 
 The classification of the indecomposable $\uq$ modules was accomplished in \cite{Suter_Uqsl2Modules} (see also \cite{GSTF_KLlogCFT}) with two differences: $(1)$ in  \cite{Suter_Uqsl2Modules} the parameter $q$ is a root of unity of even order and $(2)$ in  \cite{Suter_Uqsl2Modules} the algebra considered has generators $E,F, k^{\pm 1}$ instead of $E,F, k^{\pm 1/2}$. In the rest of the appendix, we adapt the arguments in \cite{Suter_Uqsl2Modules} to our setting in order to prove the following 
 
 \begin{proposition}\label{prop_uq}
 \begin{enumerate}
 \item The indecomposable $\uq$ modules are the ones isomorphic to the modules  $S_n, P_n$, $\Omega^{\pm k}_n, M^k_n(\Lambda)$ and these modules are pairwise inequivalent.
 \item The irreducible $\uq$ modules are the ones isomorphic to the modules $S_n$.
 \item The projective $\uq$ modules are the ones isomorphic to the modules $S_{N-1}, P_n$.
 \end{enumerate}
 \end{proposition}
 
 The proof of Proposition \ref{prop_uq} is a straightforward adaptation of the arguments in \cite{Suter_Uqsl2Modules} so we reproduce them briefly and refer the reader to \cite{Suter_Uqsl2Modules} for more details. The first step is to decompose $\uq$ into a direct sum of indecomposable left ideals. Write $P_{N-1}:=S_{N-1}$. For $n\in \mathbb{Z}/N\mathbb{Z}$, set 
 $$\varphi_n:= \sum_{i\in \mathbb{Z}/N\mathbb{Z}}q^{-ni}k^{i/2}$$
 so $k^{1/2}\varphi_n= \varphi_n k^{1/2} = q^n \varphi_n$, $E\varphi_n=\varphi_{n+1}E$ and $F\varphi_n= \varphi_{n-1}F$. An easy application of the diamond Lemma for PBW bases shows that $\uq$ admits the following basis
 $$\{ F^a \varphi_n E^b, 0\leq a,b,n \leq N-1\}.$$
 In particular, if $x\in \uq$ is such that $Fx=0$ then there exists a unique $y\in \Span(F^a \varphi_nE^b, a\leq N-2)$ such that $Fy=x$. For $0\leq n \leq N-2$, applying this remark to $x_n:=E^{\overline{n}}F^{N-1}\varphi_n$ we obtain an element $\gamma_n \in \Span(F^a \varphi_nE^b, a\leq N-2)$ uniquely characterized by the fact that $F\gamma_n=x_n$.  Set $\gamma_{N-1}:= F^{N-1} \varphi_{N-1}$ and for $0\leq n \leq N-1$ set
 $$ P_n(0):= \uq \cdot \gamma_n, \quad P_n(h)= P_n(0)\cdot E^h \quad h\in \{0, \ldots, n\}.$$

 \begin{lemma}\label{lemm7} The algebra $\uq$ admits the direct sum decomposition into left ideals 
 $$ \uq =  \oplus_{n=0}^{N-1}\oplus_{h=0}^n P_n(h)$$
 where $P_n(h)$ is isomorphic to $P_n$ as a left module. Moreover  the indecomposable projective $\uq$ are the ones isomorphic to $P_n$.
 \end{lemma}
 
 \begin{proof} Clearly the $P_n(h)$ are left ideals by definition and  have trivial pairwise intersection. The fact that their direct sum is the whole $\uq$ follows by dimension counting. The fact that $P_n(h)$ is isomorphic to $P_n$ is a straightforward verification left to the reader. Since the $P_n$ are summands of the free rank $1$ $\uq$ module, they are projective. Conversely, we easily see that the tensor product $P_a\otimes P_b$ decomposes as a direct sum of modules $P_n$, so any indecomposable summand of a free $\uq$ module is isomorphic to a $P_n$. 
 \end{proof} 
 
 Set 
 $$ \mathbb{B}_{N-1}:= \oplus_{h=0}^{N-1}P_{N-1}(h), \quad \mathbb{B}_n:= \left( \oplus_{h=0}^n P_n(h) \right) \oplus \left( \oplus_{h'=0}^{\overline{n}} P_{\overline{n}}(h')\right).$$
 We obtain a decomposition 
 $$\uq = \mathbb{B}_{N-1}\oplus \oplus_{n=0}^{(N-3)/2} \mathbb{B}_n $$
 where the $\mathbb{B}_n$ are bilateral ideals. So classifying the indecomposable representations of $\uq$ amount to classifying the indecomposable representations of each $\mathbb{B}_n$. Since $\mathbb{B}_{N-1}\cong \Mat_N(\mathbb{C})$, it admits a unique indecomposable representation which corresponds to $S_{N-1}=P_{N-1}$. Fix $0\leq n\leq (N-3)/2$ and consider the algebra
 $$\mathcal{B}:= \End_{\uq} (P_n \oplus P_{\overline{n}}).$$
 
 By Morita theory, the functor $(P_n\oplus P_{\overline{n}})\otimes_{\uq}\bullet : \mathcal{B}^{op}-\Mod \to \mathbb{B}_n-\Mod$ is an equivalence with quasi inverse $\Hom_{\mathbb{B}_n}(P_n\otimes P_{\overline{n}}, \mathbb{B}_n)\otimes_{\mathbb{B}_n} \bullet$. Recall from the proof of Lemma \ref{lemm3} that $\End_{\uq}(P_n)=\mathbb{C}id_{P_n} \oplus \mathbb{C} \pi_n$ where $\pi_n^2=0$. We easily see that $\Hom_{\uq}(P_n, P_{\overline{n}})=\mathbb{C}\alpha_n \oplus \mathbb{C} \beta_n$ where $\alpha_n, \beta_n$ are characterized by the fact that they send the cyclic vector $y_0\in P_n$ to $\alpha_n(y_0)=x_0\in P_{\overline{n}}$ and $\beta_n(y_0)=y_{n+1}\in P_{\overline{n}}$. Note that $\alpha_{\overline{n}}\beta_n= \beta_{\overline{n}}\alpha_n= \pi_n$,  
 so $\mathcal{B}$ is generated by the morphisms $\id_{P_n}, \id_{P_{\overline{n}}}, \alpha_n, \alpha_{\overline{n}}, \beta_n$ and $\beta_{\overline{n}}$. Consider the Kronecker algebra $\mathcal{K}:= \quotient{\mathbb{C}[X,Y]}{(X^2, Y^2)}$ and the embedding $j: \mathcal{B} \hookrightarrow \Mat_2(\mathcal{K})$ defined by 
 $$ j(\id_{P_n})= \begin{pmatrix} 1 & 0 \\ 0&0 \end{pmatrix}, j(\id_{P_{\overline{n}}})= \begin{pmatrix} 0& 0 \\ 0 & 1 \end{pmatrix}, j(\alpha_n)=\begin{pmatrix} 0 & 0 \\ X & 0 \end{pmatrix}, j_{\beta_n} = \begin{pmatrix} 0 & 0 \\ Y & 0 \end{pmatrix}, j(\alpha_{\overline{n}})= \begin{pmatrix} 0 & X \\ 0 & 0 \end{pmatrix}, j(\beta_{\overline{n}})= \begin{pmatrix} 0 & Y \\ 0 & 0 \end{pmatrix}.$$
 The transposition ${}^t : \Mat_2(\mathcal{K}) \to \Mat_2(\mathcal{K})^{op}$ restricts to an isomorphism $t: \mathcal{B} \xrightarrow{\cong} \mathcal{B}^{op}$.  Consider the $\Mat_2(\mathcal{K})-\mathcal{K}$ bimodule $\mathcal{K}^{\oplus 2}$. The functor $(\mathcal{K}^{\oplus 2})\otimes_{\mathcal{K}} \bullet: \mathcal{K}-\Mod \to \Mat_2(\mathcal{K})-\Mod$ is clearly an equivalence. 
\par 
 Consider the involutive automorphism $\Theta: \mathcal{B}\to \mathcal{B}$ defined by the conjugacy by $\begin{pmatrix} 0& 1 \\ 1 &0 \end{pmatrix} \in  \Mat_2(\mathcal{K})$, said differently $\Theta (\id_{P_n})=\id_{P_{\overline{n}}}$, $\Theta (\alpha_n)=\alpha_{\overline{n}}$ and $\Theta(\beta_n)= \beta_{\overline{n}}$. For $L$ a $\mathcal{B}$-module with module morphism $\rho: \mathcal{B}\to \End(L)$, let $\overline{L}$ be the $\mathcal{B}$-module with module structure $\overline{\rho} (x):= \rho(\Theta(x))$, $x\in \mathcal{B}$. Consider the two projectors $P:= \frac{1+\Theta}{2}$, $Q:= \frac{1-\Theta}{2}$  in $\Mat_2(\mathcal{K})$ so that $P+Q=\id$ and $PQ=0$. Given $L \in \Indecomp( \Mat_2(\mathcal{K}))$, we decompose it as $L=L^+\oplus L^-$ where $L^+:= PL$ and $L^-:= QL$ are two indecomposable $\mathcal{B}$ modules such that $L^-= \overline{L^+}$. We denote by $(\cdot)_+, (\cdot)_-: \Indecomp(\Mat_2(\mathcal{K})) \to \Indecomp(\mathcal{B})$ sending $L$ to $L_+$ and $L_-$ respectively. We can now define two maps $F_+, F_-: \Indecomp(\mathcal{K})\to \Indecomp(\mathbb{B}_n)$ as the compositions: 
 \begin{multline*}
 F_{+, -}: \Indecomp(\mathcal{K}) \xrightarrow[\cong]{ (\mathcal{K}^{\oplus 2})\otimes_{\mathcal{K}} \bullet} \Indecomp( \Mat_2(\mathcal{K})) \\
  \xrightarrow{ (\cdot)_{+, -}} \Indecomp(\mathcal{B}) \xrightarrow[\cong]{ t^*} \Indecomp(\mathcal{B}^{op}) \xrightarrow[\cong]{(P_n\otimes P_{\overline{n}})\otimes \bullet} \Indecomp(\mathbb{B}_n), 
 \end{multline*}
 and a map 
 $$ F:= F_+ \sqcup F_- : \Indecomp(\mathcal{K})\sqcup \Indecomp(\mathcal{K}) \to \Indecomp(\mathbb{B}_n).$$
 
 \begin{lemma}\label{lemm8} $ F: \Indecomp(\mathcal{K})\sqcup \Indecomp(\mathcal{K}) \to \Indecomp(\mathbb{B}_n)$ is a bijection.
 \end{lemma}
 
 \begin{proof} We need to prove that $(\cdot)_+ \sqcup (\cdot)_-: \Indecomp(\Mat_2(\mathcal{K}))\sqcup \Indecomp(\Mat_2(\mathcal{K})) \to \Indecomp(\mathcal{B})$ is a bijection or, equivalently, that the map $(\cdot)_{+,-} :\Indecomp(\Mat_2(\mathcal{K})) \to \quotient{\Indecomp(\mathcal{B})}{(M\sim \overline{M})}$ induced by both $(\cdot)_+$ and $(\cdot)_-$, is a bijection.  First, note that $\Mat_2(\mathcal{K})=\mathcal{B}\oplus \mathcal{B}\theta$ where $\theta= \begin{pmatrix} 0& 1 \\ 1 &0 \end{pmatrix}$.
 Let $M \in \Indecomp(\mathcal{B})$ with module structure $\rho: \mathcal{B} \to \End(M)$ and consider the $\Mat_2(\mathcal{K})$ representation $\widehat{\rho}: \Mat_2(\mathcal{K})\to \End(\widehat{M})$ where $\widehat{M}=M\oplus \overline{M}$ and 
 $$ \overline{\rho}(x) = \begin{pmatrix} \rho(x) & 0 \\ 0 & \overline{\rho}(x) \end{pmatrix}, \quad \overline{\rho}(x\theta)= \begin{pmatrix} 0 & \rho(x) \\  \overline{\rho}(x) & 0 \end{pmatrix}, \quad \mbox{ for }x\in \mathcal{B}.
 $$
 Then $\widehat{M} \in \Indecomp(\Mat_2(\mathcal{K}))$ and $(\widehat{M})_{+}=M$, $(\widehat{M})_-=\overline{M}$ so the application $M \mapsto \widehat{M}$ is the inverse of $(\cdot)_{+,-}$ which is thus a bijection.

 \end{proof}

It remains to prove that $\Indecomp(\mathcal{K})$ is in bijection with $\Delta$. This classification was accomplished by Kronecker in \cite{Kronecker_KroneckerAlgebra} and is easier to state using quiver representations in the framework of \cite{Kac_QuiverRep}. First, note that if $\rho: \Indecomp(\mathcal{K})\to \End(V)$ is an indecomposable representation with $\rho(XY)\neq 0$, then there exists $v\in V$ with $\rho(XY)v\neq 0$ and the subspace $\Span(v, \rho(X)v, \rho(Y)v, \rho(XY)v)\subset V$ is a $\mathcal{K}$ submodule isomorphic to $\mathcal{K}$ (seen as left module over itself). Since this module is free, it is projective and thus $V\cong \mathcal{K}$. 
In order to make the two modules $F_+(\mathcal{K})$ and $F_-(\mathcal{K})$ explicit, note that by Lemma \ref{lemm6}, the projection map $P_n\to S_n$ given by Lemma \ref{lemm1} (sending $y_0$ to $e_0$) is a  projective cover, so 
$$ \Ext^1(S_n, S_{\overline{n}})= \Hom_{\uq}(P_n, P_{\overline{n}})= \mathbb{C}\alpha_n \oplus \mathbb{C} \beta_n.$$
The extension of $S_n$ by $S_{\overline{n}}$ corresponding to $z_1 \alpha_n + z_2 \beta_n$ is the module $M_n(z_1,z_2)$ with basis $(v_i, i=0, \ldots N-1)$ and module structure:
\begin{align*}
{}& k^{1/2}v_i= A^{n-2i}v_i, Fv_{N-1}=0, F v_n= z_2 v_{n+1}, Fv_i=v_{i+1} \mbox{ for }i\neq n, N-1, \\ {}& Ev_0=z_1 v_{N-1}, Ev_i= [i][n-i+1]v_{i-1} \mbox{ for }i\neq 0.
\end{align*}
The extension is given by  the exact sequence
$$ 0 \to S_{\overline{n}} \xrightarrow{i} M_n(z_1,z_2) \xrightarrow{p} S_n \to 0$$
where $i(e_j)=v_{n+j+1}$, $p(v_i)=e_i$ if $i\leq n$ and $p(v_i)=0$ else. Note that $M_n(0,0)=S_{\overline{n}}\oplus S_n$ and that, for $(z_1,z_2)\neq (0,0)$, then $M_n(z_1,z_2)\cong M_n^1(\lambda)$ where $\lambda=z_1:z_2 \in \mathbb{CP}^1$.
We denote by $S_n \xrightarrow{ z_1 \alpha_n + z_2\beta_n} S_{\overline{n}}$ the extension $M_n(z_1,z_2)$. Then $\mathcal{K}$, seen as a $\mathcal{K}$ module over itself, decomposes as 

$$
\mathcal{K}=
\begin{tikzcd}
{} & \mathbb{C} \ar[ld, "X"] \ar[rd, "Y"'] & {} \\
\mathbb{C}X \ar[rd, "Y"] & {} & \mathbb{C}Y \ar[ld, "X"'] \\
{} & \mathbb{C} XY & {}
\end{tikzcd},  
$$
$$
\mbox{ so }
F_+(\mathcal{K})= 
\begin{tikzcd}
{} & S_n \ar[ld, "\alpha_n"] \ar[rd, "\beta_n"'] & {} \\
S_{\overline{n}} \ar[rd, "\beta_{\overline{n}}"] & {} & S_{\overline{n}} \ar[ld, "\alpha_{\overline{n}}"'] \\
{} & S_n & {}
\end{tikzcd} 
=P_n 
\mbox{ and }
F_-(\mathcal{K})= 
\begin{tikzcd}
{} & S_{\overline{n}} \ar[ld, "\alpha_{\overline{n}}"] \ar[rd, "\beta_{\overline{n}}"] & {} \\
S_n \ar[rd, "\beta_n"] & {} & S_n \ar[ld, "\alpha_n"] \\
{} & S_n & {}
\end{tikzcd} 
=P_{\overline{n}}.
$$

We now classify the representations of the quotient algebra $\quotient{\mathcal{K}}{(XY)}$. Consider the Kronecker quiver 
$Q:= \begin{tikzcd}
a \ar[r, bend left,  "X"] \ar[r, bend right, "Y"] &  b \end{tikzcd}
$ and let $\mathbb{C}Q$ be its path algebra; it is generated by the constant paths $1_a, 1_b$ and the two paths $X,Y$. Consider the embedding $i: \quotient{\mathcal{K}}{(XY)}\hookrightarrow \mathbb{C}Q$  defined by $i(X)=X$, $i(Y)=Y$ and $i(1)=1_a+1_b$. Recall that a representation $(\rho, V)$ of $Q$ is the data of two vector spaces $V_a$ and $V_b$ and two linear morphisms $\rho(X), \rho(Y): V_a \to V_b$ depicted as $\begin{tikzcd}
V_a \ar[r, bend left,  "\rho(X)"] \ar[r, bend right, "\rho(Y)"] &  V_b \end{tikzcd}
$ and that such a representation induces a $\mathbb{C}Q$-module structure on $V_a\oplus V_b$ in the obvious way such that one obtains an equivalence $\mathbb{C}Q-\Mod \cong \Rep(Q)$ (see \cite{Kac_QuiverRep} for details). The set of positive roots of $Q$ 
 is $\Delta_+(\widehat{\mathfrak{sl}}_2)=\{ (n,m) \in \mathbb{N}^2, (n-m)^2\leq 1\}$, the real positive  roots are the $(n,m)$ for which $(n-m)^2=1$ and the imaginary positive roots are the $(n,n), n\geq 1$. If $(\rho, V)$ is an indecomposable representation of $Q$, then the pair $(\dim(V_a), \dim(V_b))$ is a positive root and every positive roots (distinct from $(0,0)$) arise that way (see \cite{Kac_QuiverRep}). To each real positive root corresponds exactly one indecomposable representation and each imaginary positive root corresponds to several indecomposable representations: for the quiver $Q$, one has a moduli space of such representations parametrized by $\mathbb{CP}^1$. Let us now describe the indecomposable representations of $\mathcal{K}$, and thus of $\mathbb{B}_n$. 
 \par $(1)$ The real roots $(1,0)$ and $(0,1)$ correspond to the representations 
 $\begin{tikzcd} \mathbb{C} \ar[r, bend left,  "0"] \ar[r, bend right, "0"] &  0 \end{tikzcd}$ and  $\begin{tikzcd} 0 \ar[r, bend left,  "0"] \ar[r, bend right, "0"] &  \mathbb{C} \end{tikzcd}$ both induces the trivial representation $\mathbb{C}$ of $\mathcal{K}$ which itself induces the modules $F_+(\mathbb{C})=S_n$ and $F_-(\mathbb{C})=S_{\overline{n}}$.
 \par $(2)$ For $k\geq 1$, the real positive root $(k+1,k)$ corresponds to the representation 
 $$\begin{tikzcd} \mathbb{C}^{k+1} \ar[r, bend left,  "M_1"] \ar[r, bend right, "M_2"] &  \mathbb{C}^k \end{tikzcd}, \quad
  M_1= \begin{pmatrix} 1 & {} & 0 & 0\\ {} & \ddots &{} &  \vdots  \\ 0 & {} & 1 & 0 & \end{pmatrix}
 , \quad M_2=
 \begin{pmatrix} 0 & 1 & {} & 0 \\ \vdots & {} & \ddots & {} \\ 0 & 0 & {} & 1 \end{pmatrix}
 .$$
  The image by $F_+$ of this representation is the extension
  $$ \begin{tikzcd}
  S_n \ar[rd, "\beta_n"] & {} & S_n \ar[ld, "\alpha_n"] \ar[rd, "\beta_n"] & {} &S_n \ar[ld, "\alpha_n"]  & \ldots  \ar[rd, "\beta_n"] & {} &  S_n \ar[ld, "\alpha_n"] \\
   {} & S_{\overline{n}} & {} & S_{\overline{n}} & {} & \ldots & S_{\overline{n}} & {}
   \end{tikzcd}$$
 which is $\Omega^k_n$. Similarly, the image by $F_-$ is $\Omega^k_{\overline{n}}$.

 \par $(3)$ For $k\geq 1$, the real positive root $(k,k+1)$ corresponds to the representation 
 $$\begin{tikzcd} \mathbb{C}^{k} \ar[r, bend left,  "M_1"] \ar[r, bend right, "M_2"] &  \mathbb{C}^{k+1} \end{tikzcd}, \quad
  M_1= \begin{pmatrix} 1 & {} & 0 \\ {} & \ddots & {} \\ 0 & {} & 1 \\ 0 & \ldots & 0 \end{pmatrix}
 , \quad M_2=
 \begin{pmatrix}
 0 & \ldots & 0 \\ 1 & {} & 0 \\ {} & \ddots & {} \\ 0 & {} & 1 
 \end{pmatrix}
  .$$
   The image by $F_+$ of this representation is the extension
  $$ \begin{tikzcd}
  {} & S_n \ar[rd, "\beta_n"] \ar[ld, "\alpha_n"]& {} & S_n \ar[ld, "\alpha_n"] \ar[rd, "\beta_n"] & \ldots &  S_n \ar[rd, "\alpha_n"] &{} \\
    S_{\overline{n}} & {} & S_{\overline{n}} & {} & S_{\overline{n}} & \ldots & S_{\overline{n}} 
   \end{tikzcd}$$
 which is $\Omega^{-k}_n$. Similarly, the image by $F_-$ is $\Omega^{-k}_{\overline{n}}$.
 
 \par $(4)$ For $k\geq 1$ the imaginary positive root $(k,k)$ is the dimension of a family of indecomposable representations indexed by $\lambda \in \mathbb{CP}^1= \mathbb{C}\cup \{\infty\}$ defined by 
 $$\begin{tikzcd} \mathbb{C}^{k} \ar[r, bend left,  "\mathbf{1}_n"] \ar[r, bend right, "J_n(\lambda)"] &  \mathbb{C}^{k} \end{tikzcd} \mbox{, if }\lambda \in \mathbb{C}; \begin{tikzcd} \mathbb{C}^{k} \ar[r, bend left,  "J_n(0)"] \ar[r, bend right, "\mathbf{1}_n"] &  \mathbb{C}^{k} \end{tikzcd} \mbox{, if }\lambda= \infty\mbox{, where }
 J_n(\lambda)= \begin{pmatrix}
 \lambda & 1  & {} & 0 \\
 {} & \ddots & \ddots & {} \\
 {} & {} & \ddots & 1 \\
 0 & {} & {} & \lambda 
 \end{pmatrix}
  .$$
 The image by $F_+$ of this representation is the extension
 
 $$ \begin{tikzcd}
 S_n \ar[rrd, "\beta_n"] \ar[d, "\alpha_n +\lambda \beta_n"] &{}&  S_n \ar[rrd, "\beta_n"] \ar[d, "\alpha_n +\lambda \beta_n"] &{}& S_n\ar[d, "\beta_n"] \ar[rd, "\beta_n"] & \ldots & S_n \ar[d, "\beta_n"] \\
 S_{\overline{n}} &{}& S_{\overline{n}} &{}& S_{\overline{n}} & \ldots & S_{\overline{n}} 
 \end{tikzcd}$$
 if $\lambda \in \mathbb{C}$ and 
 
 $$ \begin{tikzcd}
 S_n \ar[rd, "\alpha_n"] \ar[d, "\beta_n"] & S_n \ar[rd, "\alpha_n"] \ar[d, "\beta_n"] & S_n\ar[d, "\beta_n"] \ar[rd, "\alpha_n"]& \ldots & S_n \ar[d, "\beta_n"] \\
 S_{\overline{n}} & S_{\overline{n}} & S_{\overline{n}} & \ldots & S_{\overline{n}} 
 \end{tikzcd}$$ 
 if $\lambda = \infty$
 which is $M_n^k(\lambda)$. Similarly, the image by $F_-$ is $M_{\overline{n}}^k(\lambda)$.

 \begin{proof}[Proof of Proposition \ref{prop_uq}]
 The classification of indecomposable $\uq$ modules follows from Lemma \ref{lemm8} and the above classification of $\Indecomp(\mathcal{K})$. The fact that the $P_n$ are the unique projective indecomposable was proved in Lemma \ref{lemm6}. For $0\leq n \leq N-2$, then the modules $P_n$, $\Omega^{-k}_n$, $\Omega^{k}_n$ and $M_n^k (\lambda)$ admit the proper submodules $S_{\overline{n}}$, $S_{\overline{n}}^{\oplus k+1}$, $S_{\overline{n}}^{\oplus k}$ and $S_{\overline{n}}^{\oplus k}$ respectively, so they are not simple and the $S_n$ are the unique simple $\uq$ modules.  
 
 \end{proof}

 \begin{proof}[Proof of Theorem \ref{theorem_representations_QG}]
 Theorem \ref{theorem_representations_QG} is a consequence of Lemmas  \ref{lemm5}, \ref{lemm6} and Proposition \ref{prop_uq}.
 \end{proof}

 \section{Reduced stated skein algebras of 2P-marked surfaces and quantum cluster algebras} \label{sec_cluster_algebras}

 A \textbf{2P-marked surface} is a connected marked surface $\mathbf{\Sigma}$ which has  at least two boundary edges and no inner puncture (i.e. every boundary component contains at least one boundary edge). As we shall see, the reduced stated skein algebra $\overline{\mathcal{S}}_A(\mathbf{\Sigma})$ of a 2P-marked surface is isomorphic to a quantum cluster algebra $\mathcal{A}_q(\mathbf{\Sigma})$ whose Azumaya loci have been studied in \cite{MNTY_AzumayaClusterAlgebras}. We will prove that  these algebras are Azumaya. To relate these two algebras, we first introduce the \textbf{localized M\"uller algebra} $\mathcal{M}^0_A(\mathbf{\Sigma})$ following \cite{Muller} and define two isomorphisms $\mathcal{M}^0_A(\mathbf{\Sigma}) \cong \overline{\mathcal{S}}_A(\mathbf{\Sigma})$ and $\mathcal{A}_q(\mathbf{\Sigma}) \cong \mathcal{M}^0_A(\mathbf{\Sigma})$.

\subsection{M\"uller algebras and their localizations}\label{sec_Muller}

During this subsection and the next one, we denote by $k$ an arbitrary commutative unital ring and $A^{1/2}\in k^{\times}$ an invertible element (not necessarily a root of unity).

\begin{definition}[M\"uller skein algebras](\cite{Muller})\label{def_Muller} Let $\mathbf{\Sigma}$ be an essential marked surface. 
\begin{enumerate}
\item The \textbf{M\"uller algebra} $\mathcal{M}_A(\mathbf{\Sigma})$ is the quotient of the $k$-free module generated by isotopy classes of (non stated) tangles $T\subset \mathbf{\Sigma}\times (0,1)$ by the skein relations 
\begin{equation*} 
\begin{tikzpicture}[baseline=-0.4ex,scale=0.5,>=stealth]	
\draw [fill=gray!45,gray!45] (-.6,-.6)  rectangle (.6,.6)   ;
\draw[line width=1.2,-] (-0.4,-0.52) -- (.4,.53);
\draw[line width=1.2,-] (0.4,-0.52) -- (0.1,-0.12);
\draw[line width=1.2,-] (-0.1,0.12) -- (-.4,.53);
\end{tikzpicture}
=A
\begin{tikzpicture}[baseline=-0.4ex,scale=0.5,>=stealth] 
\draw [fill=gray!45,gray!45] (-.6,-.6)  rectangle (.6,.6)   ;
\draw[line width=1.2] (-0.4,-0.52) ..controls +(.3,.5).. (-.4,.53);
\draw[line width=1.2] (0.4,-0.52) ..controls +(-.3,.5).. (.4,.53);
\end{tikzpicture}
+A^{-1}
\begin{tikzpicture}[baseline=-0.4ex,scale=0.5,rotate=90]	
\draw [fill=gray!45,gray!45] (-.6,-.6)  rectangle (.6,.6)   ;
\draw[line width=1.2] (-0.4,-0.52) ..controls +(.3,.5).. (-.4,.53);
\draw[line width=1.2] (0.4,-0.52) ..controls +(-.3,.5).. (.4,.53);
\end{tikzpicture}
\quad 
\begin{tikzpicture}[baseline=-0.4ex,scale=0.5,rotate=90] 
\draw [fill=gray!45,gray!45] (-.6,-.6)  rectangle (.6,.6)   ;
\draw[line width=1.2,black] (0,0)  circle (.4)   ;
\end{tikzpicture}
= -(A^2+A^{-2}) 
\begin{tikzpicture}[baseline=-0.4ex,scale=0.5,rotate=90] 
\draw [fill=gray!45,gray!45] (-.6,-.6)  rectangle (.6,.6)   ;
\end{tikzpicture}
;
\end{equation*}

\begin{equation*}
\begin{tikzpicture}[baseline=-0.4ex,scale=0.5,>=stealth]
\draw [fill=gray!45,gray!45] (-.7,-.75)  rectangle (.4,.75)   ;
\draw[->] (0.4,-0.75) to (.4,.75);
\draw[line width=1.2] (0.4,-0.3) to (0,-.3);
\draw[line width=1.2] (0.4,0.3) to (0,.3);
\draw[line width=1.1] (0,0) ++(90:.3) arc (90:270:.3);
\end{tikzpicture}
=0; 
\quad 
\heightexch{->}{}{} = A \heightexch{<-}{}{}
\end{equation*}
The multiplication in $\mathcal{M}_A(\mathbf{\Sigma})$ is defined by stacking the tangles as in Definition \ref{def_skein}.
\item For $p\in \mathcal{P}^{\partial}$, let $\alpha(p) \in \mathcal{M}(\mathbf{\Sigma})$ be the class of the corner arc. The \textbf{localized M\"uller algebra} $\mathcal{M}^0_A(\mathbf{\Sigma}):= \mathcal{M}_A(\mathbf{\Sigma})[\alpha_p^{-1}, p\in \mathcal{P}^{\partial}]$ is the algebra obtained from $\mathcal{M}_A(\mathbf{\Sigma})$ by formally inverting the elements $\alpha(p)$.
\end{enumerate}
\end{definition}

A diagram $D$ defines an element $[D] \in \mathcal{M}^0_A(\mathbf{\Sigma})$ and $\mathcal{M}^0_A(\mathbf{\Sigma})$ is spanned by elements of the form $[D]\alpha(p_1)^{-k_1}\ldots \alpha(p_n)^{-k_n}$ where $\mathcal{P}^{\partial}=\{p_1, \ldots, p_n\}$ and $k_i \geq 0$. An element $[D]\alpha(p_1)^{-k_1}\ldots \alpha(p_n)^{-k_n}$ is called \textit{reduced} if for each $p_i \in \mathcal{P}^{\partial}$, either $k_i=0$ or $D$ does not contain any component isotopic to $\alpha(p_i)$. 

\begin{definition}\label{def_Muller_basis2}  Let $\underline{\mathcal{B}}^M \subset \mathcal{M}^0_A(\mathbf{\Sigma})$ be the subset of reduced elements $[D]\alpha(p_1)^{-k_1}\ldots \alpha(p_n)^{-k_n}$ such that $D$ is simple.\end{definition}

\begin{proposition}\label{prop_Muller_basis2}(\cite[Proposition $5.3$]{Muller}) $\underline{\mathcal{B}}^M$ is a basis of $\mathcal{M}^0_A(\mathbf{\Sigma})$.
\end{proposition}

As remarked in \cite{LeStatedSkein}, by comparing the skein relations in Definitions \ref{def_skein} and \ref{def_Muller}, we see that  there is an algebra morphism $f: \mathcal{M}_A(\mathbf{\Sigma})\to \overline{\mathcal{S}}_A(\mathbf{\Sigma})$ sending the class $[T]$ of a tangle to $f([T]):= [T, s^+]$ where $s^+$ sends every points of $\partial T$ to $+$. Since $\alpha(p)_{++}$ has inverse $\alpha(p)_{--}$, $f$ extends to a morphism (denoted by the same letter) $f: \mathcal{M}^0_A(\mathbf{\Sigma})\to \overline{\mathcal{S}}_A(\mathbf{\Sigma})$ sending $\alpha(p)^{-1}$ to $\alpha(p)_{--}$. 

\begin{theorem}\label{theorem_SSkein_Muller}(\cite[Theorem $5.2$]{LeYu_SSkeinQTraces}) $f: \mathcal{M}^0_A(\mathbf{\Sigma})\to \overline{\mathcal{S}}_A(\mathbf{\Sigma})$ is an isomorphism.
\end{theorem}

\begin{remark} 
By comparing Definitions \ref{def_Muller_basis2} and \ref{def_Muller_basis1} and using Lemma \ref{lemma_heightexch}, we see that for each $b \in \underline{\mathcal{B}}^M$ there exists $n_b \in \mathbb{Z}$ such that $f$ sends the basis $\{ A^{n_b/2} b , b \in \underline{\mathcal{B}}^M\}$ to the basis $\mathcal{B}^M$
so the theorem follows from Propositions \ref{prop_Muller_basis} and \ref{prop_Muller_basis2}.
\end{remark}

\subsection{Quantum cluster algebras}\label{sec_QCA}

Quantum cluster algebras were introduced in \cite{BerensteinZelevinsky} after the introduction of classical cluster algebras in \cite{FominZelevinsky} and their Poisson structure in \cite{GekhtmanShapiroVainshtein03}. We will closely follow the notations in \cite{Muller}.

\begin{definition}[Quantum cluster algebras and their upper algebras]
Let $\mathcal{F}$ be a skew field with $k\subset \mathcal{F}$ and such that $\mathcal{F}$ is purely transcendental  over $k$.
\begin{enumerate}
\item A \textbf{quantum seed} in $\mathcal{F}$ is a triple $\mathbb{S}=(B, \Lambda, M)$ where:
\begin{itemize}
\item the \textbf{exchange matrix} $B$ is an $N\times ex$ $\mathbb{Z}$-valued matrix where $ex\subset \{1, \ldots, N\}$ such that, if $\pi: \mathbb{Z}^N \to \mathbb{Z}^{ex}$ denotes the projection, then $\pi B$ is skew-symmetric; 
\item the \textbf{compatibility matrix} $\Lambda$ is an $N\times N$ skew symmetric matrix valued in $\mathbb{Z}$ such that $\Lambda B= D \iota$ where $D$ is diagonal matrix with positive diagonal elements and $\iota$ is the matrix of the inclusion $\mathbb{Z}^{ex} \hookrightarrow \mathbb{Z}^N$; 
\item $M: \mathbb{Z}^N \to \mathcal{F} \setminus\{ 0 \}$ is a map such that $M(\alpha)M(\beta)= A^{\frac{1}{2} \Lambda(\alpha, \beta)} M(\alpha + \beta)$ and such that $\mathcal{F}$ is the field of fractions of the quantum torus $k M(\mathbb{Z}^N)$. 
\end{itemize}
\item For $\mathbb{S}=(B, \Lambda, M)$ a quantum seed and $i \in ex$, the \textbf{mutation of }$\mathbb{S}$ \textbf{ in the direction of }$i$ is the quantum seed $\mathbb{S}'=(B', \Lambda', M')$ characterized by the following properties:
\begin{itemize}
\item $B'_{jk} = \left\{
 \begin{array}{ll}
-B_{jk} & \mbox{, if }i=j \mbox{ or }i=k; \\
 B_{jk} + \frac{1}{2}( |B_{ji}| B_{ik} + B_{ji} |B_{ik}|) & \mbox{, else.}
\end{array} \right.$;
\item for $\alpha \in \mathbb{Z}^N$ such that $\alpha_i=0$ then $M(\alpha)=M'(\alpha)$; 
\item one has 
$$ M'(e_i) = M \left( -e_i + \sum_{B_{ji}>0} B_{ji}e_j \right) + M\left( -e_i - \sum_{B_{ji}<0} B_{ji}e_j \right).$$
\end{itemize}
By \cite[Section $4.4$]{BerensteinZelevinsky},  $\mathbb{S}'$ exists and is unique. We denote by $\Mut(\mathbb{S})$ the set of quantum seeds which can be obtained from $\mathbb{S}$ by a finite sequence of mutations. 
\item The \textbf{quantum cluster algebra} $\mathcal{A}_q(\mathbb{S})$ is the $k$ subalgebra of $\mathcal{F}$ generated by elements $M'(\alpha)$ for  $\mathbb{S}'=(B',\Lambda',M')\in \Mut(\mathbb{S})$ and $\alpha \in \mathbb{Z}^N$ such that $\alpha_i\geq 0$ for $i\in ex$ . 
\item The \textbf{quantum upper cluster algebra} $\mathcal{U}_q(\mathbb{S})$ is the intersection: 
$$ \mathcal{U}_q(\mathbb{S}) := \bigcap_{\mathbb{S}'=(B', \Lambda', M')\in \Mut(\mathbb{S})} k M'(\mathbb{Z}^N).$$
\end{enumerate}
\end{definition}
By \cite[Corollary $5.2$]{BerensteinZelevinsky}, one has $\mathcal{A}_q(\mathbb{S}) \subset \mathcal{U}_q(\mathbb{S})$ (this is the so-called Laurent phenomenon). In many cases, this inclusion is an equality. One can associate to any triangulated marked surface $(\mathbf{\Sigma}, \Delta)$ a quantum seed. This was first remarked at the classical level in \cite{GekhtmanShapiroVainshtein05} and extended to the quantum case in \cite{Muller}.

\begin{definition}[Seeds associated to triangulated marked surfaces]
\begin{enumerate}
\item A connected marked surface $\mathbf{\Sigma}$ is \textbf{triangulable} if it is essential, has no inner puncture  and is not a disc with one or two boundary edges. For such a $\mathbf{\Sigma}$, an $\mathcal{A}$-\textbf{triangulation} $\Delta$ is a maximal set of pairwise disjoint and pairwise non isotopic arcs with an indexation $\Delta= \{x_1, \ldots, x_{|\Delta|}\}$ (recall from Definition \ref{def_tangles} than an arc is an open connected diagram without crossing). Clearly every corner arc $\alpha(p)$ belongs to $\Delta$ and we denote by $\mathring{\Delta}\subset \Delta$ the set of arcs which are not corner arcs and by $ex \subset \{1, \ldots, |\Delta|\}$ the subset of indices $i$ such that $x_i \in \mathring{\Delta}$. For $x_i \in \Delta$, we denote by $\partial_1 x_i$ and $\partial_2 x_i$ its two endpoints (indexed arbitrarily). For two endpoints $\partial_i x$ and $\partial_j y$ which belong to the same boundary edge $a$ (where $x,y \in \Delta$, $i,j \in \{1,2\}$) such that $\partial_ix <_{\mathfrak{o}^+} \partial_j y$, we say that $\partial_i x$ and $\partial_j y$ are \textbf{consecutive} if there does not exist $z\in \Delta$ and $k\in \{1,2\}$ such that  $\partial_ix <_{\mathfrak{o}^+} \partial_k z <_{\mathfrak{o}^+} \partial_j y$.
\item Let $(\mathbf{\Sigma}, \Delta)$ be a triangulated marked surface and  $\mathcal{F}_{\mathbf{\Sigma}}$ the field of fraction of $\mathcal{M}_A(\mathbf{\Sigma})$ (obtained by inverting every non zero elements). Let $\mathbb{S}_{(\mathbf{\Sigma}, \Delta)}= (B^{\Delta}, \Lambda^{\Delta}, M^{\Delta})$ be the quantum seed over $\mathcal{F}_{\mathbf{\Sigma}}$ defined by: 
\begin{itemize}
\item For $i \in \{1, \ldots, |\Delta|\}$ and $j\in ex$,  
$$ B^{\Delta}_{i,j}:= \sum_{u,v \in \{1,2\}} 
\left\{ \begin{array}{ll}
-1 & \mbox{, if }\partial_u x_i \mbox{ and }\partial_v x_j \mbox{ are in the same boundary edge }  \mbox{ and } \partial_u x_i <_{\mathfrak{o}^+} \partial_v x_j \mbox{ are consecutive}; \\
+1 & \mbox{, if }\partial_u x_i \mbox{ and }\partial_v x_j \mbox{ are in the same boundary edge }  \mbox{ and } \partial_v x_j <_{\mathfrak{o}^+} \partial_u x_i \mbox{ are consecutive};  \\
0 & \mbox{, if } \partial_u x_i \mbox{ and }\partial_v x_j \mbox{ else.} 
\end{array} \right. $$
\item  For $i, j  \in \{1, \ldots, |\Delta|\}$,  
$$ \Lambda^{\Delta}_{i,j}:= \sum_{u,v \in \{1,2\}} 
\left\{ \begin{array}{ll}
0 & \mbox{, if } \partial_u x_i \mbox{ and }\partial_v x_j \mbox{ are in two distinct boundary edges;} \\
+1 & \mbox{, if }\partial_u x_i \mbox{ and }\partial_v x_j \mbox{ are in the same boundary edge }  \mbox{ and } \partial_u x_i <_{\mathfrak{o}^+} \partial_v x_j; \\
-1 & \mbox{, if }\partial_u x_i \mbox{ and }\partial_v x_j \mbox{ are in the same boundary edge }  \mbox{ and } \partial_v x_j <_{\mathfrak{o}^+} \partial_u x_i. 
\end{array} \right. $$
\item $M^{\Delta}(e_i):= x_i$ for $i\in \{1, \ldots, |\Delta|\}$.
\end{itemize}
By \cite[Proposition $7.8$]{Muller}, $\mathbb{S}_{(\mathbf{\Sigma}, \Delta)}$ is indeed a quantum seed.
\item By \cite[Theorem $7.9$]{Muller}, one has $\Mut(\mathbb{S}_{(\mathbf{\Sigma}, \Delta)}) = \{ \mathbb{S}_{(\mathbf{\Sigma}, \Delta')}, \Delta' \mbox{ an }\mathcal{A}-\mbox{ triangulation of }\mathbf{\Sigma} \}$. Therefore $\mathcal{A}_q(\mathbb{S}_{(\mathbf{\Sigma}, \Delta)})$ and $\mathcal{U}_q(\mathbb{S}_{(\mathbf{\Sigma}, \Delta)})$ only depend on $\mathbf{\Sigma}$ and not on $\Delta$; we denote them by $\mathcal{A}_q(\mathbf{\Sigma})$ and $\mathcal{U}_q(\mathbf{\Sigma})$ respectively.
\end{enumerate}
\end{definition}

\begin{theorem}\label{theorem_QCA_Muller} Let $(\mathbf{\Sigma}, \Delta)$ be a triangulated marked surface. 
\begin{enumerate}
\item \cite[Theorem $7.16$]{Muller} We have the inclusions $\mathcal{A}_q(\mathbf{\Sigma}) \subset \mathcal{M}^0_A (\mathbf{\Sigma}) \subset \mathcal{U}_q(\mathbf{\Sigma})$.
\item \cite[Theorem $9.7$]{Muller} If moreover $\mathbf{\Sigma}$ is a 2P-marked surface, then $\mathcal{A}_q(\mathbf{\Sigma}) = \mathcal{M}^0_A (\mathbf{\Sigma}) =\mathcal{U}_q(\mathbf{\Sigma})$.
\end{enumerate}
\end{theorem}
 
 Therefore, for a 2P-marked surface $\mathbf{\Sigma}$ which is not the bigon,  using Theorems \ref{theorem_SSkein_Muller} and \ref{theorem_QCA_Muller}, we have an isomorphism $f: \overline{\mathcal{S}}_A(\mathbf{\Sigma}) \xrightarrow{\cong} \mathcal{A}_q(\mathbf{\Sigma})$. 
 
 \subsection{Representations of reduced stated skein algebras of 2P-surfaces}
 
 For now on, we suppose that $k=\mathbb{C}$ and that $A^{1/2}\in \mathbb{C}^*$ is a root of unity of odd order $N\geq 3$. So $\mathcal{A}_q(\mathbb{S})$ is now a complex algebra.
 
 \begin{definition}[Frobenius morphisms of quantum cluster algebras] Suppose that the field $\mathcal{F}$ is purely transcendental over $\mathbb{C}$ and let $\mathbb{S}= (B, \Lambda, M)$ a quantum seed for $\mathcal{F}$. Let $Z_{\mathbb{S}}$ denote the center of $\mathcal{A}_q(\mathbb{S})$. Let $\mathcal{A}_{+1}(\mathbb{S})$ be the classical cluster algebra where we set $A^{1/2}=+1$.
 \begin{enumerate}
 \item The \textbf{Frobenius morphism} $Fr_{\mathbb{S}}: \mathcal{A}_{+1}(\mathbb{S}) \hookrightarrow Z_{\mathbb{S}}$ is the embedding defined by $Fr_{\mathbb{S}}(M(\alpha))= M(N \alpha)$. 
 \item We write $\widehat{X}(\mathbb{S}):= \Specm (Z_{\mathbb{S}})$ and $X(\mathbb{S}):= \Specm(\mathcal{A}_{+1}(\mathbb{S}))$ and denote by $\pi: \widehat{X}(\mathbb{S}) \to X(\mathbb{S})$ the dominant map defined by $Fr_{\mathbb{S}}$. 
 \end{enumerate}
 \end{definition}
 Note that 
 $$Fr_{\mathbb{S}}(M(\alpha))M(\beta)=M(N\alpha) M(\beta)=(A^{1/2})^N M(N\alpha + \beta)= M(\beta)= M(\beta+ N\alpha)= M(\beta)Fr_{\mathbb{S}}(M(\alpha))$$ so the image of the Frobenius is indeed central.
Since $\mathcal{A}_q(\mathbb{S})\subset \mathcal{F}$ is a subalgebra of a field, it is a domain. It is affine and is finitely generated over the image of $Fr_{\mathbb{S}}$ so it is almost Azumaya. We denote by 
$\mathcal{AL}(\mathbb{S}) \subset \widehat{X}(\mathbb{S})$ its Azumaya locus  and by  $\mathcal{FAL}(\mathbb{S}) \subset X(\mathbb{S})$ its fully Azumaya locus (i.e. the set of $x\in X(\mathbb{S})$ such that $\pi^{-1}(x)\subset \mathcal{AL}(\mathbb{S})$).

 \begin{theorem}\label{theorem_MNKY}(\cite[Theorem $6.1$]{MNTY_AzumayaClusterAlgebras}) Suppose that the skew field $\mathcal{F}$ is a purely transcendental over $\mathbb{C}$ and let $\mathbb{S}= (B, \Lambda, M)$ a quantum seed for $\mathcal{F}$. 
 Suppose that 
 \begin{enumerate}
 \item  $N$ is prime to the diagonal entries of the diagonal matrix $D$ defined by $\Lambda B = D \iota$ (recall that $\iota$ is the matrix of the inclusion $\mathbb{Z}^{ex} \hookrightarrow \mathbb{Z}^N$) and 
 \item one has $\mathcal{A}_q(\mathbb{S})=\mathcal{U}_q(\mathbb{S})$.
 \end{enumerate}
 Then the fully Azumaya locus of $\mathcal{A}_q(\mathbb{S})$ contains the regular locus of $X(\mathbb{S})$. In particular, if $X(\mathbb{S})$ is smooth, then  $\mathcal{A}_q(\mathbb{S})$ is Azumaya. 
 \end{theorem}
 
 Theorem \ref{theorem_MNKY} permits to prove Theorem \ref{main_theorem_intro} for 2P-marked surfaces. We will also use the following: 
 
 \begin{theorem}( \cite[Theorem $1$]{KojuMCGRepQT})\label{theorem_MCG} If $\mathbf{\Sigma}= (\Sigma_{g,1}, \{a\})$ is a connected marked surface with a single boundary component which has a single boundary edge, then $\overline{\mathcal{S}}_A(\mathbf{\Sigma})$ is Azumaya.
 \end{theorem}

 \begin{corollary}\label{coro_2Psurfaces}
 For $\mathbf{\Sigma}$ a connected essential marked surface without inner puncture, then $\overline{\mathcal{S}}_A(\mathbf{\Sigma})$ is Azumaya and every semi-weight representation is a weight representation.
 \end{corollary}
 
 \begin{proof} Let us  prove that  $\overline{\mathcal{S}}_A(\mathbf{\Sigma})$ is Azumaya. The second assertion will then follow from Theorem \ref{theorem_FAL}.
  If $\mathbf{\Sigma}=\mathbb{B}$ is the bigon (a disc with two boundary edges), then $\overline{\mathcal{S}}_A(\mathbb{B})\cong \mathbb{C}[X^{\pm 1}]$ is commutative, so the result is obvious.
 If $\mathbf{\Sigma}= (\Sigma_{g,1}, \{a\})$ has a single boundary component which has a single boundary edge, it is Theorem \ref{theorem_MCG}.
  Else, $\mathbf{\Sigma}$ is a triangulable 2P marked surface; let $\Delta$ be an  $\mathcal{A}$-triangulation. The isomorphism  $f: \overline{\mathcal{S}}_{+1}(\mathbf{\Sigma}) \xrightarrow{\cong} \mathcal{A}_{+1}(\mathbf{\Sigma})$ induces an isomorphism $X(\mathbf{\Sigma}) \cong X(\mathbb{S}_{(\mathbf{\Sigma}, \Delta)})$ so Theorem \ref{theorem_smooth} implies that $X(\mathbb{S}_{(\mathbf{\Sigma}, \Delta)})$ is smooth and it suffices to prove that the two hypotheses of Theorem \ref{theorem_MNKY} are satisfied for $\mathbb{S}_{(\mathbf{\Sigma}, \Delta)}$. In the proof of \cite[Proposition $7.8$]{Muller}, it is proved that $\Lambda^{\Delta} B^{\Delta}=4 \iota$ and since $N$ is odd, the first hypothesis is satisfied. The second follows from Theorem \ref{theorem_QCA_Muller}. Therefore Theorem \ref{theorem_MNKY} implies that $\mathcal{A}_q(\mathbf{\Sigma})\cong \overline{\mathcal{S}}_A(\mathbf{\Sigma})$ is Azumaya. 
 \end{proof}

 \section{Gluing marked surfaces together} \label{sec_gluing}
 
  If $V_1$ and $V_2$ are modules over $ \overline{\mathcal{S}}_A(\mathbf{\Sigma}_1)$ and $ \overline{\mathcal{S}}_A(\mathbf{\Sigma}_2)$ respectively and $a_1, a_2$ some boundary edges of $\mathbf{\Sigma}_1$ and $\mathbf{\Sigma}_2$ respectively, then $V_1\otimes V_2$ is a  $\overline{\mathcal{S}}_A(\mathbf{\Sigma}_1 \cup_{a_1 \# a_2} \mathbf{\Sigma}_2)$-module by precomposing with $\theta_{a_1 \# a_2}$. 
 The goal of this section is to prove the 
 
 \begin{theorem}\label{theorem_gluing} Let  $\mathbf{\Sigma}_1$ and $\mathbf{\Sigma}_2$ be marked surfaces and for $i=1,2$ let $a_i$ be a boundary edges of $\mathbf{\Sigma}_i$ such that the connected component of $\partial \Sigma_i$ which contains $a_i$ also contains at least another boundary edge. 
 Write $\mathbf{\Sigma}:=\mathbf{\Sigma}_1 \cup_{a_1\# a_2} \mathbf{\Sigma}_2$.
 Then 
$$ \widehat{p}_{a_1 \# a_2} \left(\mathcal{AL}(\mathbf{\Sigma}_1) \times \mathcal{AL}(\mathbf{\Sigma}_2)\right) = \mathcal{AL} (\mathbf{\Sigma}).$$
Moreover, any indecomposable (semi) weight $\overline{\mathcal{S}}_A(\mathbf{\Sigma})$-module is isomorphic to a module $V_1\otimes V_2$ with $V_i$ a (semi) weight indecomposable $ \overline{\mathcal{S}}_A(\mathbf{\Sigma})$ module. Conversely, any such $ \overline{\mathcal{S}}_A(\mathbf{\Sigma})$ module $V_1\otimes V_2$ is indecomposable of the same kind (weight or semi weight).
\end{theorem}
 
Let $Z^1_{\mathbf{\Sigma}}$ be the subalgebra of $Z_{\mathbf{\Sigma}}$ generated by $Z^0_{\mathbf{\Sigma}}$ and the elements $\alpha_{\partial}^{\pm 1}$ for $\partial \in \Gamma^{\partial}$; thus $Z^0_{\mathbf{\Sigma}}\subset Z^1_{\mathbf{\Sigma}} \subset Z_{\mathbf{\Sigma}}$.
 
 \begin{lemma}\label{lemma_gluing} Under the hypotheses of Theorem \ref{theorem_gluing}, any element of $ \overline{\mathcal{S}}_A(\mathbf{\Sigma}_1)\otimes \overline{\mathcal{S}}_A(\mathbf{\Sigma}_2)$ is the product of an element of the image of  $\theta_{a_1 \# a_2}$ with an element of $Z^1_{\mathbf{\Sigma}_1}\otimes Z^1_{\mathbf{\Sigma}_2}$.
 \end{lemma}
 
 \begin{proof}
 Let us introduce some notations illustrated in Figure \ref{fig_gluing_surjective} $(i)$.
Let $\partial_1$ and $\partial_2$ be the boundary components of $\Sigma_1$ and $\Sigma_2$ containing $a_1$ and $a_2$. The orientations of $\Sigma_1, \Sigma_2$ induces orientations of $\partial_1$ and $\partial_2$ thus a cyclic order on their sets of boundary edges and an orientation of each boundary edge. Let $p_1, p_1'$ be the two punctures adjacent to $a_1$ such that $a_1$ is oriented from $p_1'$ to $p_1$. Let $p_2, p_2'$ be the two punctures adjacent to $a_2$ such that $a_2$ is oriented from $p_2$ to $p_2'$.  Write $\mathbf{\Sigma}= \mathbf{\Sigma}_1\cup_{a_1 \# a_2} \mathbf{\Sigma}_2$. Let $d\subset \Sigma$ be the common image of $a_1$ and $a_2$ in the quotient map $\pi:\Sigma_1 \sqcup \Sigma_2 \to \Sigma$ and write $p:= \pi(p_1)=\pi(p_2)$ and $p':= \pi(p_1')=\pi(p_2')$. Let $b_1, c_1$ be the boundary edges in $\partial_1$ such that $p_1$ is adjacent to $a_1$ and $c_1$ and $p_2$ is adjacent to $a_1$ and $b_1$ (if $\partial_1$ only has two boundary edges, then $c_1=b_1$). Let $b_2,c_2$ be the boundary edges in $\partial_2$ such that $p_2$ is adjacent to $a_2$ and $c_2$ and $p_2'$ is adjacent to $b_2$ and $a_2$. Let $R \subset \overline{\mathcal{S}}_A(\mathbf{\Sigma}_1)\otimes \overline{\mathcal{S}}_A(\mathbf{\Sigma}_2)$ be the subalgebra of elements which are a product of an element of $\theta_{a_1 \# a_2}\left( \overline{\mathcal{S}}_A(\mathbf{\Sigma}) \right)$ with an element of 
$Z^1_{\mathbf{\Sigma}_1}\otimes Z^1_{\mathbf{\Sigma}_2}$.
 We need to show that for every $[D_1, s_1] \in \mathcal{B}^M(\mathbf{\Sigma}_1)$ and for every $[D_2,s_2]\in \mathcal{B}^M(\mathbf{\Sigma}_2)$ then $[D_1,s_1]\otimes 1 \in R$ and $1\otimes [D_2,s_2] \in R$. 
 Let us first prove the following:
 \par \textbf{Claim:} For $\varepsilon = \pm$,  then the four elements  $\alpha(p_1)_{\varepsilon \varepsilon} \otimes 1$, $\alpha(p'_1)_{\varepsilon \varepsilon} \otimes 1$, $1\otimes \alpha(p_2)_{\varepsilon \varepsilon} $ and $1\otimes \alpha(p'_2)_{\varepsilon \varepsilon} $ are in $R$.  
 \\ We prove the claim for $\alpha(p_1)_{++}\otimes 1$ and leave the other similar cases to the reader. 
 \par $\bullet$ $\underline{ \alpha(p_1)_{--} \alpha(p_1')_{++} \otimes 1  \in R:}$ Choose two points $v\in b_1$, $w\in c_1$ and consider an arc $\beta$ oriented from $v$ to $w$ such that $\beta$ can be isotoped relatively to its boundary to an arc $\beta'\subset \partial_1$ containing $a_1$. As illustrated in Figure \ref{fig_gluing_surjective} $(ii)$, one has $\theta_{a_1 \# a_2} (\beta_{+-})= A^{1/2}\alpha(p_1)_{--}\alpha(p_1')_{++}\otimes 1 \in R$.
 \par $\bullet$ $\underline{ \alpha(p_1)_{--} \alpha(p_1')_{--} \otimes 1  \in R:}$ Let $(p_1', p_1, p_3, p_4, \ldots, p_n)$ be the punctures of $\partial_1$ cyclically ordered. Clearly, for $i=3, \ldots, n$,  $\alpha(p_i)_{--}\otimes 1$ is in the image of $\theta_{a_1 \# a_2}$. Moreover $\alpha_{\partial_1}^{-1}= \alpha(p_1)_{--}\alpha(p_1')_{--} \alpha(p_3)_{--}\ldots \alpha(p_n)_{--} \in Z_{\mathbf{\Sigma}_1}$, therefore $ \alpha(p_1)_{--} \alpha(p_1')_{--} \otimes 1  \in R$.
 \par $\bullet$ $\underline{ \alpha(p_1)_{++}^N \otimes 1  \in R:}$ By Lemma \ref{lemma_heightexch}, there exists $n\in \mathbb{Z}$ such that $(\alpha(p_1)_{++})^N = A^{n/2} \alpha(p_1)_{++}^{(N)}$. Since $\alpha(p_1)_{++}^{(N)}$ is in the image of the Frobenius morphism by Theorem \ref{theorem_Frobenius}, one has $(\alpha(p_1)_{++})^N \otimes 1 \in R$. 
 \par $\bullet$ $\underline{ \alpha(p_1)_{++} \otimes 1  \in R:}$ First one has $(\alpha(p_1)_{--})^2= (\alpha(p_1)_{--} \alpha(p_1')_{++} )(\alpha(p_1)_{--} \alpha(p_1')_{--} ) \in R$. Write $N=2k+1$. Then $\alpha(p_1)_{++} = (\alpha(p_1)_{--})^{2k} (\alpha(p_1)_{++})^N \in R.$
 \par The other cases of the claim are proved similarly. Let $[D,s] \in \mathcal{B}^M(\mathbf{\Sigma}_2)$ and let us prove that $1\otimes [D,s] \in R$. By Lemma \ref{lemma_heightexch} there exist $n\in \mathbb{Z}$ and $D^0\subset D$ a sub diagram such that $[D,s]=A^{n/2} [D^0, s^+] \prod_{p \in \mathcal{P}^{\partial}} \alpha(p)_{--}^{n_p}$, where $s^+$ is totally positive in the sense that it sends every endpoints in $\partial D_0$ to $+$. Since every elements $1 \otimes (\alpha(p)_{--})^{n_p}$ are in $R$, it suffices to prove that $1\otimes [D^0, s^+] \in R$ for a totally positive simple stated diagram. For such a $D^0 \subset \Sigma_2$, let $\widehat{D^0} \subset \Sigma$ be the simple diagram obtained from $D^0$ by pushing slightly each point of $\partial_{a_2} D^0$ to $b_1$ while forming a corner arc as in Figure \ref{fig_gluing_surjective} $(iii)$. If $m:= |\partial_{a_2}D^0|$ then  as illustrated in Figure \ref{fig_gluing_surjective} $(iii)$, one has 
 $$ \theta_{a_1 \# a_2} ([\widehat{D^0}, s^+]) = (\alpha(p'_1)_{++})^m \otimes [D^0, s^+].$$
 Since $\alpha(p'_1)_{++}\otimes 1 \in R$ by the claim, we obtain that $1\otimes [D^0, s^+] \in R$. So every element of the form $1\otimes [D,s]$ is in $R$. We prove that the elements $[D,s]\otimes 1$ belong to $R$ similarly. This concludes the proof.

 \begin{figure}[!h] 
\centerline{\includegraphics[width=14cm]{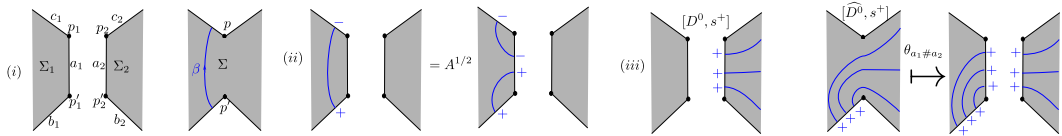} }
\caption{$(i)$ The marked surfaces $\mathbf{\Sigma}_1, \mathbf{\Sigma}_2, \mathbf{\Sigma}$ and some boundary edges; $(ii)$ An illustration of the equality  $\theta_{a_1 \# a_2} (\beta_{+-})= A^{1/2}\alpha(p_1)_{--}\alpha(p_1')_{++}\otimes 1$; $(iii)$ A diagram $(D^0, s^+)$, its associated diagram $(\widehat{D^0}, s^+)$ and an illustration of the equality $ \theta_{a_1 \# a_2} ([\widehat{D^0}, s^+]) = (\alpha(p'_1)_{++})^m \otimes [D^0, s^+]$.
} 
\label{fig_gluing_surjective} 
\end{figure}

 \end{proof}
 
 \begin{lemma}\label{lemma_partial} Let $\rho:  \overline{\mathcal{S}}_A(\mathbf{\Sigma})\to \End(V)$  be  an indecomposable semi-weight representation and $\partial \in \Gamma^{\partial}$. Then  there exists $h_{\partial}\in \mathbb{C}^*$ such that $\rho(\alpha_{\partial})=h_{\partial}\id_V$.
 \end{lemma}
 
 \begin{proof} The operator $\rho(\alpha_{\partial})$ is annihilated by the polynomial $P(X)=X^{N} - \chi_x(\alpha_{\partial})$ so the minimal polynomial of $\rho(\alpha_{\partial})$ divides $P(X)$. Since $\rho(\alpha_{\partial})$ commutes with the image of $\rho$ and since $\rho$ is indecomposable, its minimal polynomial has the form $(X-c)^n$. 
 Since $P(X)$ has its roots of multiplicity one, the minimal polynomial of  $\rho(\alpha_{\partial})$ has the form $X-h_{\partial}$ for some $h_{\partial}$ such that $h_{\partial}^N=\chi_x(\alpha_{\partial})$.
 \end{proof}
 
  Therefore, for $\rho:  \overline{\mathcal{S}}_A(\mathbf{\Sigma})\to \End(V)$   an indecomposable semi-weight representation, one can associate a tuple $\widetilde{x}=(x; h_{\partial})_{\partial \in \Gamma^{\partial}} \in \Specm(Z^1_{\mathbf{\Sigma}})=: \widetilde{X}(\mathbf{\Sigma})$ which we call the middle shadow of $\rho$. Then $\rho$ factorizes through the algebra 
  $$ \overline{\mathcal{S}}_A(\mathbf{\Sigma})_{\widetilde{x}} := \quotient{\overline{\mathcal{S}}_A(\mathbf{\Sigma})}{\widetilde{x}\overline{\mathcal{S}}_A(\mathbf{\Sigma})}.$$

 \begin{proof}[Proof of Theorem \ref{theorem_gluing}]
 \par \underline{\textit{The case of weight representations:}}
 Let $\widehat{x} \in \widehat{X}(\mathbf{\Sigma})$. By Theorem \ref{theorem_gluing_surjective}, there exist $\widehat{x}_i \in  \widehat{X}(\mathbf{\Sigma}_i)$, $i=1,2$, such that $\widehat{p}_{a_1\# a_2}(\widehat{x}_1, \widehat{x}_2)=\widehat{x}$. The splitting morphism $\theta_{a_1\# a_2}$ induces an injective morphism 
 $$ \widetilde{\theta}_{a_1 \# a_2} : \overline{\mathcal{S}}_A(\mathbf{\Sigma})_{\widehat{x}}  \to \overline{\mathcal{S}}_A(\mathbf{\Sigma}_1)_{\widehat{x}_1}\otimes \overline{\mathcal{S}}_A(\mathbf{\Sigma}_2)_{\widehat{x}_2}.$$
 By Lemma \ref{lemma_gluing}, $ \widetilde{\theta}_{a_1 \# a_2}$ is surjective. To prove that it is injective, recall that by definition, $ \widetilde{\theta}_{a_1 \# a_2}$ is obtained from the injective morphism $ {\theta}_{a_1 \# a_2}$ by passing to the quotient. So we need to prove the inclusion $ {\theta}_{a_1 \# a_2}^{-1}\left( \mathcal{I}_{\widehat{x}_1}\otimes 1 + 1\otimes \mathcal{I}_{\widehat{x}_2}\right) \subset \mathcal{I}_{\widehat{x}}$; this follows from the inclusion $ {\theta}_{a_1 \# a_2}^{-1}\left(Z_{\mathbf{\Sigma}_1}\otimes 1 + 1\otimes Z_{\mathbf{\Sigma}_2}\right) \subset Z_{\mathbf{\Sigma}}$ which immediately follows from the description of the centers in Theorem \ref{theorem_center}.
 \par So  $\widetilde{\theta}_{a_1 \# a_2}$ is an isomorphism. By definition, an indecomposable weight $ \overline{\mathcal{S}}_A(\mathbf{\Sigma})$ module with full shadow $\widehat{x}$ is the same  than an indecomposable representation of the finite dimensional algebra $ \overline{\mathcal{S}}_A(\mathbf{\Sigma})_{\widehat{x}}$ (and similarly for $\mathbf{\Sigma}_i$). Therefore there is a $1:1$ correspondance between isomorphism classes of pairs of indecomposable weight representation $(V_1, V_2)$ where $V_i$ has full shadow $\widehat{x}_i$ and the set of isomorphism classes of weight $ \overline{\mathcal{S}}_A(\mathbf{\Sigma})$ representations with full shadow $\widehat{x}$; the isomorphism sends $(V_1,V_2)$ to $V_1\otimes V_2$. By Theorem \ref{theorem_AL}, $\widehat{x} \in \mathcal{AL}(\mathbf{\Sigma})$ if and only if $\dim \left(\overline{\mathcal{S}}_A(\mathbf{\Sigma})_{\widehat{x}}\right) = (D_{\mathbf{\Sigma}})^2$ and if  $\widehat{x} \notin \mathcal{AL}(\mathbf{\Sigma})$ then $\dim \left(\overline{\mathcal{S}}_A(\mathbf{\Sigma})_{\widehat{x}}\right) < (D_{\mathbf{\Sigma}})^2$ (and similarly for $\mathbf{\Sigma}_i$). By Lemma \ref{lemma_additivityD}, we have $D_{\mathbf{\Sigma}}=D_{\mathbf{\Sigma}_1}D_{\mathbf{\Sigma}_2}$ so $\widehat{x} \in \mathcal{AL}(\mathbf{\Sigma})$ if and only if $x_i \in \mathcal{AL}(\mathbf{\Sigma}_i)$ for $i=1,2$.
 \vspace{2mm}
 \par  \underline{\textit{The case of semi weight representations:}} Let $\widetilde{x}=(x, h_{\partial})_{\partial \in \Gamma^{\partial}}\in \widetilde{X}(\mathbf{\Sigma})$ and denote by $\widetilde{p}_{a_1\# a_2} : \widetilde{X}(\mathbf{\Sigma}_1)\times  \widetilde{X}(\mathbf{\Sigma}_2) \to  \widetilde{X}(\mathbf{\Sigma})$ the morphism induced by $\theta_{a_1\# a_2}$.
 As showed in the proof of Theorem \ref{theorem_gluing_surjective}, there exists $(\widetilde{x}_1, \widetilde{x}_2) \in \widetilde{X}(\mathbf{\Sigma}_1)\times  \widetilde{X}(\mathbf{\Sigma}_2)$ such that  $\widetilde{p}_{a_1\# a_2} (\widetilde{x}_1, \widetilde{x}_2)= \widetilde{x}$. $\theta_{a_1\# a_2}$ induces a morphism 
 $$ \widetilde{\theta}_{a_1 \# a_2} : \overline{\mathcal{S}}_A(\mathbf{\Sigma})_{\widetilde{x}}  \to \overline{\mathcal{S}}_A(\mathbf{\Sigma}_1)_{\widetilde{x}_1}\otimes \overline{\mathcal{S}}_A(\mathbf{\Sigma}_2)_{\widetilde{x}_2}.$$
 Again Lemma \ref{lemma_gluing} proves the surjectivity of $ \widetilde{\theta}_{a_1 \# a_2}$ and the injectivity follows from the inclusion 
 \\ $ {\theta}_{a_1 \# a_2}^{-1}\left( \mathcal{I}_{\widetilde{x}_1}\otimes 1 + 1\otimes \mathcal{I}_{\widetilde{x}_2}\right) \subset \mathcal{I}_{\widetilde{x}}$, so  $ \widetilde{\theta}_{a_1 \# a_2}$ is an isomorphism. Therefore we have a $1:1$ correspondance between indecomposable representations $V$ of  $ \overline{\mathcal{S}}_A(\mathbf{\Sigma})$ with shadow $\widetilde{x}$ and pairs $(V_1, V_2)$ with $V_i$ an indecomposable $ \overline{\mathcal{S}}_A(\mathbf{\Sigma}_i)$ semi weight module with shadow $\widetilde{x}_i$ (the correspondence sends $(V_1, V_2)$ to $V_1\otimes V_2$). This concludes the proof.
 \end{proof}

 \section{Classification of semi-weight representations of reduced stated skein algebras}\label{sec_classification}
 
 Let $\mathbf{\Sigma}$  be a connected essential marked surface.
 For $x\in X(\mathbf{\Sigma})$, we denote by $\rho_x: \Pi_1(\Sigma, \mathbb{V}) \to \SL_2$ the functor associated by Corollary \ref{coro_classical_limit}. Recall that $x$ \textit{is central at} $p\in \mathring{\mathcal{P}}$ if $\rho_x(\gamma_p)=\pm \mathds{1}_2$ for $\gamma_p$ a peripheral loop encircling $p$.
\par   Let $\widehat{x}=(x,h_p,h_{\partial})_{p, \partial} \in \widehat{X}(\mathbf{\Sigma})$ and decompose the set of inner punctures as $\mathring{\mathcal{P}}= \mathring{\mathcal{P}}_0 \sqcup \mathring{\mathcal{P}}_1 \sqcup \mathring{\mathcal{P}}_2$ where 
\begin{align*}
&  \mathring{\mathcal{P}}_2:= \{p\in \mathring{\mathcal{P}}, \mbox{ such that } \rho_x(\gamma_p)=\pm \mathds{1}_2 \mbox{ and }h_p \neq \mp 2\} \\
&  \mathring{\mathcal{P}}_1:= \{p\in \mathring{\mathcal{P}}\setminus \mathring{\mathcal{P}}_2, \mbox{ such that } \tr(\rho_x(\gamma_p)) = \pm 2 \mbox{ and }h_p \neq \pm 2\}.
\end{align*}
\begin{definition}[Colorings]
A map $\sigma : \mathring{\mathcal{P}}\to \Delta \bigsqcup \overline{\Delta} $ is called a \textbf{coloring compatible with} $\widehat{x}$ if $(1)$ for $p\in \mathring{\mathcal{P}}_0$, then $\sigma(p)=S$ and $(2)$ if $p\in \mathring{\mathcal{P}}_1$, then $\sigma(p)\in \{S,P\}$. By convention, if $\mathring{\mathcal{P}}=\emptyset$, we consider that every $\widehat{x}$ has a unique coloring. The set of colorings compatible with $\widehat{x}$ will be denoted by $\col(\widehat{x})$ and we write $m(\widehat{x}):= | \mathring{\mathcal{P}}_2|$.
\end{definition}
Let $\Indecomp(\mathbf{\Sigma})$ be the set of isomorphism classes of indecomposable semi-weight $\overline{\mathcal{S}}_A(\mathbf{\Sigma})$ modules and let $\Indecomp(\mathbf{\Sigma}, \widehat{x})\subset \Indecomp(\mathbf{\Sigma})$ the subset of these modules with maximal shadow $\widehat{x}$. 
Denote by $\overline{\mathcal{C}}$  the category of semi weight $\overline{\mathcal{S}}_A(\mathbf{\Sigma})$ modules. The main theorem of this paper is the following:
 
\begin{theorem}\label{main_theorem} Let $\mathbf{\Sigma}=(\Sigma, \mathcal{A})$ be a connected essential marked surface  which either has a boundary component with at least two boundary edges or which does not have any inner puncture.
 Let  $\widehat{x}=(x, h_p, h_{\partial}) \in \widehat{X}(\mathbf{\Sigma})$.
\begin{enumerate}
\item We have $\mathcal{AL}(\mathbf{\Sigma})= \widehat{X}^{reg}(\mathbf{\Sigma})$, i.e. the Azumaya locus of $\overline{\mathcal{S}}_A(\mathbf{\Sigma})$ is equal to the regular locus of $\widehat{X}(\mathbf{\Sigma})$. Moreover 
 $\widehat{x}\in \mathcal{AL}(\mathbf{\Sigma})$ if and only if $m=0$.
\item The fully Azumaya locus $\mathcal{FAL}(\mathbf{\Sigma})$ of  $\overline{\mathcal{S}}_A(\mathbf{\Sigma})$ is the locus of representations which are not central at any inner puncture. For $x\in \mathcal{FAL}(\mathbf{\Sigma})$ the indecomposable semi-weight representations with classical shadow $x$ are classified (up to isomorphism) by their full shadows, so their set  is in $1:1$ correspondance with the fiber $\widehat{\pi}^{-1}(x) \subset \widehat{\mathcal{X}}(\mathbf{\Sigma})$.
\item The indecomposable semi-weight representations with maximal shadow  $\widehat{x}$ are in $1:1$ correspondence with the set of colorings compatible with $\widehat{x}$ so one has 
$$ \Indecomp(\mathbf{\Sigma})= \bigsqcup_{\widehat{x}\in \widehat{X}(\mathbf{\Sigma})} \Indecomp(\mathbf{\Sigma}, \widehat{x}) \cong \bigsqcup_{\widehat{x}\in \widehat{X}(\mathbf{\Sigma})} \col(\widehat{x}).$$
\item The indecomposable weight representations with full shadow  $\widehat{x}$ correspond to the colorings taking values in $\{S,\overline{S}, ((1,1), 1), \overline{((1,1), 1)}\}$: there are thus $4^m$ such representations.
\item The irreducible representations with full shadow  $\widehat{x}$ correspond to the colorings taking values in $\{S, \overline{S}\}$: there are thus $2^m$ such representations.
\item The indecomposable semi-weight representations with maximal shadow $\widehat{x}$ which are projective objects in $\overline{\mathcal{C}}$ correspond to the colorings sending the elements of $\mathring{\mathcal{P}}_1$ to $P$ and the elements of $\mathring{\mathcal{P}}_2$ to an element in $\{P, \overline{P}\}$: there are thus $2^m$ such representations.
\end{enumerate}
\end{theorem}
 
 \begin{proof}
 If $\mathbf{\Sigma}$ does not have inner puncture then the theorem follows from Corollary \ref{coro_2Psurfaces}. Else, it contains a boundary component with at least two boundary edges.
 Write $\mathring{\mathcal{P}}=\{p_1, \ldots, p_n\}$.
$\mathbf{\Sigma}$ can be obtained from a 2P-marked surface $\mathbf{\Sigma}^0$ and $n$ copies of $\mathbb{D}_1$ together, i.e.
 $$ \mathbf{\Sigma}= \mathbf{\Sigma}^0 \cup_{a_1 \# b_1} \mathbb{D}_1 \cup_{a_2\# b_2} \ldots \cup_{a_n \# b_n} \mathbb{D}_1$$
 for some boundary edges $a_i,b_i$. 
 Let $\theta: \overline{\mathcal{S}}_A (\mathbf{\Sigma}) \hookrightarrow \overline{\mathcal{S}}_A (\mathbf{\Sigma}^0) \otimes \overline{\mathcal{S}}_A (\mathbb{D}_1)^{\otimes n}$ be the induced splitting morphism and $\pi: \widehat{X}(\mathbf{\Sigma}^0) \times \widehat{X}(\mathbb{D}_1)^{\times n} \to \widehat{X}(\mathbf{\Sigma})$ the morphism defined by the restriction of $\theta$ to the center. By Theorem \ref{theorem_gluing}, one has $\pi( \mathcal{AL}(\mathbf{\Sigma}^0) \times \mathcal{AL}(\mathbb{D}_1)^{\times n})= \mathcal{AL}(\mathbf{\Sigma})$. By Corollary \ref{coro_2Psurfaces}, $\mathcal{AL}(\mathbf{\Sigma}^0)= \widehat{X}(\mathbf{\Sigma}^0)$ and by Corollary \ref{coro_QG}, $\mathcal{AL}(\mathbb{D}_1)$ is the set of elements $(g, h_p, h_{\partial})$ such that either $\varphi(g) \neq \pm \mathds{1}_1$ or $\varphi(g)=\pm 2$ and  $h_p = \mp 2$. Therefore $\mathcal{AL}(\mathbf{\Sigma})$ is the set of elements $\widehat{x}=(x,h_p, h_{\partial})$ such that for every $p\in \mathring{\mathcal{P}}$ then either $\rho_x(\alpha_p)\neq \pm \mathds{1}_1$ or $\rho_x(\alpha_p)=\pm \mathds{1}_1$ and $h_p=\mp2$. Together with Theorem \ref{theorem_smooth2}, this proves $(1)$. 
 Assertion $(2)$ follows from Assertion $(1)$ together with Theorem \ref{theorem_FAL}.
 
 Let $\widehat{x} \in \widehat{X}(\mathbf{\Sigma})$. By Theorem \ref{theorem_gluing_surjective}, there exists $(\widehat{x}_0, \widehat{x}_1, \ldots, \widehat{x}_n) \in \widehat{X}(\mathbf{\Sigma}^0) \times \widehat{X}(\mathbb{D}_1)^{ n} $ such that $p(\widehat{x}_0, \widehat{x}_1, \ldots, \widehat{x}_n) =\widehat{x}$. 
  By Theorem \ref{theorem_gluing} the (semi) weight indecomposable $\overline{\mathcal{S}}_A(\mathbf{\Sigma})$ module with classical shadow $\widehat{x}$ have the form 
 $V_0 \otimes V_1 \otimes \ldots \otimes V_n$ where $V_0$ is an indecomposable (semi) weight  $\overline{\mathcal{S}}_A (\mathbf{\Sigma}^0)$ module with shadow $\widehat{x}_0$ and $V_i$ is an indecomposable (semi) weight $\overline{\mathcal{S}}_A (\mathbb{D}_1)$-module with shadow $\widehat{x}_i$. By Corollary \ref{coro_2Psurfaces} $V_0$ is unique up to isomorphism.
For $i\in \{1, \ldots, n\}$, write $\widehat{x}_i= (g_i, h_{p_i}, h_{\partial}^i)$ and let $I^2:= \{i \mbox{ such that }p_i \in \mathring{\mathcal{P}}^2\}$ and $I^1:= \{i \mbox{ such that }p_i \in \mathring{\mathcal{P}}^1\}$.
 By Corollary \ref{coro_QG}, for $i\in \{1, \ldots, n\}$, we have up to isomorphism: $(i)$ a single possible $V_i$ if $i\notin I^1 \cup I^2$, $(ii)$ two possible semi-weight modules $V_i$ if $i\in I^1$ and one of them is a  weight module, $(iii)$ 
 if $i \in I^2$, the indecomposable semi-weight modules $V_i$ are in $1:1$ correspondance with $\Delta\sqcup \overline{\Delta}$. Therefore, the isomorphism classes of semi weight $\overline{\mathcal{S}}_A(\mathbf{\Sigma})$ modules 
with shadow $\widehat{x}$ are in $1:1$ correspondance with the set of colorings admissible at $\widehat{x}$ and we have proved $(3)$. Such an indecomposable representation $V_0 \otimes V_1 \otimes \ldots \otimes V_n$ is weight/ simple/projective in  $\overline{\mathcal{C}}$ if and only if each of the $V_i$ has the same property so $(4)$, $(5)$, $(6)$ follow from  Corollary \ref{coro_QG}.
  This  concludes the proof.
  
 \end{proof}
 
 \section{Representations of unreduced stated skein algebras and the \QMAs}\label{sec_QMS}
 
 \subsection{Statement of the main result}
 
 Let $\mathbf{\Sigma}$ be an essential  marked surface and denote by $\mathbf{\Sigma}^*$ the marked surface obtained from $\mathbf{\Sigma}$ by gluing a triangle along each boundary edge of $\mathbf{\Sigma}$. So each boundary edge $a$ of $\mathbf{\Sigma}$ corresponds to two boundary edges, say $a'$ and $a''$ in $\mathbf{\Sigma}^*$. Let $i: \mathbf{\Sigma} \to \mathbf{\Sigma}^*$ be the embedding which is the identity outside a small neighborhood of $\mathcal{A}$ and embedding $a$ into $a'$. It is proved in  \cite{LeYu_SSkeinQTraces} that the morphism $i_* : \mathcal{S}_A(\mathbf{\Sigma}) \to \overline{\mathcal{S}}_A(\mathbf{\Sigma}^*)$  induces an injective morphism 
  $$ j : \mathcal{S}_A(\mathbf{\Sigma}) \hookrightarrow \overline{\mathcal{S}}_A(\mathbf{\Sigma}^*).$$
 Therefore, for each (semi-weight) representation $r: \overline{\mathcal{S}}_A(\mathbf{\Sigma}^*)\to \End(V)$, we can associate a representation $r^*= r\circ j : \mathcal{S}_A(\mathbf{\Sigma})  \to \End(V)$. Such a representation will be called a \textbf{reduced representation}. Note that since we assume that $\mathbf{\Sigma}$ is essential, then each connected component of  $\mathbf{\Sigma}^*$ as a boundary component with at least two boundary edges. Thus Theorem \ref{main_theorem} provides a classification of such reduced representations and thus a lot of examples. The goal of this subsection is to prove the 
 
 \begin{theorem}\label{theorem_Skein}
 Let $\mathbf{\Sigma}=(\Sigma_{g,n}, \mathcal{A})$ be a connected essential marked surface of genus $g$ with $n$ boundary components such that each boundary component contains at most one boundary edge. So it has $k:= n-|\mathcal{A}|$ inner punctures $p_1, \ldots, p_k$ and $|\mathcal{A}|$ boundary punctures $p_{\partial_1}, \ldots, p_{\partial_{|\mathcal{A}|}}$.
 \begin{enumerate}
 \item The center $\underline{Z}_{\mathbf{\Sigma}}$ of $\mathcal{S}_A(\mathbf{\Sigma})$ is generated by the image $\underline{Z}^0_{\mathbf{\Sigma}}$ of the Frobenius morphism together with the peripheral curves $\gamma_{p_1}, \ldots, \gamma_{p_k}$. So the closed points of $\widehat{\underline{X}}(\mathbf{\Sigma}):= \Specm(\underline{Z}_{\mathbf{\Sigma}})$ are in $1:1$ correspondence with the set of elements $\widehat{\rho}=(\rho, h_{p_1}, \ldots, h_{p_k})$ with $\rho: \Pi_1(\Sigma, \mathbb{V})\to \SL_2$ and $h_{p_i}\in \mathbb{C}$ is such that $T_N(h_{p_i})=-\tr(\rho(\gamma_{p_i}))$. 
 \item The PI-degree $\underline{D}_{\mathbf{\Sigma}}$ of  $\mathcal{S}_A(\mathbf{\Sigma})$ is equal to $N^{3g-3+n+2|\mathcal{A}|}(= D_{\mathbf{\Sigma}^*})$. 
 \item For $\rho: \Pi_1(\Sigma, \mathbb{V})\to \SL_2$, write $\mu(\rho):= (\rho(\alpha(p_{\partial_1})), \ldots, \rho(\alpha(p_{\partial_{|\mathcal{A}|}})) \in (\SL_2)^{|\mathcal{A}|}$.
 Let  $\widehat{\rho}=(\rho, h_{p_i}) \in \widehat{\underline{X}}(\mathbf{\Sigma})$.
 \begin{itemize}
 \item[(i)] If $\mu(\rho)\in (\SL_2^0)^{|\mathcal{A}|}$ and for each $1\leq i \leq k$, either $\tr(\rho(\gamma_{p_i}))\neq \pm 2$ or $\tr(\rho(\gamma_{p_i}))=\pm 2$ and $h_{p_i}=\mp 2$, then $\widehat{\rho}$ belongs to the Azumaya locus of $\mathcal{S}_A(\mathbf{\Sigma})$.
 \item[(ii)] Suppose that either $\mu(\rho)\in (\SL_2^1)^{|\mathcal{A}|}$ or that $\mu(\rho)\in (\SL_2^0)^{|\mathcal{A}|}$ and there exists $i$ such that $\tr(\rho(\gamma_{p_i}))=\pm 2$ and $h_{p_i}\neq \mp 2$, then $\widehat{\rho}$ does not belong to the Azumaya locus of $\mathcal{S}_A(\mathbf{\Sigma})$.
\end{itemize}
 \end{enumerate}
 \end{theorem}
 
 \begin{remark} Soon after the prepublication of the present paper, Yu computed in \cite{Yu_CenterSSkein} the center and PI-degrees of every stated skein algebras, thus generalizing the first and second items of Theorem \ref{theorem_Skein}. As detailed in  Subsection \ref{sec_QMA}, in the particular case where $|\mathcal{A}|=1$, our theorem re-proves some results of \cite{BaseilhacRoche_LGFT1,BaseilhacRoche_LGFT2, GanevJordanSafranov_FrobeniusMorphism}.
 \end{remark}
 
 \subsection{Quantum groups and their braided transmutation}
 
 Let us introduce two marked surfaces: 
 the \textbf{bigon} $\mathbb{B}=(D^2, \{a_L, a_r\})$ is a disc with two boundary edges; the \textbf{punctured monogon} $\mathbf{m}_1=(\Sigma_{0,2}, \{b\})$ is an annulus with a single boundary edge. The skein algebras of the bigon and the punctured monogon have  quantum group interpretations as follows.
 
 \begin{definition}[Quantum groups and their transmutation] Set $q:= A^2$.
 \begin{enumerate}
 \item The \textbf{quantum group} $\mathcal{O}_q\SL_2$ is the algebra defined by the generators $a,b,c,d$ and relations
 $$ab = q^{-1}ba, \quad  ac=q^{-1}ca,
\quad  db = q bd, \quad dc=q cd,
\quad ad=1+q^{-1}bc, \quad  da=1 + q bc, 
\quad bc=cb.$$
 \item The \textbf{braided quantum group} $B_q\SL_2$ is the algebra with the same underlying vector space than $\mathcal{O}_q\SL_2$ but with product given by 
 \begin{align*}
& ba=q^{2} ab, \quad ca=q^{-2}ac, \quad da=ad, \quad bc=cb+(1-q^{-2})a(d-a) \\
&db=bd+(1-q^{-2})ab, \quad cd=dc+(1-q^{-2})ca, \quad, ad-q^{2}cb=1.
\end{align*} 
\item The \textbf{Frobenius morphism} $Fr: \mathcal{O}[\SL_2] \to \mathcal{O}_q\SL_2$ is the central embedding defined by $Fr(x):= x^N$ for $x=a,b,c,d$. 
 \end{enumerate}
 \end{definition}
 
 \begin{remark}
 The quantum group $\mathcal{O}_q\SL_2$ has a natural coribbon Hopf algebra structure with coproduct and antipode
 
 \begin{equation*}
 \begin{pmatrix} \Delta (a) & \Delta (b) \\ \Delta(c) & \Delta(d) \end{pmatrix} 
 = 
 \begin{pmatrix} a & b \\ c & d \end{pmatrix} 
 \boxtimes 
 \begin{pmatrix} a & b \\ c & d \end{pmatrix} 
\quad
 \begin{pmatrix} \epsilon(a) & \epsilon(b) \\ \epsilon(c) & \epsilon(d) \end{pmatrix} =
\begin{pmatrix} 1 &0 \\ 0& 1 \end{pmatrix}  
\quad
\begin{pmatrix} S(a) & S(b) \\ 	S(c) & S(d) \end{pmatrix} 
	= 
	\begin{pmatrix} d & -q b \\ -q^{-1}c & a \end{pmatrix} .
 \end{equation*}
and with co-R matrix and co-twist given by
$$ r \left(  \begin{pmatrix} a & b \\ c & d \end{pmatrix} 
 \boxtimes 
 \begin{pmatrix} a & b \\ c & d \end{pmatrix} \right)
= \begin{pmatrix} q^{1/2} & 0 & 0 & 0 \\ 0 & 0 &q^{-1/2} & 0 \\ 0 & q^{-1/2} & q^{1/2}-q^{-3/2} & 0 \\ 0 & 0 & 0 & q^{1/2} \end{pmatrix}
\quad
\Theta  \begin{pmatrix} a & b \\ c & d \end{pmatrix} = -q^{3/2}\begin{pmatrix} 1 &0 \\ 0& 1 \end{pmatrix}.$$
Majid's transmutation is a procedure which associates to any coribbon Hopf $H$ algebra an Hopf algebra object $BH$ in the category $H-\Comod$  where $BH=H$ as a coalgebra and the \textit{transmuted product} is given by 
$$
\underline{\mu} (x\otimes y) := \sum x_{(2)}y_{(2)} r( S(x_{(1)}) x_{(3)} \otimes S(y_{(1)})) 
$$
and the \textit{transmuted antipode} is 
$$
 \underline{S} (x) = \sum S(x_{(2)})r((S^{2}(x_{(3)}) S(x_{(1)}) \otimes x_{(4)}))
 $$
 $B_q\SL_2$ is the transmutation of $\mathcal{O}_q\SL_2$ so this explains its origin. We refer to \cite{Majid_BQG_Rank, Majid_BraidedHopfAlg, Majid_QGroups} for details.
 \end{remark}
 
 For $i,j \in \{-, +\}$, consider the elements $\beta_{ij}:= \adjustbox{valign=c}{\includegraphics[width=0.8cm]{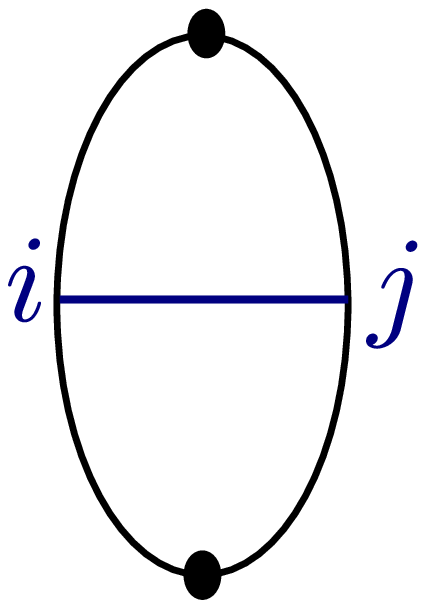}} \in \mathcal{S}_A(\mathbb{B})$ and $\gamma_{ij} := \adjustbox{valign=c}{\includegraphics[width=1cm]{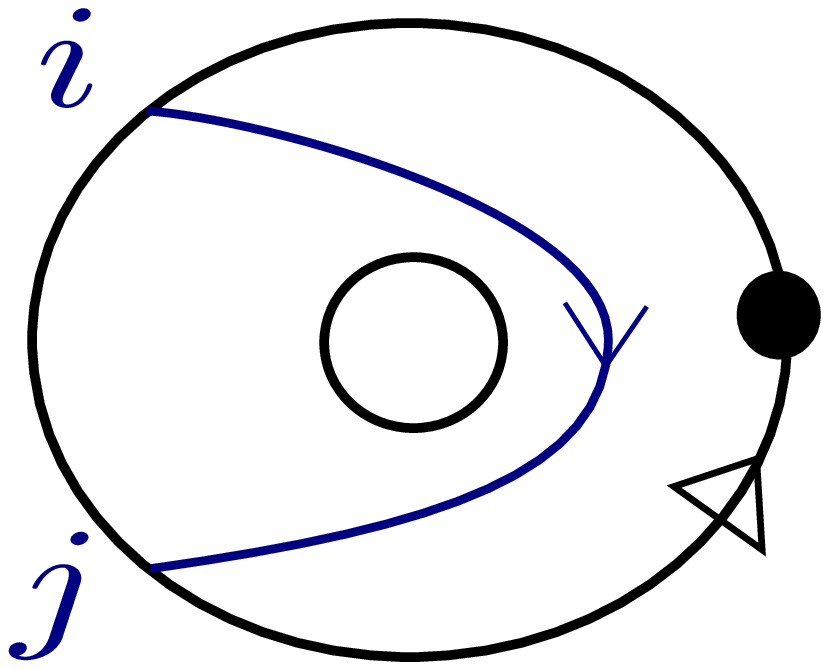}} \in \mathcal{S}_A(\mathbf{m}_1)$.
 
 \begin{theorem}
 \begin{enumerate}
 \item (\cite{KojuQuesneyClassicalShadows, CostantinoLe19}) One has an algebra isomorphism $\varphi_1: \mathcal{S}_A(\mathbb{B}) \cong \mathcal{O}_q[\SL_2]$ sending $\beta_{++}$, $\beta_{+-}$, $\beta_{-+}$, $\beta_{--}$ to $a,b,c,d$ respectively.
\item (\cite{CostantinoLe19}, see also \cite{KojuMurakami_QCharVar, LeSikora_SSkein_SLN}) One has an algebra isomorphism $\varphi_2: \mathcal{S}_A(\mathbf{m}_1) \cong B_q[\SL_2]$ defined by 
$$\varphi_2\begin{pmatrix} a & b \\ c & d \end{pmatrix} := \begin{pmatrix} 0 & -A^{5/2} \\ A^{1/2} & 0 \end{pmatrix} \begin{pmatrix} \gamma_{++} & \gamma_{+-} \\ \gamma_{-+} & \gamma_{--} \end{pmatrix} =\begin{pmatrix} -A^{5/2}\gamma_{-+} & -A^{5/2}\gamma_{--} \\ A^{1/2}\gamma_{++} & A^{1/2}\gamma_{+-} \end{pmatrix}.$$
 \end{enumerate}
 \end{theorem}
 
 Note that both isomorphisms $\varphi_1$ and $\varphi_2$ intertwines the Frobenius morphisms coming from quantum groups and skein algebras (here we consider that $\mathcal{O}_q\SL_2$ and $B_q\SL_2$ are equal as $\mathcal{O}[\SL_2]$ modules).
 
 \begin{theorem}\label{theorem_QG_free} Let $Z^0 \subset \mathcal{O}_q\SL_2$ the image of the Frobenius morphism. 
 \begin{enumerate}
 \item (\cite{DeConciniLyubashenko_OqG}, see also  \cite[Proposition $2.2$]{BrownGordon_OqG}) $\mathcal{O}_q\SL_2$ is free over $Z^0$ of rank $N^3$. 
 \item (\cite{DabrowskiReinaZampa_BasisOqSL2_Z0}) An explicit basis of $\mathcal{O}_q\SL_2$ over $Z^0$ is given by 
 $$\mathcal{B}_0:= \{ a^m b^n c^{s'}:  1\leq m \leq s' \leq N-1, 0\leq n \leq N-1\} \sqcup \{b^n c^{s''}d^r: 0\leq n,r \leq N-1, 0 \leq s'' \leq N-1-r\}.$$
 \end{enumerate}
 \end{theorem}
 
 Let $\gamma_p \in \mathcal{S}_A(\mathbf{m}_1)$ be the peripheral curve around the unique inner puncture. A simple skein computation shows that $\gamma_p= A^{1/2}\gamma_{-+} - A^{5/2}\gamma_{+-}$ so $\omega:= \varphi_2^{-1}(\gamma_p)=-q^{-1} a -q d\in B_q\SL_2$. 
 
 \begin{lemma}\label{lemma_centerBqSL2}
Let  $Z'\subset B_q\SL_2$ be the subalgebra generated by $Z^0$ and $\omega$. Then  $\dim_{Q(Z')} (B_q\SL_2\otimes_{Z'} Q(Z'))=N$.
 \end{lemma}
 
 \begin{proof} As a module over $Z^0$, $Z'\cong \quotient{Z^0[\omega]}{( T_N(\omega)=Fr(a+d))}$ so $Z'$ is freely generated over $Z^0$ of rank $N$. Since $B_q\SL_2$ is freely generated over $Z^0$ of rank $N^2$, we obtain
 $$ \dim_{Q(Z')} (B_q\SL_2\otimes_{Z'} Q(Z')) = \frac{[B_q\SL_2:Z^0]}{[Z':Z^0]}= \frac{N^2}{N}= N.$$
 \end{proof}

 For now on and by abuse of notations, we will identify $\mathcal{O}_q\SL_2$ with $\mathcal{S}_A(\mathbb{B})$ and $B_q\SL_2$ with $\mathcal{S}_A(\mathbf{m}_1)$ using $\varphi_1$ and $\varphi_2$ respectively. 
 \par 
 Let $\mathbf{\Sigma}$ be a marked surface and $\alpha \subset \Sigma$ an arc. Let $v,w$ be the endpoints of $\alpha$ and $a,b$ be the boundary edges containing $v,w$ respectively. If $a=b$ we suppose that $v<_a w$. Write $\alpha_{ij} \in \mathcal{S}_A(\mathbf{\Sigma})$ for the class of $\alpha$ with state $s(v)=i$ and $s(w)=j$. When $a\neq b$, one has an algebra morphism 
 $$ \varphi_{\alpha}: \mathcal{O}_q\SL_2 \to \mathcal{S}_A(\mathbf{\Sigma}), \quad \varphi_{\alpha}(\beta_{ij}):= \alpha_{ij}.$$
 If $a=b$, one has an algebra morphism 
 $$ \varphi_{\alpha}: B_q\SL_2 \to \mathcal{S}_A(\mathbf{\Sigma}), \quad \varphi_{\alpha} (\gamma_{ij}):= \alpha_{ij}.$$
 Note that, since $\mathcal{O}_q\SL_2$ and $B_q\SL_2$ are equal as vector spaces, in both cases ($a=b$ or $a\neq b$) we get a linear morphism $\varphi_{\alpha} : \mathcal{O}_q\SL_2 \to \mathcal{S}_A(\mathbf{\Sigma})$. 
 
 \begin{theorem}(\cite[Corollary $3.8$]{KojuPresentationSSkein}) \label{theorem_graph_linear_isom}
 Suppose that $\mathbf{\Sigma}$ is essential and let $\Gamma$ be a presenting graph with edges $\mathcal{E}(\Gamma)=\{\alpha_1, \ldots, \alpha_n\}$. Consider the linear morphism
  $$\varphi^{\Gamma}:= \otimes_{\alpha \in \mathcal{E}(\Gamma)} \varphi_{\alpha}: (\mathcal{O}_q\SL_2)^{\otimes  \mathcal{E}(\Gamma)} \to \mathcal{S}_A(\mathbf{\Sigma}), \quad \varphi^{\Gamma}(x_1\otimes \ldots \otimes x_n)= \varphi_{\alpha_1}(x_1)\ldots \varphi_{\alpha_n}(x_n).$$
  Then $\varphi^{\Gamma}$ is a linear isomorphism which intertwines the Frobenius morphisms $Fr_{\mathbf{\Sigma}}$ and $Fr^{\otimes \mathcal{E}(\Gamma)}$. 
 \end{theorem}
 Here we have identified $\mathcal{S}_{+1}(\mathbf{\Sigma})\cong \mathcal{R}_{\SL_2}(\mathbf{\Sigma})$ with $\mathcal{O}[\SL_2]^{\otimes \mathcal{E}(\Gamma)}$ using the isomorphism $\varphi_{\Gamma}$ of Section \ref{sec_moduli_spaces} which sends a representation $\rho$ to $(\rho(\alpha_i))_{i=1, \ldots, n}$. Note that $\varphi^{\Gamma}$ is a quantum analogue of $\varphi_{\Gamma}$. 
 
 \subsection{Proof of Theorem \ref{theorem_Skein}}

 \begin{lemma}\label{lemma_Skein1} Let $\mathbf{\Sigma}$ be a connected essential marked surface with underlying surface $\Sigma_{g,n}$ and let $\mathring{n}:=|\mathring{\mathcal{P}}|$ represent the number of inner punctures. Let 
  $\underline{Z}'_{\mathbf{\Sigma}}\subset \underline{Z}_{\mathbf{\Sigma}}$ be the algebra generated by the image of the Frobenius and the peripheral curves $\gamma_{p}$ for $p\in \mathring{\mathcal{P}}$ and denote by $\underline{Z}^0_{\mathbf{\Sigma}}\subset \underline{Z}'_{\mathbf{\Sigma}}$ the image of the Frobenius morphism. 
 \begin{enumerate}
 \item As a $\underline{Z}_{\mathbf{\Sigma}}^0$-module,  $\mathcal{S}_A(\mathbf{\Sigma})$ is free of rank $R_0=N^{6g-6+3|\mathcal{A}| + 3n}$. 
 \item Let $K:= Q(\underline{Z}'_{\mathbf{\Sigma}})$ be the fraction field of $\underline{Z}'_{\mathbf{\Sigma}}$. Then $R:= \dim_K \mathcal{S}_A(\mathbf{\Sigma})\otimes_{\underline{Z}'_{\mathbf{\Sigma}}} K=R_0 N^{-\mathring{n}}$.
 \item Moreover if   $\mathbf{\Sigma}$ contains at most one boundary edge per boundary component, then  $R= (D_{\mathbf{\Sigma}^*})^2$.
 \end{enumerate}
  \end{lemma}
 
 \begin{proof}
 Let $\Gamma$ be the presenting of Example  \ref{example_presenting_graph} where $\mathcal{E}(\Gamma)$ is obtained from $\mathbb{G}$ by removing a corner arc $\alpha(p_{\partial_0})$. 
 By Theorem \ref{theorem_graph_linear_isom}, one has the isomorphism $\varphi^{\Gamma}:= \otimes_{\alpha \in \mathcal{E}(\Gamma)} \varphi_{\alpha}: (\mathcal{O}_q\SL_2)^{\otimes  \mathcal{E}(\Gamma)} \xrightarrow{\cong} \mathcal{S}_A(\mathbf{\Sigma})$ which commutes with the Frobenius morphisms. By Theorem \ref{theorem_QG_free}, $\mathcal{O}_q\SL_2$ is free over $\mathcal{O}[\SL_2]$ of rank $N^3$ so $\mathcal{S}_A(\mathbf{\Sigma})\cong (\mathcal{O}_q\SL_2)^{\otimes  \mathcal{E}(\Gamma)}$ is free over $Z^0_{\mathbf{\Sigma}}\cong (\mathcal{O}[\SL_2])^{\otimes  \mathcal{E}(\Gamma)}$ of rank 
  $$R_0:= [\mathcal{S}_A(\mathbf{\Sigma}):  \underline{Z}_{\mathbf{\Sigma}}^0] = [\mathcal{O}_q\SL_2 : Z^0]^{|\mathcal{E}(\Gamma)|}= (N^3)^{2g-2+n + |\mathcal{A}|}= N^{6g-6 +3n +3|\mathcal{A}|}.$$
  Note that an explicit basis for $\mathcal{S}_A(\mathbf{\Sigma})$ over $Z^0_{\mathbf{\Sigma}}$ is given by $\varphi^{\Gamma}( \mathcal{B}_0^{\otimes \mathcal{E}(\Gamma)})$. 
 Consider the decomposition $\mathcal{E}(\Gamma)= \mathcal{E}_1 \sqcup \mathcal{E}_2$ where $\mathcal{E}_1$ corresponds to the edges of the form $\gamma_p$ with $p\in \mathcal{\mathring{P}}$. Then $\varphi^{\Gamma}$ sends isomorphically the subspace $(Z')^{\otimes \mathcal{E}_1}\otimes (Z^0)^{\otimes \mathcal{E}_2}$ to $\underline{Z}'_{\mathbf{\Sigma}}$. So using Lemma \ref{lemma_centerBqSL2} and writing $n=\mathring{n}+n^{\partial}$, we obtain the equalities
 $$R:= \dim_K \mathcal{S}_A(\mathbf{\Sigma})= (\dim_{Q(Z)} B_q\SL_2\otimes_Z)^{|\mathcal{E}_1|}  ([\mathcal{O}_q\SL_2 : Z^0])^{|\mathcal{E}_2|}= (N^2)^{\mathring{n}}(N^3)^{2g-2+n_{\partial}+ |\mathcal{A}|} = R_0 N^{-\mathring{n}}.$$
 Eventually, if $\mathbf{\Sigma}$ has at most one boundary edge per boundary component, then $| \mathcal{A}|= n^{\partial}$ so 
 $$R= N^{6g-6+2 \mathring{n} +3n^{\partial} + 3|\mathcal{A}|}= N^{6g-6+2(\mathring{n} + n^{\partial}) + 4|\mathcal{A}|}= \left( N^{3g-3+n+2|\mathcal{A}|}\right)^2 = (D_{\mathbf{\Sigma}^*})^2.$$
 \end{proof}

 \par 
 We now want to define the localization $\mathcal{S}_A(\mathbf{\Sigma})^{loc}=\mathcal{S}_A(\mathbf{\Sigma})[(\alpha(p_{\partial_j})_{-+})^{-1}]$  obtained from $\mathcal{S}_A(\mathbf{\Sigma})$ by localizing by all bad arcs $\alpha(p_{{\partial}_j})_{-+}$. Recall that if $R$ is a ring and $S\subset R$ is a multiplicative subset not containing $0$, then the localization $R[S^{-1}]$ has a well-defined ring structure  if $S$ satisfies the following \textit{Ore condition}: 
  $$ \forall b,s \in R\times S, \quad \exists b_1,s_1 \in R\times S, \mbox{ such that } bs_1=sb_1.$$
  In this case, we can endow $R[S^{-1}]$ with the product $(as^{-1})\cdot (bt^{-1}):= (ab_1)(ts_1)^{-1}$. 
  \begin{lemma}(\cite[Lemma $4.4(a)$]{LeYu_SSkeinQTraces}) For every bad arc $\alpha_{bad}$ and every basis element $[D,s] \in \mathcal{B}$ there exists $n\in \mathbb{Z}$ such that $\alpha_{bad}[D,s]= A^n [D,s]\alpha_{bad}$. In particular
   the multiplicative subset $S\subset \mathcal{S}_A(\mathbf{\Sigma})$ generated by bad arcs satisfies the Ore condition.
  \end{lemma}
  
  Therefore $\mathcal{S}_A(\mathbf{\Sigma})^{loc}=\mathcal{S}_A(\mathbf{\Sigma})[(\alpha(p_{\partial_j})_{-+})^{-1}]$ is well defined.
  
 \begin{lemma}\label{lemma_Skein2} 
  $j(\alpha(p_{\partial_j})_{-+})$ is invertible in $\overline{\mathcal{S}}_A(\mathbf{\Sigma}^*)$ so the morphism $j$ induces a morphism
 $$ j^{loc}: \mathcal{S}_A(\mathbf{\Sigma})^{loc} \hookrightarrow \overline{\mathcal{S}}_A(\mathbf{\Sigma}^*).$$
Moreover
 $\overline{\mathcal{S}}_A(\mathbf{\Sigma}^*)$ is generated, as an algebra, by the image of $j^{loc}$ together with its center $Z_{\mathbf{\Sigma}^*}$. In particular $j\left( \underline{Z}_{\mathbf{\Sigma}}\right) \subset Z_{\mathbf{\Sigma}^*}$.
 \end{lemma}
 
 \begin{proof}
 Let us first assume that $\mathbf{\Sigma}= \mathbf{m}_1$. Then $\mathbf{m}_1^* = \mathbb{D}_1$, so $j_{\mathbf{m}_1}$ embeds $B_q\SL_2$ into $\Dq$. This map is a skein analogue of a celebrated Majid's morphism in quantum groups theory (see \cite{Majid_QGroups}). Let $\alpha_{ij}, \beta_{kl} \in \Dq$ be the stated arcs defined in Section \ref{sec_D1} (so $\alpha_{-+}=\beta_{+-}=0$). A simple computation, shown in Figure \ref{fig_MajidMorphism}, shows that 
 \begin{align*}
 {}&  j_{\mathbf{m}_1} (\gamma_{++})= A^{-1/2} \beta_{++}\alpha_{+-}, &  j_{\mathbf{m}_1} (\gamma_{--})= A^{-1/2} \beta_{-+}\alpha_{--}, \\
 {} &  j_{\mathbf{m}_1} (\gamma_{-+})= A^{-1/2} \beta_{++}\alpha_{--}, &  j_{\mathbf{m}_1} (\gamma_{+-})= A^{-1/2} \beta_{-+}\alpha_{+-}- A^{-5/2}\beta_{--}\alpha_{++}.
  \end{align*}
    
 \begin{figure}[!h] 
\centerline{\includegraphics[width=9cm]{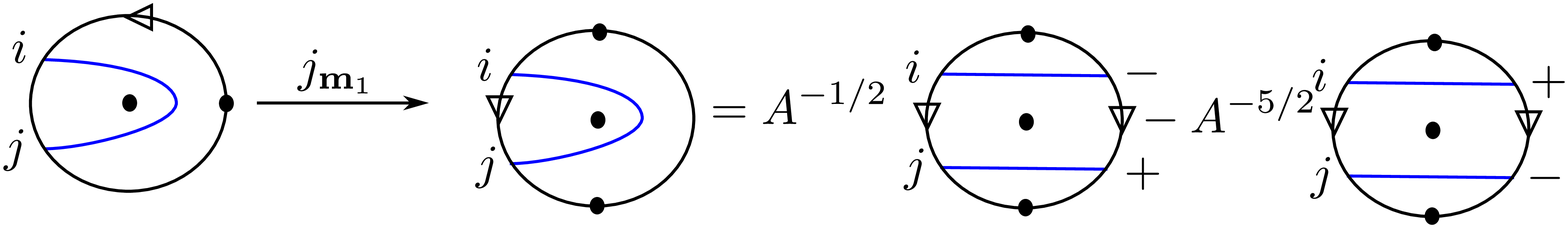} }
\caption{An illustration of the equality $j_{\mathbf{m}_1}(\gamma_{ij})= A^{-1/2} \beta_{j+}\alpha_{i-} - A^{-5/2} \beta_{j-}\alpha_{i+}$.
} 
\label{fig_MajidMorphism} 
\end{figure} 

  In particular for the bad arc $\gamma_{-+}$ we see that  $ j_{\mathbf{m}_1} (\gamma_{-+})$ is invertible with inverse $ j_{\mathbf{m}_1} (\gamma_{-+})^{-1}= A^{1/2} \alpha_{++}\beta_{--}$. From these formulas, together with the fact that $\alpha_{++}\beta_{++}, \alpha_{--}\beta_{--} \in Z_{\mathbb{D}_1}$ and that the peripheral central element $\gamma_p$ associated to the unique inner puncture $p$ of $\mathbb{D}_1$ satisfies
  $$ \gamma_p= -q^{-1}\beta_{--}\alpha_{++} -q \beta_{++}\alpha_{--} + \beta_{-+}\alpha_{+-} \in Z_{\mathbb{D}_1},$$
   we deduce that the algebra $A_{\mathbf{m}_1}\subset \Dq$ generated by the image of $j_{\mathbf{m}_1}^{loc}$ together with $Z_{\mathbb{D}_1}$, contains every elements of the form $\beta_{kl}\alpha_{ij}$ except possibly $\beta_{--}\alpha_{+-}$ and $\beta_{-+}\alpha_{++}$.  To prove that these two elements also belong to $A_{\mathbf{m}_1}$, note that $\beta_{--}^2=(\beta_{--}\alpha_{++})(\beta_{--}\alpha_{--}) \in A_{\mathbf{m}_1}$ so $\beta_{--}\alpha_{+-}=  \beta_{--}^2(\beta_{++}\alpha_{+-}) \in A_{\mathbf{m}_1}$. Similarly, $\beta_{-+}\alpha_{++}= (\beta_{-+}\alpha_{--})\alpha_{++}^2= (\beta_{-+}\alpha_{--})  (\beta_{--}\alpha_{++})(\beta_{++}\alpha_{++})\in A_{\mathbf{m}_1}$. 
  Therefore for every $i,j,k,l \in \{ \pm \}$, $\beta_{kl}\alpha_{ij}\in A_{\mathbf{m}_1}$. So $\alpha_{ij}^2 = (\beta_{++}\alpha_{ij} )( \beta_{--}\alpha_{ij}) \in A_{\mathbf{m}_1}$ and similarly $\beta_{kl}^2 \in A_{\mathbf{m}_1}$.
  Write $N=2n+1$ and recall that $\alpha_{ij}^N, \beta_{kl}^N \in A_{\mathbf{m}_1}$. So $\alpha_{++}= \left( \alpha_{++}^2 \right)^{n+1} \alpha_{--}^N \in A_{\mathbf{m}_1}$ and similarly $\alpha_{--}, \beta_{++}, \beta_{--} \in A_{\mathbf{m}_1}$. 
  This implies that $\alpha_{+-}= \beta_{++}(\beta_{--}\alpha_{+-}) \in A_{\mathbf{m}_1}$ and similarly $\beta_{-+}= (\beta_{-+}\alpha_{++})\alpha_{--}\in A_{\mathbf{m}_1}$. So all elements $\alpha_{ij}, \beta_{kl}$ belong to $A_{\mathbf{m}_1}$. Since these elements generate $\Dq$, we have proved that $A_{\mathbf{m}_1}=\Dq$.
  
  \vspace{2mm}
  \par When $\mathbf{\Sigma}=(\Sigma, \mathcal{A})$ is a general surface satisfying the hypotheses of Theorem \ref{theorem_Skein}, for each boundary puncture $p_{\partial_j}$ contained in a boundary component $\partial_j$, by embedding $\mathbf{m}_1$ into a neighborhood of $\partial_j$, we obtain an algebra morphism $\mu_q^j: B_q\SL_2 \to \mathcal{S}_A(\mathbf{\Sigma})$ sending $\alpha(p)_{kl}$ to $\alpha(p_{\partial_j})_{kl}$. The morphism 
  $$\mu_q:= \otimes_{j=1}^{|\mathcal{A}|} \mu_q^j : (B_q\SL_2)^{\otimes \mathcal{A}} \to \mathcal{S}_A(\mathbf{\Sigma})$$ is called the \textbf{quantum moment map} in literature (see \cite{KojuMCGRepQT, KojuSurvey} and reference therein). Similarly, by embedding a copy of  $\mathbb{D}_1$ in the neighborhood of $\partial_j$ seen as a boundary component of $\mathbf{\Sigma}^*$ this time (recall that $\Sigma=\Sigma^*$ but that each $\partial_j$ contains now two boundary edges), we obtain an algebra morphism $\mu_q^* : (\Dq)^{\otimes |\mathcal{A}|} \to \overline{\mathcal{S}}_A(\mathbf{\Sigma}^*)$ such that the following diagram commutes: 
  $$ \begin{tikzcd}
  (B_q\SL_2)^{\otimes |\mathcal{A}|} 
  \ar[rr, hook, "(j_{\mathbf{m}_1})^{\otimes |\mathcal{A}|}"] 
  \ar[d, "\mu_q"] &{}&
   (\Dq)^{\otimes |\mathcal{A}|} 
   \ar[d, "\mu_q^*"] \\
  \mathcal{S}_A(\mathbf{\Sigma}) \ar[rr, hook, "j_{\mathbf{\Sigma}}"]
   &{}& \overline{\mathcal{S}}_A(\mathbf{\Sigma}^*)
  \end{tikzcd}$$
   \begin{figure}[!h] 
\centerline{\includegraphics[width=4cm]{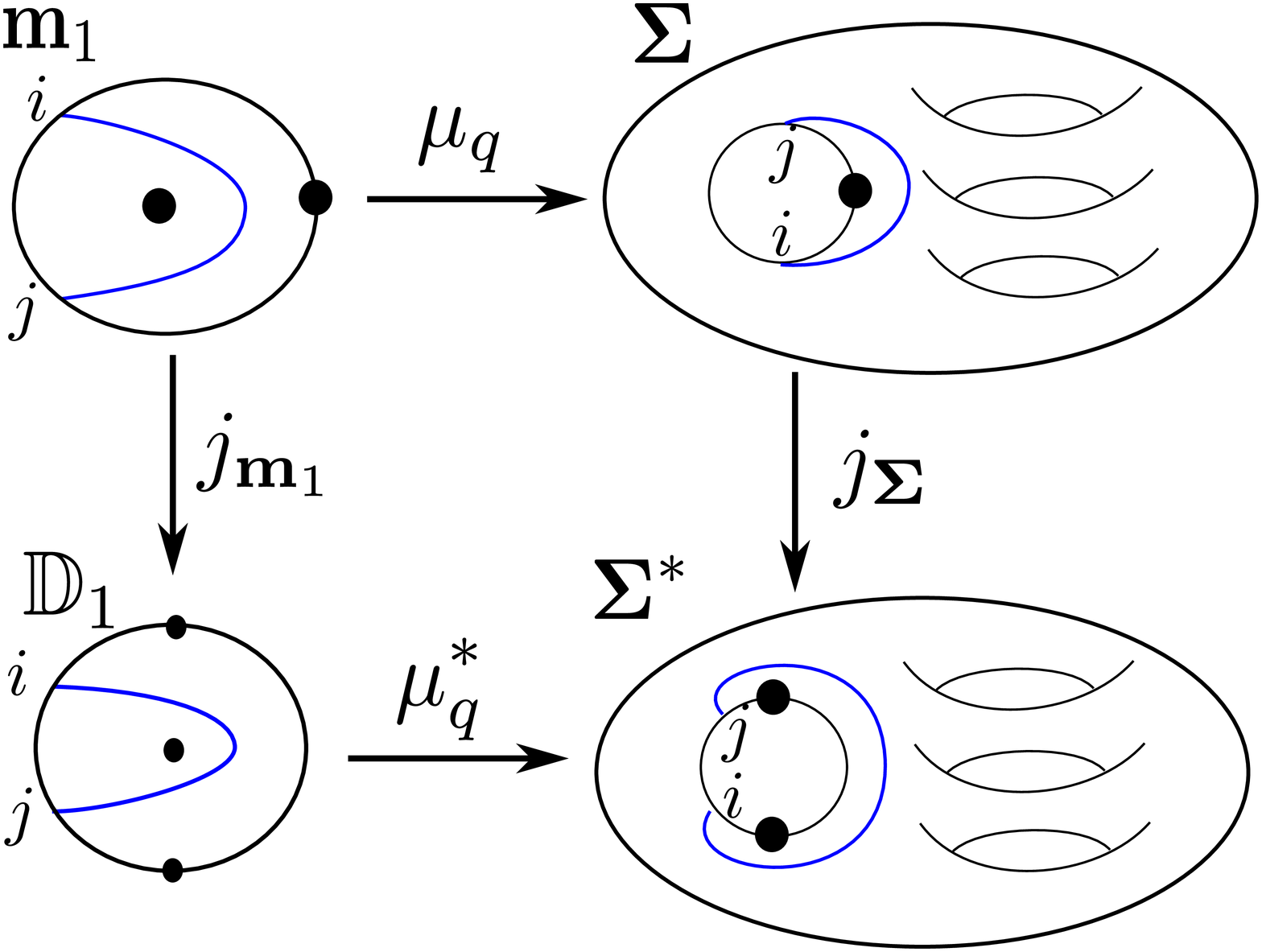} }
\caption{An illustration of the quantum moment maps $\mu_q$ and $\mu_q^*$.
} 
\label{fig_qMomentMaps} 
\end{figure} 
  So the fact that $j_{\mathbf{m}_1}(\alpha(p)_{-+})$ is invertible implies that $j_{\mathbf{\Sigma}}(\alpha(p_{\partial_j})_{-+}) $ is invertible as well so $j^{loc}_{\mathbf{\Sigma}}$ is well defined. Let $A_{\mathbf{\Sigma}} \subset \overline{\mathcal{S}}_A(\mathbf{\Sigma}^*)$ be the subalgebra generated by the image of $j_{\mathbf{\Sigma}}^{loc}$ and the center $Z_{\mathbf{\Sigma}^*}$. By the preceding case and the commutativity of the above diagram, we see the image $\Image(\mu_q^*)$ of $\mu_q^*$ is included in $A_{\mathbf{\Sigma}}$. Let $a_1, \ldots, a_{|\mathcal{A}|}$ be the boundary edges of $\mathbf{\Sigma}$. By construction, each $a_i$ corresponds to two boundary edges $a_i'$ and $a_i''$ of $\mathbf{\Sigma}^*$ and the image of $j_{\mathbf{\Sigma}}$ is spanned by the classes $[D,s]$ of stated diagrams such that $D\cap a_i'' = \emptyset$ for all $1\leq i \leq |\mathcal{A}|$. Let $[D,s] \in \overline{\mathcal{S}}_A(\mathbf{\Sigma}^*)$ be the class of an arbitrary stated diagram. As illustrated in Figure \ref{fig_slide}, by performing skein relations in small neighborhoods of each boundary component $\partial_j$, we see that $[D,s]$ is equal to a sum $[D,s]= \sum_i x_i [D_i,s_i]$ where $x_i \in \Image(\mu_q^*)$ and $D_i \cap a_j''=\emptyset$ for all $j$. This proves that $[D,s]\in A_{\mathbf{\Sigma}}$ so $A_{\mathbf{\Sigma}}=\overline{\mathcal{S}}_A(\mathbf{\Sigma}^*)$ and the proof is complete.
     \begin{figure}[!h] 
\centerline{\includegraphics[width=6cm]{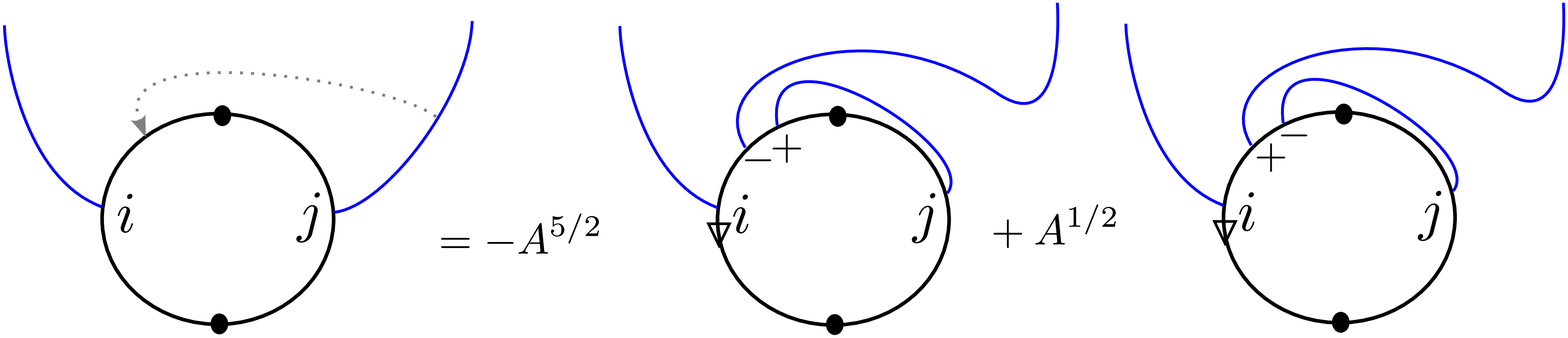} }
\caption{Using a skein relation, we can write any class $[D,s]$ as a linear combination of products of elements in the image of $\mu_q^*$ (here corner arcs $\alpha(p)_{j+}$ and $\alpha(p)_{j-}$) and elements $[D_i, s_i]$ where the $D_i$ do not intersect the boundary arcs $a''$ (here the boundary arc on the right).
} 
\label{fig_slide} 
\end{figure} 
 \end{proof}

 \begin{lemma}\label{lemma_Skein4} We have $\underline{Z}_{\mathbf{\Sigma}}=\underline{Z}'_{\mathbf{\Sigma}}$.
 \end{lemma}
 
 \begin{proof}
 By Lemma \ref{lemma_Skein2}, we have $j(\underline{Z}_{\mathbf{\Sigma}}) = Z_{\mathbf{\Sigma}^*} \cap \Image(j)$ so we need to prove the inclusion $Z_{\mathbf{\Sigma}^*} \cap \Image(j) \subset j(\underline{Z}'_{\mathbf{\Sigma}})$ to conclude. Let $\mathcal{B}_{\mathbf{\Sigma}}$ and $\mathcal{B}_{\mathbf{\Sigma}^*}$ be the L\^e bases of $\mathcal{S}_A(\mathbf{\Sigma})$ and $\overline{\mathcal{S}}_A(\mathbf{\Sigma}^*)$ respectively. So they are made of the classes $[D,s]$ where $D$ is simple, $s$ is $\mathfrak{o}^+$ increasing with the additional assumption that $[D,s]\in \mathcal{B}_{\mathbf{\Sigma}^*}$ does not contain any bad arc. Since $j$ sends $\mathcal{B}_{\mathbf{\Sigma}}$ to the subset $j(\mathcal{B}_{\mathbf{\Sigma}})\subset \mathcal{B}_{\mathbf{\Sigma}^*}$ of $[D,s]$ which do not intersect the boundary edges $a''$ for $a\in \mathcal{A}$, we see that $\Image(j)$ is spanned by $j(\mathcal{B}_{\mathbf{\Sigma}})$. 
  For  $(D,s)\in \mathcal{B}_{\mathbf{\Sigma}^*}$, denote by $T(D,s)$ the class of the diagram obtained from $(D,s)$ as follows. Fix an arbitrary total order $\prec$ on the set of connected diagrams in $\mathbf{\Sigma}^*$. Write $(D,s)=\sqcup_{i=1}^m \beta_i^{(n_i)}$ for the connected components of $(D,s)$ such that the $\beta_i$ are pairwise distinct ordered by $i<j \Rightarrow \beta_i \prec \beta_j$ and $\beta_i^{(n_i)}$ means that we have $n_i$ parallel copies of $\beta_i$. If $\beta_i$ is a loop, write $T(\beta_i^{(n_i)}):= T_{n_i}(\beta_i) \in \overline{\mathcal{S}}_A(\mathbf{\Sigma}^*)$. If $\beta_i$ is a stated arc, set $T(\beta_i^{(n_i)})$ to be the tangle made of $n_i$ parallel copies of $\beta_i$ pushed in the framing direction. Eventually set $T(D,s):=T(\beta_1^{(n_1)})\ldots T(\beta_m^{(n_m)}) $ (ordered product). 
  By \cite[Proposition $2.14$]{KojuQuesneyClassicalShadows} the set  $T\mathcal{B}_{\mathbf{\Sigma}^*}$ made of elements $T(D,s)$ for $[D,s]\in \mathcal{B}_{\mathbf{\Sigma}^*}$ is a basis of $\overline{\mathcal{S}}_A(\mathbf{\Sigma}^*)$. On the one hand,  $Z_{\mathbf{\Sigma}^*}$ is spanned by the subset $S_1\subset T\mathcal{B}_{\mathbf{\Sigma}^*}$ made of elements $T(D,s)$ where in the decomposition $(D,s)=\sqcup_{i=1}^m \beta_i^{(n_i)}$ for every $1\leq i \leq n$ then either $\beta_i$ is a peripheral curve $\gamma_p$ around an inner puncture, a boundary element $\alpha_{\partial}^{\pm 1}$  or $N$ divides $n_i$. On the other hand, $\Image(j)$ is spanned by the subset $S_2 \subset T\mathcal{B}_{\mathbf{\Sigma}^*}$ made of elements $T(D,s)$ such that $D$ does not intersect any boundary edge $a''$ for $a\in \mathcal{A}$. Therefore $j(\underline{Z}_{\mathbf{\Sigma}}) = Z_{\mathbf{\Sigma}^*} \cap \Image(j)$ is spanned by $S_1\cap S_2$: these elements are all in $j(\underline{Z}_{\mathbf{\Sigma}^*}')$. So $j(\underline{Z}_{\mathbf{\Sigma}^*})=j(\underline{Z}_{\mathbf{\Sigma}})$ and we conclude using the injectivity of $j$.  
 
 \end{proof}

  Let $\underline{\widehat{X}}(\mathbf{\Sigma})=\Specm (\underline{Z}_{\mathbf{\Sigma}})$ and $j^*: \widehat{X}(\mathbf{\Sigma}^*) \to \underline{\widehat{X}}(\mathbf{\Sigma})$ the dominant map defined by the embedding $j: \underline{Z}_{\mathbf{\Sigma}} \hookrightarrow  Z_{\mathbf{\Sigma}^*}$. 
  
  \begin{lemma}\label{lemma_Skein3} The image of $j^*$ is the set of elements $\widehat{\rho}=(\rho, h_{p_1}, \ldots, h_{p_k}) \in \underline{\widehat{X}}(\mathbf{\Sigma})$ such that $\mu(\rho)\in \SL_2^0 \times \ldots \times \SL_2^0$. 
  \end{lemma}
  
  \begin{proof} For each boundary edge $a \in \mathcal{A}$ of $\mathbf{\Sigma}$ fix a point $v_a \in a$ and let $\mathbb{V}:=\{v_a\}_{a\in \mathcal{A}}$. By definition $\mathbf{\Sigma}^*$ is the marked surface with $\Sigma$ as underlying surface and for which each boundary edge $a$ has been  replaced by two boundary edges $a'$ and $a''$.
  For each $a\in \mathcal{A}$ fix $v_{a'}:= v_a \in a'$ and $v_{a''}\in a''$ and $\mathbb{V}^*=\{v_{a'}, v_{a''}\}_{a\in \mathcal{A}}$. Using Theorem \ref{theorem_classical_limit}, we identify $\underline{\widehat{X}}(\mathbf{\Sigma})$ with the set of elements $\widehat{\rho}=(\rho, h_{p_i})_i$ where $\rho: \Pi_1(\Sigma, \mathbb{V}) \to \SL_2$ and $h_{p_i}\in \mathbb{C}$ is such that $\tr(\rho(\gamma_{p_i}))=-T_N(h_{p_i})$. Recall that $\widehat{X}(\mathbf{\Sigma}^*)$ is identified as the set of elements $\widehat{\rho}^*=(\rho, h_{p_i}, h_{\partial_j})_{i,j}$ where $\rho: \Pi_1(\Sigma, \mathbb{V}^*) \to \SL_2$ sends every corner path $\alpha(p)$ to an element of $\SL_2^1$. Since $\mathbb{V}\subset \mathbb{V}^*$, we have a full embedding $ \Pi_1(\Sigma, \mathbb{V}) \to \Pi_1(\Sigma^*, \mathbb{V}^*)$ and thus a restriction morphism $\Hom(\Pi_1(\Sigma^*, \mathbb{V}^*), \SL_2) \to \Hom(\Pi_1(\Sigma, \mathbb{V}), \SL_2)$ sending $\rho^*$ to  $\restriction{ \rho^*}{\Pi_1(\Sigma, \mathbb{V})}$. By definition, the map $j^*$ satisfies 
  $$ j^*( (\rho^*, h_{p_i}, h_{\partial_j})) = (\restriction{ \rho^*}{\Pi_1(\Sigma, \mathbb{V})}, h_{p_i}).$$
  So given a functor $\rho: \Pi_1(\Sigma, \mathbb{V})\to \SL_2$ such that $\rho(\alpha(p))\in \SL_2^0$ for all boundary puncture $p\in \mathcal{P}^{\partial}$ of $\mathbf{\Sigma}$, we need to prove that there exists a functor $\rho^*: \Pi_1(\Sigma^*, \mathbb{V}^*)\to \SL_2$ such that $\restriction{ \phi_*}{\Pi_1(\Sigma, \mathbb{V})}= \rho$ and such that $\rho^*(\alpha(p'))\in \SL_2^1$ for all boundary puncture $p'$ of $\mathbf{\Sigma}^*$. For a boundary component $\partial$ of $\Sigma$ containing a boundary puncture $p_{\partial}\in P^{\partial}$, the same boundary component considered in $\Sigma^*$ contains two boundary punctures $p_{\partial}'$ and $p_{\partial}''$ such that $\alpha(p_{\partial})= \alpha(p''_{\partial})\alpha(p'_{\partial})$. So lifting $\rho$ to a functor $\rho^*$ amounts to choose for each $p_{\partial}\in P^{\partial}$ two elements $\rho^*(\alpha(p'_{\partial})), \rho^*(\alpha(p''_{\partial})) \in \SL_2^1$ such that $\rho(\alpha(p_{\partial}))= \rho^*(\alpha(p'_{\partial}))\rho^*(\alpha(p''_{\partial}))$. The existence of such lift thus follows from the surjectivity of the map
  $$ \pi : \SL_2^1 \times \SL_2^1 \to \SL_2^0, \quad \pi(A,B):=AB.$$
  Indeed, denote by $B^+, B^- \subset \SL_2$ the subalgebras of upper and lower triangular matrices and write $w:= \begin{pmatrix} 0 & 1 \\ -1 & 0 \end{pmatrix}$ so that $\SL_2^0= B^-B^+$ and $\SL_2^1=B^- w B^+$. Then given $M\in \SL_2^0$ such that $M=M_- M_+$ with $M_{\pm} \in B^{\pm}$ set $A:= M_-w \in \SL_2^1$ and $B:= - wM_+\in \SL_2^1$ then $\pi(AB)=M$. Thus $\pi$ is surjective. This concludes the proof.

  \end{proof}

 \begin{proof}[Proof of Theorem \ref{theorem_Skein}]
 Write $D:= D_{\mathbf{\Sigma}^*}$.
Let $\widehat{\rho}=(\rho, h_{p_1}, \ldots, h_{p_k}) \in \underline{\widehat{X}}(\mathbf{\Sigma})$ such that $\mu(\rho)\in \SL_2^0 \times \ldots \times \SL_2^0$ and $\widehat{\rho}^* \in \widehat{X}(\mathbf{\Sigma}^*)$ such that $j^*(\widehat{\rho}^*)=\widehat{\rho}$. 
By Lemma \ref{lemma_Skein2}, $j$ induces a surjective map 
$$ j_{\widehat{\rho}}: \mathcal{S}_A(\mathbf{\Sigma})_{\widehat{\rho}} \to \overline{\mathcal{S}}_A(\mathbf{\Sigma}^*)_{\widehat{\rho}^*}.$$
Let $r: \overline{\mathcal{S}}_A(\mathbf{\Sigma}^*) \to \End(V)$ be an irreducible representation with classical shadow $\widehat{\rho}^*$ and $r':= r \circ j: \mathcal{S}_A(\mathbf{\Sigma})\to \End(V)$. Then $r'$ has classical shadow $\widehat{\rho}$ and is irreducible by surjectivity of $ j_{\widehat{\rho}}$ and the Sch\"ur lemma. First suppose that for all $i$, either  $\tr(\rho(\gamma_{p_i}))\neq \pm 2$ or $\tr(\rho(\gamma_{p_i}))= \pm  2$ and $h_{p_i}=\mp 2$. By Theorem \ref{main_theorem}, $\widehat{\rho}^*$ belongs to the Azumaya locus of $\overline{\mathcal{S}}_A(\mathbf{\Sigma}^*)$ so $\dim(V)=D$. Since the locus $\mathcal{O}\subset \widehat{\underline{X}}(\mathbf{\Sigma})$ of such $\widehat{\rho}$ is open dense, some of these irreducible representations $r'$ must have their classical shadows in the Azumaya locus of $\mathcal{S}_A(\mathbf{\Sigma})$. Therefore the PI-degree of  $\mathcal{S}_A(\mathbf{\Sigma})$ is $D$. This in turn implies that  $\mathcal{O}$ is included into the Azumaya locus of $\mathcal{S}_A(\mathbf{\Sigma})$. 
\par Next suppose that there exists $i$ such that $\tr(\rho(\gamma_{p_i}))= \pm  2$ and $h_{p_i}\neq \mp 2$. Then $\widehat{\rho}^*$ does not belong to the Azumaya locus of $\overline{\mathcal{S}}_A(\mathbf{\Sigma}^*)$ so $\dim(V)< D_{\mathbf{\Sigma}^*}$. This implies that $\widehat{\rho}$ does not belong to the Azumaya locus of  $\mathcal{S}_A(\mathbf{\Sigma})$. 
\vspace{2mm}
\par Eventually consider $\widehat{\rho} \in \underline{\widehat{X}}(\mathbf{\Sigma})$ such that $\mu(\rho)\in \SL_2^{1} \times \ldots \times \SL_2^1$. We can consider an irreducible representation $r: \overline{\mathcal{S}}_A(\mathbf{\Sigma}) \to \End(V)$ such that the composition $r': \mathcal{S}_A(\mathbf{\Sigma})\to \overline{\mathcal{S}}_A(\mathbf{\Sigma}) \xrightarrow{r} \End(V)$ has classical shadow $\widehat{\rho}$. Since $D_{\mathbf{\Sigma}}<D_{\mathbf{\Sigma}^*}$ then $\dim(V)< D_{\mathbf{\Sigma}^*}$ so $\widehat{\rho}$ does not belong to the Azumaya locus of $\mathcal{S}_A(\mathbf{\Sigma})$. This concludes the proof. 
 
 \end{proof}
 
 \subsection{Quantum moduli algebras and lattice gauge field theory}\label{sec_QMA}

  In the particular case where $\mathbf{\Sigma}=\mathbf{\Sigma}(g,n):=(\Sigma_{g,n+1}, \{a\})$ is a genus $g\geq 0$ surface with $n+1$ boundary components and exactly one of them has a single boundary edge $a$ whereas the others are inner punctures, the algebra $\mathcal{L}_{g,n}:= \mathcal{S}_A(\mathbf{\Sigma})$ is called the  \textbf{\QMA}. The algebra $\mathcal{L}_{g,n}$ appeared in literature in various contexts.
\begin{enumerate}
\item They were first defined independently   in \cite{AlekseevGrosseSchomerus_LatticeCS1,AlekseevGrosseSchomerus_LatticeCS2, AlekseevSchomerus_RepCS} and \cite{BuffenoirRoche, BuffenoirRoche2} where they appeared as deformation quantization of the  Fock-Rosly \cite{FockRosly} and Alekseev-Kosmann-Malkin-Meinrenken \cite{AlekseevMalkin_PoissonLie, AlekseevKosmannMeinrenken, AlekseevMalkin_PoissonCharVar} moduli spaces (see also \cite{KojuTriangularCharVar}). Under this approach, they were intensively studied in \cite{BaseilhacRoche_LGFT1, BaseilhacRoche_LGFT2, BaseilhacFaitgRoche_LGFT3}.
\item They were rediscovered independently by Habiro under the name \textit{quantum representation variety} in \cite{Habiro_QCharVar}.
\item As we just defined, they are particular cases of stated skein algebras.
\item They eventually appeared under the name \textit{internal skein algebras} in the work of Ben Zvi-Brochier-Jordan \cite{BenzviBrochierJordan_FactAlg1, BenzviBrochierJordan_FactAlg2} and further studied  in \cite{Cooke_FactorisationHomSkein, GunninghamJordanSafranov_FinitenessConjecture, GanevJordanSafranov_FrobeniusMorphism}.
\end{enumerate}
The equivalence between \QMAs and stated skein algebras is proved in \cite{BullockFrohmanKania_LGFT,Faitg_LGFT_SSkein, KojuPresentationSSkein}. The equivalence between internal skein algebras and \QMAs{ }  is proved in \cite{BenzviBrochierJordan_FactAlg1}. The equivalence between internal skein algebras and stated skein algebras is proved in \cite{Haioun_Sskein_FactAlg}. The equivalence between quantum representation varieties and stated skein algebras is proved in \cite{KojuMurakami_QCharVar}.

\vspace{2mm}
\par By definition $\mathbf{\Sigma}(g,n)^*=(\Sigma_{g,n+1}, \{a', a''\})$ is the surface $\Sigma_{g,n+1}$ with two boundary edges $\{a', a''\}$ in the same boundary component. Note that  $\mathbf{\Sigma}(g,n)^*$ is obtained from $g$ copies of $\mathbf{\Sigma}(1,0)^*$ and $n$ copies of $\mathbf{\Sigma}(0,1)^*$ by gluing some pairs of boundary edges. Therefore one has a splitting morphism 
$$\theta: \overline{\mathcal{S}}_A(\mathbf{\Sigma}(g,n)^*) \hookrightarrow \overline{\mathcal{S}}_A(\mathbf{\Sigma}(1,0)^*)^{\otimes g} \otimes \overline{\mathcal{S}}_A(\mathbf{\Sigma}(0, 1)^*)^{\otimes n}.$$
 Note that $\mathbf{\Sigma}(0, 1)^*=\mathbb{D}_1$ so $\overline{\mathcal{S}}_A(\mathbf{\Sigma}(0, 1)^*)\cong \Dq$. On the other hand, $\mathcal{HH}_q:=\mathcal{S}_A(\mathbf{\Sigma}(1, 0)^*)$ is the so-called \textit{Heisenberg double} algebra. To see this, let $\alpha, \beta$ be the arcs of Figure \ref{fig_D1} and consider the matrices
 $$N(\alpha):= \begin{pmatrix} \alpha_{++} & \alpha_{+-} \\ \alpha_{-+} & \alpha_{--} \end{pmatrix}, N(\beta):=  \begin{pmatrix} \beta_{++} & \beta_{+-} \\ \beta_{-+} & \beta_{--} \end{pmatrix}, \mathscr{R}:=\begin{pmatrix} A & 0 & 0 & 0 \\ 0 & 0 &A^{-1} & 0 \\ 0 & A^{-1} & A-A^{-3} & 0 \\ 0 & 0 & 0 & A \end{pmatrix} .$$
  By \cite[Theorem $1.1$]{KojuPresentationSSkein},  $\mathcal{HH}_q$ is generated by the elements $\alpha_{ij}, \beta_{kl}$ with relations $\det_q(N(\alpha))=\det_q(N(\beta))=1$ (where $\det_q \begin{pmatrix} a & b \\ c& d\end{pmatrix}=ad-q^{-1}bc$) and 
 $$ N(\alpha)\odot N(\alpha) = \mathscr{R}^{-1} (N(\alpha)\odot N(\alpha))\mathscr{R}, N(\beta)\odot N(\beta) = \mathscr{R}^{-1} (N(\beta)\odot N(\beta))\mathscr{R}, N(\alpha)\odot N(\beta) = \mathscr{R} (N(\alpha)\odot N(\beta))\mathscr{R}.$$
 Here $\odot$ denotes the Kronecker product and we recognize here Alekseev's relations defining the Heisenberg double. Therefore $\overline{\mathcal{HH}}_q:= \overline{\mathcal{S}}_A(\mathbf{\Sigma}(1, 0)^*)$ is the quotient of $\mathcal{HH}_q$ by the ideal generated by the two bad arcs.
 
     \begin{figure}[!h] 
\centerline{\includegraphics[width=3cm]{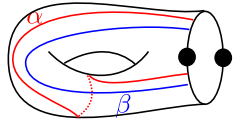} }
\caption{The marked surface $\mathbf{\Sigma}(1,0)^*=(\Sigma_{1,1}, \{a', a''\})$ defining the Heisenberg double $\mathcal{HH}_q$.} 
\label{fig_D1}
\end{figure} 

\par Consider the composition: 
$$ \Phi: \mathcal{L}_{g,n} \xrightarrow{j} \overline{\mathcal{S}}_A(\mathbf{\Sigma}(g,n)^*) \xrightarrow{\theta} (\overline{\mathcal{HH}}_q)^{\otimes g} \otimes (\Dq)^{\otimes n}.$$
An analogue of the map $\Phi$ appeared in Alekseev's work \cite{Alekseev_AlekseevMorphism} and was named \textit{Alekseev's morphism} in \cite{BaseilhacFaitgRoche_LGFT3}. Given $V_1, \ldots, V_n$ and $W_1, \ldots, W_g$ some semi-weight indecomposable representations of $U_q\mathfrak{gl}_2$ and $\overline{\mathcal{HH}}_q$ respectively, using the morphism $\Phi$, we endow $W_1\otimes \ldots \otimes W_g \otimes V_1 \otimes \ldots \otimes V_n$ with a structure of module over $\mathcal{L}_{g,n}$.
Theorem \ref{main_theorem} implies that every reduced $\mathcal{L}_{g,n}$ module is isomorphic to a direct sum of such  modules and proves that the representations $W_i$ are in $1:1$ correspondence with the character over the center of $\overline{\mathcal{HH}}_q$ and gives a complete classification for the modules $V_i$. Therefore we obtain an explicit construction of the reduced representations of the \QMAs. 
The map $\mu: \underline{X}(\mathbf{\Sigma})\to \SL_2$ sending $\rho$ to $\rho(\alpha_{\partial})$ was named the \textit{moment map} in  \cite{AlekseevKosmannMeinrenken}.
Theorem \ref{theorem_Skein} implies

\begin{corollary}\label{coro_QMA}
\begin{enumerate}
\item The center of $\mathcal{L}_{g,n}$ is generated by the image of the Frobenius and the peripheral curves $\gamma_{p_1}, \ldots, \gamma_{p_n}$.
\item The PI-degree of  $\mathcal{L}_{g,n}$ is $N^{3g+n}$.
\item The Azumaya locus of $\mathcal{L}_{g,n}$ is the locus of elements $\widehat{\rho}=(\rho, h_{p_i})$ such that $\mu(\rho)\in \SL_2^0$ and $\tr(\rho(\gamma_{p_i}))=\pm 2 \Rightarrow h_{p_i}=\mp 2$.
\end{enumerate}
\end{corollary}

When $n=0$, Corollary \ref{coro_QMA} was proved in \cite[Theorem $1$]{GanevJordanSafranov_FrobeniusMorphism}. When $g=0$, the first point of Corollary \ref{coro_QMA} was proved in \cite[Theorem $1.1$]{BaseilhacRoche_LGFT1}, while the second point was proved in \cite[Theorem $1.3$]{BaseilhacRoche_LGFT2}. 
\par Let us end by a discussion on the classification of representations of $\mathcal{L}_{g,n}$. Let $r: \mathcal{L}_{g,n}\to \End(V)$ be an indecomposable semi-weight representation and consider the only bad arc $\alpha(p)_{-+}\in \mathcal{L}_{g,n}$.
\begin{enumerate}
\item If $r(\alpha(p)_{-+})$ is invertible, then $r$ is isomorphic to a reduced representation, i.e. it factorizes through $\overline{\mathcal{S}}_A(\mathbf{\Sigma}^*)$. These representations are completely classified by Theorem \ref{main_theorem} and the skein Alekseev morphism $\Phi$ permits an explicit construction of them as explained before. Note that every irreducible representation whose classical shadow belongs to the Azumaya locus belongs to this class.
\item If $r(\alpha(p)_{-+})=0$, $r$ factorizes through the reduced stated skein algebra $\overline{\mathcal{S}}_A(\mathbf{\Sigma})$. When $n=0$, by \cite{KojuMCGRepQT} the latter algebra is Azumaya so these representations are classified. When $n\geq 1$, the techniques of the present paper are insufficient to classify such representations though we conjecture that Theorem \ref{main_theorem} still holds in this case. 
\item The last class of representations, for which $r(\alpha(p)_{-+})^N=0$ but $r(\alpha(p)_{-+})\neq 0$, seems very difficult to classify. Such a classification might even be an unsolvable problem. 
\end{enumerate}
 
 \appendix 
  
  \section{Glossary}\label{appendix_glossary}
  
  We summarize the notations and definitions used in the paper.
  \vspace{2mm}
  \par \underline{\textbf{Marked surfaces}}
    \vspace{2mm}
\par $\bullet$ A \textbf{marked surface} is a pair $\mathbf{\Sigma}=(\Sigma, \mathcal{A})$ where $\Sigma$ is a compact oriented surface and $\mathcal{A}\subset \partial \Sigma$ a finite collection of pairwise disjoint arcs named \textbf{boundary edges}. 
A connected component of $\partial \Sigma \setminus \mathcal{A}$ is a \textbf{puncture}; it is an \textbf{inner puncture} if it is a circle, else it is a \textbf{boundary puncture}. We decompose as $\mathcal{P}=\mathring{\mathcal{P}} \sqcup \mathcal{P}^{\partial}$ the set of punctures. We denote by $\Gamma^{\partial} \subset \pi_0(\partial \Sigma)$ the subset of boundary components which are not inner punctures. 
\par $\bullet$
$\mathbf{\Sigma}$ is \textbf{essential} if each connected component of $\Sigma$ contains some boundary edge. $\mathbf{\Sigma}$ is a \textbf{2P-marked surface} if it has no inner puncture and each connected component of $\Sigma$ contains at least two boundary edges. 
\par $\bullet$ The \textbf{bigon} $\mathbb{B}$ is a disc with two boundary edges. The \textbf{punctured bigon} $\mathbb{D}_1$ is an annulus with two boundary edges in the same boundary component. $\mathbf{\Sigma}(g,n)$ is the surface $\Sigma_{g,n+1}$ with one boundary edge.  The \textbf{punctured monogon} $\mathbf{m}_1=(\Sigma_{0,2}, \{b\})$ is an annulus with a single boundary edge.

  \vspace{2mm}
  \par \underline{\textbf{Algebras}}
    \vspace{2mm}
  \par $\bullet$ An \textbf{almost Azumaya algebra} $\mathscr{A}$ is a complex algebra which is prime, affine, of finitely generated as a module over its center $Z$. Its \textbf{PI-degree} $D$ is defined by $D^2=\dim_{Q(Z)}\mathscr{A}\otimes_ZQ(Z)$. Its \textbf{Azumaya locus} is $$\mathcal{AL}(\mathscr{A}):= \{ x \in \Specm(Z): \quotient{\mathscr{A}}{x \mathscr{A}} \cong \Mat_D(\mathbb{C})\}.$$
  \par $\bullet$ $A^{1/2}\in \mathbb{C}$ is a root of unity of odd order $N\geq 3$. $\mathcal{S}_A(\mathbf{\Sigma})$ is the \textbf{stated skein algebra} of $\mathbf{\Sigma}$. $\overline{\mathcal{S}}_A(\mathbf{\Sigma})$ is the \textbf{reduced stated skein algebra} of $\mathbf{\Sigma}$. We set $\mathcal{L}_{g,n}:= \mathcal{S}_A(\mathbf{\Sigma}(g,n))$. 
  \\ We consider $Z^0_{\mathbf{\Sigma}}\subset Z^1_{\mathbf{\Sigma}} \subset Z_{\mathbf{\Sigma}}\subset \overline{\mathcal{S}_A}(\mathbf{\Sigma})$ where $Z_{\mathbf{\Sigma}}$ is the center of $\overline{\mathcal{S}_A}(\mathbf{\Sigma})$, $Z^0_{\mathbf{\Sigma}}$ is the image of the Frobenius and $Z^1_{\mathbf{\Sigma}}$ is the central subalgebra generated by $Z^0_{\mathbf{\Sigma}}$ and the boundary central elements $\alpha_{\partial}^{\pm 1}$ (see Definition \ref{def_central_elements}). 
  \par $\bullet$  We write
  $$ X(\mathbf{\Sigma}):= \Specm(Z^0_{\mathbf{\Sigma}}), \quad \widehat{X}(\mathbf{\Sigma}):= \Specm(Z_{\mathbf{\Sigma}}) \quad \widehat{\mathcal{X}}(\mathbf{\Sigma}):=\{ \mathfrak{m}\in \Spec(Z_{\mathbf{\Sigma}}): \mathfrak{m}\cap Z_{\mathbf{\Sigma}}^0 \in X(\mathbf{\Sigma})\}$$
  and denote by $\pi: \widehat{X}(\mathbf{\Sigma})\to X(\mathbf{\Sigma})$ and $\widehat{\pi}: \widehat{\mathcal{X}}(\mathbf{\Sigma})\to X(\mathbf{\Sigma})$  the natural projections sending $\mathfrak{m}\subset Z_{\mathbf{\Sigma}}$ to $\mathfrak{m}\cap Z^0_{\mathbf{\Sigma}} \subset Z^0_{\mathbf{\Sigma}}$.
  \\ Set $\mathcal{AL}(\mathbf{\Sigma}):= \mathcal{AL}(\overline{\mathcal{S}_A}(\mathbf{\Sigma}))$. The \textbf{fully Azumaya locus} is 
   $$ \mathcal{FAL}(\mathbf{\Sigma}):= \{ x \in X(\mathbf{\Sigma}) : \pi^{-1}(x)\subset \mathcal{AL}(\mathbf{\Sigma}) \}.$$
   \par $\bullet$ Similarly we set $\underline{Z}^0_{\mathbf{\Sigma}} \subset \underline{Z}_{\mathbf{\Sigma}} \subset \mathcal{S}_A(\mathbf{\Sigma})$ the image of the Frobenius and the center respectively and set 
   $$\underline{X}(\mathbf{\Sigma}):= \Specm(\underline{Z}^0_{\mathbf{\Sigma}}), \quad \underline{\widehat{X}}(\mathbf{\Sigma}):= \Specm(\underline{Z}_{\mathbf{\Sigma}}).$$

     \vspace{2mm}
  \par \underline{\textbf{Representations}}
    \vspace{2mm}
  \par $\bullet$ A representation $r: \overline{\mathcal{S}}_A(\mathbf{\Sigma})\to \End(V)$ is a \textbf{weight representation} if it is semi-simple as module over $Z_{\mathbf{\Sigma}}$ and is a \textbf{semi-weight representation} if it is semi-simple as a module over $Z^0_{\mathbf{\Sigma}}$. A representation $r: \mathcal{S}_A(\mathbf{\Sigma})\to \End(V)$ is \textbf{reduced} if it factorizes as $r: \mathcal{S}_A(\mathbf{\Sigma}) \xrightarrow{j} \overline{\mathcal{S}}_A(\mathbf{\Sigma}^*) \to \End(V)$.
  \par $\bullet$ Let $r: \overline{\mathcal{S}}_A(\mathbf{\Sigma})\to \End(V)$ be an indecomposable semi-weight representation.
  \begin{enumerate}
  \item 
   Its \textbf{classical shadow} is the maximal ideal $x_r\in X(\mathbf{\Sigma})$ defined by $x_r=\ker(r)\cap Z^0_{\mathbf{\Sigma}}$. 
   \item Its \textbf{full shadow} is the prime ideal $\mathcal{I}_r \in \widehat{\mathcal{X}}(\mathbf{\Sigma})$ defined by $\mathcal{I}_r:= \ker(r) \cap Z_{\mathbf{\Sigma}}$. 
   \item Its \textbf{maximal shadow} is the unique maximal ideal $\widehat{x}_r\in \widehat{X}(\mathbf{\Sigma})$ which contains $\mathcal{I}_r$.
   \end{enumerate}
   In particular ${\pi}(\widehat{x}_r)=\widehat{\pi}(\mathcal{I}_r)=x_r$ and $\mathcal{I}_r=\widehat{x}_r\in \widehat{X}(\mathbf{\Sigma})$ if and only if $r$ is weight representation.
    \vspace{2mm}
  \par \underline{\textbf{Moduli spaces}}
    \vspace{2mm}
  \par $\bullet$ When $\mathbf{\Sigma}$ is essential, the \textbf{representation variety} is 
  $$ \mathcal{R}_{\SL_2}(\mathbf{\Sigma}):= \Hom(\pi_1(\Sigma, \mathbb{V}), \SL_2).$$
 The \textbf{reduced representation variety} is
 $$ \overline{\mathcal{R}}_{\SL_2}(\mathbf{\Sigma}): \{ \rho \in \mathcal{R}_{\SL_2}(\mathbf{\Sigma}) : \rho(\alpha_{\partial}) \in \SL_2^1, \forall \partial \in \Gamma^{\partial} \}.$$
 Here we consider the simple Bruhat decomposition $\SL_2= \SL_2^0 \sqcup \SL_2^1$ where 
 $$ \SL_2^0 := \left\{ \begin{pmatrix} a & b \\ c & d \end{pmatrix} \in \SL_2 : a\neq 0\right\}.$$

\bibliographystyle{amsalpha}
\bibliography{biblio}

\end{document}